\documentclass[12pt]{amsart}
\usepackage{epsfig}
\usepackage[dvipsnames]{xcolor}

\usepackage{hyperref}
\hypersetup{
    colorlinks=true, 
    linktoc=all,     
    linkcolor=blue,  
}

\usepackage{amsrefs}

\usepackage{mathrsfs}
\usepackage{amssymb}

\theoremstyle{definition}

\newtheorem{theorem}{Theorem}[section]
\newtheorem{definition}[theorem]{Definition}

\newtheorem{proposition}[theorem]{Proposition}
\newtheorem{lemma}[theorem]{Lemma}
\newtheorem{remark}[theorem]{Remark}
\newtheorem{corollary}[theorem]{Corollary}

\newtheorem{question}[theorem]{Question}

\numberwithin{equation}{section}

\usepackage{enumerate}

\newcommand\ind{\operatorname{ind}}
\DeclareMathOperator*{\nul}{null}
\DeclareMathOperator*{\spa}{span}

\renewcommand\div{\operatorname{div}}

\newcommand{\tr}{\operatorname{tr}}
\DeclareMathOperator*{\id}{id}

\newcommand{\pr}{\partial}

\newcommand{\Lap}{\Delta}

\newcommand{\Ric}{\operatorname{Ric}}

\newcommand{\Hess}{\operatorname{Hess}}

\renewcommand*\d{\mathop{}\!\mathrm{d}}

\DeclareMathOperator{\ct}{ct}

\newcommand{\pit}{\frac{\pi}{2}}

\newcommand{\Conf}{\operatorname{Conf}}
\DeclareMathOperator*{\artanh}{artanh}

\usepackage{comment}

\usepackage{esint}
\usepackage{subcaption}

\def\MR#1{}

\title[Capillary minimal surfaces in spherical caps: Low energy]{Free boundary and capillary minimal surfaces in spherical caps II: Low energy}

\author{Jonathan J. Zhu}
\address{University of Washington, Seattle, WA, USA}
\email{jonozhu@uw.edu}

\begin{document}
\begin{abstract}
This is the second of two articles in which we investigate the geometry of free boundary and capillary minimal surfaces in balls $B_R\subset\mathbb{S}^3$. In this article, we find monotonicity formulae which imply that capillary minimal surfaces maximise a certain modified energy in their conformal orbit (preserving $B_R$). In the hemisphere, this energy is precisely the capillary energy. We also prove a partial characterisation by index for capillary minimal surfaces in the hemisphere, analogous to Urbano's characterisation of the Clifford torus. 
\end{abstract}
\date{\today}
\maketitle

\vspace{-2em}
\section{Introduction}
\label{sec:intro}

This article is the second of two papers (together with \cite{NZ25a}) in which we investigate the geometry of free boundary and capillary minimal surfaces $\Sigma$ in geodesic balls $B^3_R\subset \mathbb{S}^3$. We endeavour to study these surfaces under a unified framework, with key parameters given by the cap radius $R$ and the contact angle $\gamma$. In this article, we focus on rigidity phenomena for surfaces of low energy or of low index. 
Capillary minimal surfaces arise as (relative) boundaries of regions $\Omega \subset B_R$ that are critical for the \textit{wetting energy}
\begin{equation}\label{eq:cap-energy-intro}
 \mathcal{A}^\gamma(\Omega) = |\pr\Omega \cap B_R| + \cos\gamma\, |\pr\Omega \cap \pr B_R| ,
 \end{equation}
where $\gamma \in (0,\pit]$ is fixed. Classical solutions arising this way yield a surface $\Sigma = \pr \Omega \cap B_R$, and a `wet' portion $S^- = \pr\Omega \cap \pr B^n_R$ of the barrier that satisfies the contact angle condition \[\sin\gamma \, \bar{\eta} = \eta + \cos\gamma \, \bar{\nu}\]
along $\pr\Sigma=\pr S^-$, where $\bar{\eta}$ is the outward normal of $B_R$ in $\mathbb{S}^n$, $\eta$ is the outer conormal of $\pr\Sigma$ in $\Sigma$ and $\bar{\nu}$ is the outer conormal of $\pr\Sigma$ in $S^-$. The capillary setting has seen a sharp rise in recent interest - see for instance \cite{CX19, LZZ, So23, HS23, CEL24, dMEL}; and the \textit{free boundary} setting $\gamma=\pit$ is an important special case, with a wealth of recent progress - we highlight \cite{FS13, Li19, LM23, Me23, KKMS, CFM25}. 

\textit{Dual} to the free boundary setting (see \cite{NZ25a}), we consider capillary minimal surfaces in the hemisphere ($R=\pit$), which are also significant as such surfaces model tangent cones (hence singularities) of other capillary minimal surfaces. In this distinguished setting, we pose:

\begin{question}
\label{q:energy}
What are the capillary minimal surfaces in $\mathbb{S}^3_+$ of \textit{least energy} $\mathcal{A}^\gamma$? 
\end{question}

By monotonicity formulae for free boundary minimal varifolds (and also from our Theorem \ref{thm:conf-max-sph-intro-cap} below), the least energy capillary minimal surface in $\mathbb{S}^3_+$ is a totally geodesic half-equator. The question becomes: What is the capillary surface of \textit{second} least energy?

In the model setting of \textit{closed} minimal surfaces in $\mathbb{S}^3$, the least area surface is the totally geodesic equator. By the work of Marques-Neves \cite{MN14} on the Willmore conjecture, the second least area surface is the Clifford torus. In \cite{MN14}, two results related to the index are crucial: their construction of 5-parameter families that decrease area (using the $\Conf(\mathbb{S}^3)$-action); and Urbano's classification of (2-sided) minimal surfaces with Morse index at most 5 \cite{Ur90}. 

In this article, we show that the $\Conf(B_R)$-action decreases the following modified energy of a capillary minimal surface in $B_R$ with contact angle $\gamma$ and wet portion $S^-$:

\begin{equation}
\mathcal{E}^{R,\gamma}(\Sigma, S^-) = |\Sigma| +\cos\gamma \csc^2 R\, |S^-| + \sin\gamma \cot R\, |\pr\Sigma|.
\end{equation}

We have the following conformal maximisation result, stated somewhat more generally:

\begin{theorem}
\label{thm:conf-max-sph-intro-cap}
Let $(\Sigma^{n-1},\pr\Sigma)$ be a capillary minimal surface in $(B^n_R , \pr B^n_R)$ with contact angle $\gamma \in (0,\pit]$. Given a conformal diffeomorphism $\Psi \in \Conf(B_R)$, consider the image surface $\Sigma_* = \Psi_*\Sigma$, which also contacts $\pr B_R$ at angle $\gamma$. Similarly, let $S^-_* = \Psi_* S^-$.

Assume that either $n-1=2$, or $R=\pit$. Then $  \mathcal{E}^{R,\gamma}(\Sigma_*, S^-_*)\leq \mathcal{E}^{R,\gamma}(\Sigma, S^-)$. 
\end{theorem}

This result is new even for free boundary minimal surfaces ($\gamma=\pit$), in which case the wet portion $S^-$ does not contribute to the above functional. (The special case $R=\gamma=\pit$ follows from the results of Li-Yau \cite{LY82} for closed surfaces, and Fraser-Schoen \cite{FS11} established a version in the Euclidean limit $R\to0$.) Returning to the hemisphere case $R=\pit$, we point out that $\mathcal{E}^{\pit,\gamma}$ precisely coincides with the capillary energy $\mathcal{A}^\gamma$, which we highlight as: 

\begin{corollary}
\label{cor:conf-max-sph-cap-intro}
Let $(\Sigma,\pr\Sigma)$ be a capillary minimal surface in $\mathbb{S}^3_+$, which contacts $\pr \mathbb{S}^3_+$ at angle $\gamma \in (0,\pit]$, and bounds the region $\Omega \subset\mathbb{S}^3_+$. Given $\Psi \in \Conf(\mathbb{S}^3_+)$, let $\Omega_* = \Psi(\Omega)$. 

Then $\mathcal{A}^\gamma(\Omega_*)\leq \mathcal{A}^\gamma(\Omega) $. 
\end{corollary}

Similarly, a linear analysis of $\mathcal{A}^\gamma$ suggests that any capillary minimal surface in the hemisphere $\mathbb{S}^3_+$ should have $1+3=4$ unstable directions; again Theorem \ref{thm:conf-max-sph-intro-cap} implies that 3 of these directions may be realised by the $\Conf(\mathbb{S}^3_+)$-orbit. Towards an index characterisation analogous to Urbano's theorem, we will show:

\begin{theorem}
\label{thm:urbano-intro}
Let $(\Sigma,\pr\Sigma) \looparrowright (\mathbb{S}^3, \pr B_\pit)$, be a 2-sided capillary minimal surface with contact angle $\gamma$. If $\Sigma$ is a totally geodesic disc, then $\ind(\Sigma)=1$.

If $\Sigma$ is not totally geodesic, then $\ind(\Sigma)\geq 4$. 

Further, assume that $\cos\gamma \in [0, \frac{2}{5})$ and $\int_{\pr\Sigma}H^{\pr\Sigma}_\Sigma \leq 0$, where $H^{\pr\Sigma}_\Sigma$ is the geodesic curvature of $\pr\Sigma$ in $\Sigma$. Then $\ind(\Sigma)=4$ holds only if $\Sigma$ is a rotationally symmetric annulus.
\end{theorem}

(See Theorem \ref{thm:urbano} and Lemma \ref{lem:tot-geo-index} for precise statements.)

For $\gamma=\pit$, a free boundary minimal surface in $\mathbb{S}^3_+$ must have $H^{\pr\Sigma}_\Sigma\equiv 0$, so the boundary curvature condition above is automatically satisfied. For other $\gamma$, there are also certain natural conditions on $\Sigma$ that guarantee $H^{\pr\Sigma}_\Sigma\leq 0$ - for instance, if $\Sigma$ is graphical (cf. \cite[Section 4.4]{PTV}). We remark that some boundary curvature condition as above should be necessary to obtain a full uniqueness result. Indeed, when $\gamma <\pit$ we may not expect uniqueness of rotationally symmetric capillary minimal annuli. (As we discussed in \cite[Remark 1.3]{NZ25a}, numerical experiments suggest that there is a $\gamma_*$ so that for each contact angle $\gamma \in(\gamma_*,\pit)$, there may be multiple embedded rotationally symmetric capillary minimal annuli in $\mathbb{S}^3_+$; however, only one is a radial graph and hence satisfies $H^{\pr\Sigma}_\Sigma < 0$.) Finally, the angle restriction $\cos\gamma < \frac{2}{5}$ appears to be a technical condition, at least with current methods. 

In this and our companion paper \cite{NZ25a}, we (generally) take a parametric approach towards minimal surfaces. Accordingly, there are weaker notions of capillary surface for which our methods hold (see Section \ref{sec:contact-angle} for discussion).  

An interesting feature of our approach to conformal maximisation is that Theorem \ref{thm:conf-max-sph-intro-cap} will actually follow from a monotonicity formula approach, which surprisingly \textit{leaves} the group $\Conf(B^3_R)$ and instead gives monotonicity on paths in the larger conformal group $\Conf(\mathbb{S}^3) \simeq \Conf(\mathbb{B}^4)$. 

We now briefly discuss the above and other features of our conformal maximisation and index characterisation results.

\subsection{Conformal monotonicity}

In the setting of closed surfaces in $\mathbb{S}^3$, Li and Yau \cite{LY82} observed (using the Willmore energy) that minimal surfaces maximise area in their $\Conf(\mathbb{S}^3)$-orbit. They used this observation to define an important notion of \textit{conformal volume}. An alternative proof was given by Montiel and Ros \cite{MR86} (using a special vector field method). In the closed setting, this conformal maximisation was extended to minimal submanifolds $\Sigma^k \looparrowright \mathbb{S}^n$ of any dimension and codimension by el Soufi and Ilias \cite{eSI}, using a monotonicity approach. 

Turning to free boundary surfaces in the Euclidean ball, Fraser and Schoen showed that a free boundary minimal surface $\Sigma^2 \subset \mathbb{B}^n$ maximises \textit{boundary length} $|\pr\Sigma|$ in its $\Conf(\mathbb{B}^n)$-orbit. Their first proof in \cite{FS11} again uses the Willmore energy, while in \cite{FS13} they gave another proof using a special vector field. Fraser-Schoen defined analogous notions of \textit{boundary conformal volume} and \textit{relative conformal volume}, and posed two natural questions:

\begin{itemize}
\item Does a free boundary minimal submanifold $\Sigma^k \subset \mathbb{B}^n$ maximise boundary mass $|\pr\Sigma|$ in its $\Conf(\mathbb{B}^n)$-orbit, for any dimension $k$? (See \cite[Conjecture 3.7]{FS13}.)
\item Does a free boundary minimal submanifold $\Sigma^k \subset \mathbb{B}^n$ maximise interior area $|\Sigma|$ in its $\Conf(\mathbb{B}^n)$-orbit? (See \cite[Question 3.2]{FS13}; also \cite[Section 7]{Li19}.)
\end{itemize}

For free boundary surfaces in a spherical cap $B^n_R \subset \mathbb{S}^n_R$, the functional $\mathcal{E}^{\pit,R}$ will not depend on the wet surface; as such, submanifolds $(\Sigma^k, \pr\Sigma) \looparrowright (\mathbb{S}^n, \pr B^n_R)$ we more generally define
\[
\mathcal{E}^{R}(\Sigma) := |\Sigma|  +  \cot R\, |\pr\Sigma|.
\]

Our methods will yield the following:

\begin{theorem}
\label{thm:conf-max-sph-intro}
Let $(\Sigma^k,\pr\Sigma)$ be a free boundary minimal submanifold in $(\mathbb{S}^n, \pr B^n_R)$. Given $\Psi \in \Conf(B^n_R)$, let $\Sigma_* = \Psi_* \Sigma$. 

Assume either $k=2$, or $R=\pit$. Then $\mathcal{E}^{R}(\Sigma_*) \leq \mathcal{E}^{R}(\Sigma)$. 
\end{theorem}

Note that as the radius $R\to 0$, the ball $B^n_R \subset \mathbb{S}^n$ should, after a suitable rescaling, converge to the Euclidean ball $\mathbb{B}^n$. Upon this rescaling, when $k=2$, the functionals $\mathcal{E}^{R}$ will converge to the \textit{boundary length} in the limit $R\to 0$; Indeed, in Section \ref{sec:euclid-limit} we show that our methods, combined with a limiting argument, will recover the result of Fraser-Schoen. We believe this may be evidence that the boundary length is indeed the more natural quantity associated to deformations by the conformal group, that is, the first question of Fraser-Schoen above may be more plausible than the second. (However, under our approach the first question should be connected with finding a suitable monotone quantity for higher dimensional minimal submanifolds in the \textit{spherical} setting, which also seems difficult at present; see Remarks \ref{rmk:higher-dim} and \ref{rmk:higher-dim-euc}.) From the contrapositive perspective, the conformal action may not be the ideal candidate to construct canonical \textit{area}-decreasing families. 

In this article, our main approach to conformal maximisation is the monotonicity-style approach of \cite{eSI}, which will also allow some access to higher dimensions (at least when $R=\pit$). As mentioned above, the key is to allow movement through the entire $\Conf(\mathbb{S}^n)$-orbit, adjusting the radius $R$ accordingly (as $\Conf(\mathbb{S}^n)$ always maps geodesic balls to geodesic balls). To state this, recall that $\Conf(\mathbb{S}^n)$ is generated by rotations together with flows by the conformal vector fields $V_a = \pi_{T\mathbb{S}^n}(a)$, where $a$ is a fixed vector in $\mathbb{S}^n$. 

\begin{theorem}
\label{thm:conf-max-sph-intro-2}
Let $(\Sigma^k,\pr\Sigma)$ be a free boundary minimal submanifold in $(\mathbb{S}^n, \pr B^n_R)$. 

Fix $a\in\mathbb{S}^n$ and let $\Psi^a_t$ be the flow of the conformal vector field $V_a$. Let $\Sigma_t = (\Psi^a_t)_*\Sigma$, which is a free boundary submanifold in $(\mathbb{S}^n, \pr B^n_{R_t}(o_t))$, where $B^n_{R_t}(o_t) = \Psi^a_t(B^n_R)$. 

Assume either: $k=2$; or $R=\pit$ and $a\perp e_0$. Then for $t\geq 0$ we have $\frac{d}{dt} \mathcal{E}^{R_t}(\Sigma_t) \leq 0$. 
\end{theorem}

Note that Theorem \ref{thm:conf-max-sph-intro-2} certainly implies Theorem \ref{thm:conf-max-sph-intro}, as $\Conf(B_R)$ is just the subgroup of $\Conf(\mathbb{S}^n)$ that preserves the ball $B_R$, and so each element of $\Conf(B_R)$ may certainly be realised by some $\Psi^a_t$ (with $R_t=R$). The more general Theorem \ref{thm:conf-max-sph-intro-cap} follows from the monotonicity Theorems \ref{thm:mono-sph-capillary} and \ref{thm:mono-sph-capillary-hemi} in exactly the same manner. 

We may visualise this scheme as in Figure \ref{fig:conf-BR}: Up to rotations, elements of $\Conf(\mathbb{S}^n)\simeq \Conf(\mathbb{B}^{n+1})$ may be parametrised by points in $\mathbb{B}^{n+1}$. Furthermore, the subgroups $\Conf(B_R)$ may be put in correspondence with the sets $\mathcal{S}_{\cos R}$, where $\mathcal{S}_c := \{ x\in\mathbb{B}^{n+1} | \langle x,e_0\rangle = |x|^2 c \}$. (The sets $\mathcal{S}_c$, $c\in \mathbb{R}$ foliate $\mathbb{B}^{n+1}$, but only those with $|c|<1$ correspond to transformations which preserve some $B_R$.) Flows of the vector fields $V_a$ correspond to rays from the origin, and so may reach any point in $\mathbb{B}^{n+1}$. Finally, we remark that each ray intersects a given slice $\mathcal{S}_{\cos R}$ in (at most) one point, other than the origin. 

\begin{figure}
\includegraphics[scale=0.6]{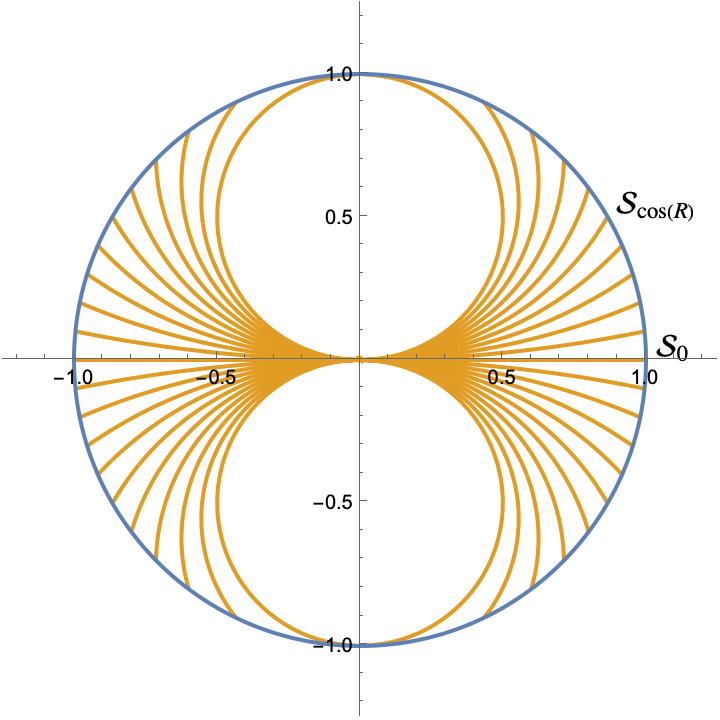}
\caption{Diagram representing the sets $\mathcal{S}_{\cos R} \subset \mathbb{B}^{n+1}$, which correspond to subgroups $\Conf(B^n_R)\subset \Conf(\mathbb{B}^{n+1})$.}
\label{fig:conf-BR}
\end{figure}

We also obtain similar conformal monotonicity results for other contact angles; see Section \ref{sec:conf-mono-cap} for details. 

\subsection{Index characterisation}

The Morse index of a minimal surface is a fundamental geometric quantity that (essentially) counts the dimension of its unstable manifold, when considered as a critical point of the area functional. 

In the closed setting, the minimal surface in $\mathbb{S}^3$ of least index is the equator, which has index 1. Urbano \cite{Ur90} showed that there is a gap to the second least index - again realised by the Clifford torus, which has index 5. He further showed that having index equal to 5 characterises the Clifford torus. As mentioned above, this was one of the crucial properties in Marques-Neves' proof of the Willmore conjecture \cite{MN14}. 

Let us briefly mention the analogue of Question \ref{q:energy} in the setting of free boundary minimal surfaces. In the Euclidean ball $\mathbb{B}^3$, the flat equatorial disc has least area and least index (equal to 1), and it is an open question whether the critical catenoid has second least area (see \cite[Open Question 10]{Li19}) or second least index (see \cite[Open Question 6]{Li19}). The index of the critical catenoid was computed to be 4 by several authors \cite{Dev19, SZ19, Tr20}. Tran \cite{Tr20} additionally proved that having index 4 characterises the critical catenoid under the additional hypothesis that the minimal surface is an \textit{annulus}. Similar results were obtained by Medvedev \cite{Me23} for free boundary minimal surfaces in $B_R\subset \mathbb{S}^3$. 

We also mention the interesting work of Devyver \cite{Dev21}, who points out that the conformal method of Urbano \cite{Ur90} seems to require a suitable variational inequality, which classically arises from the min-max characterisation of eigenvalues for the index form, together with natural eigenfunctions. For free boundary minimal surfaces in $\mathbb{B}^3$, an immediate issue is that one does not have such easily identifiable eigenfunctions. (Devyver introduces a new eigenvalue problem which partially addresses this issue.) We believe that this issue is related to the observation above that the $\Conf(\mathbb{B}^3)$-action seems more suitable for decreasing boundary length than the area functional, whereas $\Sigma$ is critical for the latter. 

For capillary minimal surfaces in the hemisphere $\mathbb{S}^3_+$, however, our conformal maximisation results above show that the $\Conf(\mathbb{S}^3_+)$-action does decrease the natural capillary energy $\mathcal{E}^{\pit,\gamma} = \mathcal{A}^\gamma$. Moreover, one can verify (see Section \ref{sec:cap-morse}) that in this setting, the components of the Gauss map indeed give eigenfunctions of the index form. This natural interaction of the conformal deformations with the capillary energy $\mathcal{A}^{\gamma}$ is what allows us to prove Theorem \ref{thm:urbano-intro}, and furthermore leads us to believe that Question \ref{q:energy} is most promising for capillary minimal surfaces in the hemisphere. (Note again that the energies of capillary minimal surfaces in a hemisphere correspond to the energy densities of capillary minimal \textit{cones}.)

Analogously, we suggest that, to make progress on whether the critical catenoid has second least area, it may be prudent to explore canonical deformations of free boundary minimal surfaces \textit{other than conformal deformations}, since in the Euclidean ball, the conformal group appears to interact more naturally with boundary length rather than area. 

\subsection{Overview of the paper}

We begin by discussing some basic notions in Section \ref{sec:prelim}. 

In Section \ref{sec:conf-ball}, we review some explicit presentations of the conformal group of round balls; these may not be new, but we have included this section (and Appendix \ref{sec:conf-mono-overflow}) as the precise descriptions may not be readily available in the literature. We then proceed to the conformal maximisation - first using the Willmore energy in dimension 2 (Section \ref{sec:conf-willmore}), then by the more general monotonicity methods (Section \ref{sec:conf-mono}). In Section \ref{sec:euclid-limit} (and Appendix \ref{sec:model}) we discuss applications in the Euclidean limit. 

Similarly, in Section \ref{sec:quadratic} we review some of the spectral theory associated to quadratic forms on manifolds with boundary; again these results may not be new, but we find it helpful to give a clear treatment, and introduce our (somewhat nonstandard) realisation of the spectra. Finally, in Section \ref{sec:index-apps} we apply the spectral theory to discuss the spectral index and Morse index of minimal submanifolds, and in particular prove our index characterisation for capillary minimal surfaces in the hemisphere. 

\subsection*{Acknowledgements}

The author would like to acknowledge the manyfold contributions of his friend and collaborator Keaton Naff. This project owes much of its development to the many insightful conversations we have enjoyed, as well as Keaton's constant encouragement. We hope that Keaton's ideas and influence will be evident in this work and in many more works to come. 

The author would also like to thank Ailana Fraser for her interest in this work, as well as Hung Tran for insightful discussions about his work. 

JZ was supported in part by a Sloan Research Fellowship, and the National Science Foundation under grant DMS-2439945. 

\section{Preliminaries}
\label{sec:prelim}

\subsection{Notation}
\label{sec:notation}

We will denote by $\mathbb{B}^n_R(o)\subset \mathbb{R}^n$ the Euclidean ball of radius $R$ and centre $o$, and by $B^n_R(o)\subset\mathbb{S}^n$ the geodesic ball of (spherical) radius $R \in (0,\pi)$ and centre $o$. If the centre is not specified, we will take $\mathbb{B}^n_R = \mathbb{B}^n_R(0)$ or $B^n_R = B^n_R(e_0)$. If the radius of a Euclidean ball is not specified it will be taken to be 1, that is, $\mathbb{B}^n = \mathbb{B}^n_1$. 

\subsection{Submanifolds and contact angles}
\label{sec:contact-angle}

In this article, we will consider immersed compact submanifolds with boundary $(\Sigma^k,\pr\Sigma) \looparrowright (M^n,S)$, where $M$ is either $\mathbb{R}^{n}$ or $\mathbb{S}^{n}$, and $S$ is a closed embedded hypersurface in $M$ with a unit normal $\bar{\eta}$. We emphasise that the metric on $M$ and the normal $\bar{\eta}$ on $S$ are chosen as part of this setup, although their notation may be suppressed. If the barrier is written as the boundary of a region, $S=\pr \Omega$, then the barrier normal $\bar{\eta}$ should be understood to be the \textit{outer} unit normal. We denote by $\eta$ the outer unit conormal of $\pr\Sigma$ in $\Sigma$. We will often conflate $\eta$ with its pushforward $\pr_\eta x$, where $x$ is the immersion map.

In general, we use $\bar{\nabla}$ for the connection on $M$ and $\nabla$ for the connection on the submanifold $\Sigma$. We will reserve $D$ for the Euclidean connection and $\delta$ for the Euclidean metric. In particular, if $M=\mathbb{S}^n \hookrightarrow \mathbb{R}^{n+1}$ then $\bar{\nabla}$ will refer to the connection on $\mathbb{S}^n$ induced by the unit round metric $\bar{g}$. 

We write $\pi^{\perp}$ for the projection to the normal bundle of $\Sigma$ (in $M$), and the shorthand $\bar{\nabla}^{\perp} = \pi^{\perp} \circ \bar{\nabla}$. The second fundamental form of $\Sigma$ in $M$ is the normal bundle-valued 2-tensor $\mathbf{A}_\Sigma(X,Y) =\bar{\nabla}^{\perp}_X Y$, and the mean curvature vector is $\mathbf{H}_\Sigma = \tr_{T\Sigma} \mathbf{A}_\Sigma$. We emphasise that if $M=\mathbb{S}^n\hookrightarrow \mathbb{R}^{n+1}$ then, unless otherwise qualified, $\mathbf{H}_\Sigma$ still denotes the mean curvature vector in $\mathbb{S}^n$. The submanifold $\Sigma$ is \textit{minimal} in $M$ if $\mathbf{H}_\Sigma \equiv 0$. 

If $\Sigma$ is a 2-sided hypersurface, then we denote by $\nu$ the unit normal of $\Sigma$ in $M$. In the hypersurface case, we may work with the scalar-valued second fundamental form, for which our convention is $A_\Sigma(X,Y) = \langle \bar\nabla_X \nu, Y\rangle$. The second fundamental form of the hypersurface $S$ is always $k_S(X,Y) = \langle \bar \nabla_X \bar{\eta}, Y\rangle$. We will drop the subscripts when clear from context.

There are a number of parametric concepts which may be associated to variational problems for submanifolds with boundary, and many different conventions may be chosen. As such, we now take particular care to describe the precise assumptions we will work with.

\begin{definition}
Let $(\Sigma^k,\pr\Sigma) \looparrowright (M^n,S)$ be an immersed submanifold with boundary. We say that $\Sigma$ contacts $S$ with angle $\gamma\in(0,\frac{\pi}{2}]$ along $\pr\Sigma\subset S$ if 
\begin{equation}\label{eq:contact-angle-def}\tag{$\dagger$} \langle \eta,\bar{\eta}\rangle = \sin\gamma>0.
\end{equation}

In the special case $\gamma =\pit$, the condition (\ref{eq:contact-angle-def}) is equivalent to $\eta = \bar{\eta}$ along $\pr\Sigma\subset S$, and we say that $\Sigma$ is a \textit{free boundary} submanifold. 

We say that an immersion is \textit{constrained} in $\Omega\subset M$ if the image of (the interior of) $\Sigma$ lies in (the interior of) $\Omega$. 
\end{definition}

For a discussion our terminology for `constant contact angle', `proper' and `constrained', the reader is encouraged to consult \cite[Section 2.1]{NZ25a}, particularly Remarks 2.2-2.4 therein. As it will be useful, we reproduce the following abbreviated remark:
\begin{remark}
\label{rmk:constrained}
Consider a minimal submanifold $(\Sigma, \pr\Sigma)\looparrowright (\mathbb{S}^N, \pr B_R)$. If $R\leq \pit$ and $\Sigma$ lies in the hemisphere $\{x_0\geq 0\}$, then an application of the strong maximum principle to $x_0$ shows that $\Sigma$ must be constrained in $B_R$. 
\end{remark}

The capillary problem typically involves a notion of wetting energy. We set a version of this out from a parametric perspective as follows. (The resulting notion of `capillary' surface will be somewhat more general than that discussed in the introduction, and the below notion is the one we will use for the remainder of the article.)

\begin{definition}
\label{def:capillary}
Let $(\Sigma^{n-1},\pr\Sigma) \looparrowright (M^n,S)$ be a 2-sided hypersurface that contacts $S$ with angle $\gamma$ along $\pr\Sigma$. 

If there is an \textit{immersed} $(n-1)$-manifold with boundary $S^-\looparrowright S$ so that $\pr S^- = \pr \Sigma$, and the outer conormal $\bar{\nu}$ of $\pr\Sigma$ in $S^-$ satisfies 
\begin{equation}\label{eq:angle-balance}\tag{$\ddagger$}\sin\gamma \, \bar{\eta} = \eta + \cos\gamma \, \bar{\nu},\end{equation}
then we say that $\Sigma$ is a \textit{capillary} hypersurface with wet surface $S^-$. Note that (\ref{eq:angle-balance}) certainly implies the contact angle condition (\ref{eq:contact-angle-def}).

We may consider additional conditions on the wet surface; in decreasing order of strength:
\begin{enumerate}[(a)]
\item $S^-$ is embedded as a region in $S$. 
\item Each component of $S^-$ is embedded as a region in $S$.
\end{enumerate}

For hypersurfaces with constant contact angle $\gamma \in(0,\pit)$ but which may not admit a \textit{global} wet surface as above, we may still define a \textit{local} wetting energy as follows: (for $\gamma\neq \pit$) the projection $\eta - \sin\gamma\,\bar{\eta}$ of $\eta$ to $TS$ is nonzero and orthogonal to $\pr\Sigma$. By normalising, it follows that there is a unit normal $\bar{\nu}$ of $\pr\Sigma$ in $S$ satisfying the balancing condition ($\ddagger$). Given a smooth variation $X: \Sigma \times [0,t] \looparrowright M$ of $\Sigma$ (which maps $\pr\Sigma$ into $S$ for each $t$), one may consider the local wetting energy $W(t)=\int_{\Sigma \times [0,t]} X^*\d\mu_S$, where the integral is defined with respect to the orientation induced on $\pr\Sigma$ in $S$ by $\bar{\nu}$. This corresponds to the (signed) area swept out by $\pr\Sigma$. 
\end{definition}

\begin{remark}
\label{rmk:wet-surface}
Consider a 2-sided hypersurface $(\Sigma^{n-1},\pr\Sigma) \looparrowright (M^n,S)$ that contacts $S$ with angle $\gamma$ along $\pr\Sigma$ as above. Note that for embedded $\Sigma$, even though $\pr\Sigma$ will separate $S$ into 2 regions, when $\pr\Sigma$ has more than 1 component it may not be clear that Definition \ref{def:capillary}(a) can be satisfied - it may be that neither component of $S\setminus \pr\Sigma$ satisfies $(\ddagger)$ on all of $\pr\Sigma$. 

Nevertheless, even if we only assume that each component of the boundary $\pr\Sigma$ is embedded in $S$, it is always possible to find a wet surface $S^-$ satisfying Definition \ref{def:capillary}(b) above. Indeed, each component $\pr\Sigma^{(j)}$ of $\pr\Sigma$ will itself separate $S$, so one may choose $S^-_j$ to be the component of $S\setminus \pr\Sigma^{(j)}$ that satisfies $(\ddagger)$ on $\pr\Sigma^{(j)}$. In particular, taking the abstract union of the $S^-_j$ always gives an \textit{immersed} wet surface $S^-$ (although there may be pairwise overlaps). (For example, if $n-1=2$ and $S$ is a topological 2-sphere, then one may take $S^-$ to be an abstract union of discs in $S$, each bounded by one of the loops of $\pr\Sigma$.)

On the other hand, if $\Sigma$ is a smoothly embedded capillary hypersurface which arises from the \textit{global} variational problem for (\ref{eq:cap-energy-intro}) - that is, $\Sigma$ is the relative boundary of a critical domain $\Omega$ in $(M, \pr M)$ - then $\Sigma$ is certainly constrained in $(M, \pr M)$ and the subset $S^- = \pr\Omega \cap \pr M$ of $S=\pr M$ gives a wet surface satisfying $(\ddagger)$, as per Definition \ref{def:capillary}(a). 
\end{remark}

The following lemma is a standard calculation but we include it to illustrate our conventions. Consider a 2-sided hypersurface $\Sigma$ and recall that $A(X,Y) = - \langle \bar{\nabla}_X Y,\nu \rangle$ and $k_S(X,Y) = - \langle \bar{\nabla}_X Y, \bar{\eta} \rangle$ denote the scalar-valued second fundamental forms of $\Sigma$ and $S$ in $M$ respectively. Note that the normal bundle of $\pr\Sigma$ in $M$ has orthonormal frames given by $\{\nu,\eta\}$ or $\{\bar{\eta},\bar{\nu}\}$. Let $h(X,Y) = \bar{\nabla}^\perp_X Y$ denote the second fundamental form of $\pr \Sigma$ in $M$, so that $h^\eta = -\langle h,\eta \rangle$ and $h^{\bar{\nu}}=-\langle h,\bar{\nu} \rangle$ are the scalar-valued second fundamental forms of $\pr\Sigma$ in $\Sigma$ and $S$ respectively. Similarly, in this case we define the scalar mean curvatures $H^\Sigma = \tr_{T\Sigma} A$, $H^S = \tr_{TS} k$, $H^{\pr\Sigma}_\Sigma = \tr_{T\pr \Sigma} h^\eta$, $H^{\pr\Sigma}_S = \tr_{T\pr\Sigma} h^{\bar{\nu}}$. 

\begin{lemma}
\label{lem:bdry-mc}
Let $(\Sigma,\pr\Sigma)$ be a 2-sided hypersurface in $(M,S)$ that contacts $S$ with angle $\gamma$ along $\pr\Sigma$. Let $\bar{\nu}$ be a unit conormal of $\pr\Sigma$ in $S$, chosen so that 
 \begin{equation}\tag{$\ddagger$}\sin\gamma \, \bar{\eta} = \eta + \cos\gamma \, \bar{\nu}.\end{equation}

Then for $X,Y \in T\pr \Sigma$, we have \[ k_S(X,Y) = \cos \gamma \, A(X,Y) + \sin\gamma \, h^\eta(X,Y),\] 
\[ A(X,Y) = \cos \gamma \, k_S(X,Y) + \sin\gamma \, h^{\bar{\nu}}(X,Y).\] 

In particular, 


\[ H^{S} - k_S(\bar{\nu},\bar{\nu}) = \cos\gamma \, (H^\Sigma - A(\eta,\eta)) + \sin\gamma \, H^{\pr \Sigma}_\Sigma,\]
\[ H^\Sigma - A(\eta,\eta) = \cos\gamma \, (H^{S} - k_S(\bar{\nu},\bar{\nu})) + \sin\gamma \, H^{\pr \Sigma}_{S}.\]
\end{lemma}
\begin{proof}
The first two identities follow by respectively pairing both sides of $\bar{\eta} = \cos \gamma \, \nu + \sin \gamma \, \eta$  and $\nu = \cos\gamma \, \bar{\eta} + \sin\gamma\, \bar{\nu}$ with $\bar{\nabla}_X Y$, where $X,Y\in T\pr\Sigma$. 

Tracing over $T\pr\Sigma$ gives the last two identities, as $\eta, \bar{\nu}$ are the respective conormals. 
\end{proof}

\begin{corollary}
\label{cor:bdry-mc}
Let $(\Sigma,\pr\Sigma)$ be a 2-sided surface in $(\mathbb{S}^3,\pr B_R)$ that contacts $S=\pr B_R$ with angle $\gamma$ along $\pr\Sigma$. Let $\bar{\nu}$ be a unit conormal of $\pr\Sigma$ in $S$, satisfying $(\ddagger)$ as above. Then 

\[ \cot\gamma \,(H^\Sigma - A(\eta,\eta)) =\csc \gamma \cot R - H^{\pr\Sigma}_\Sigma = \cos\gamma \cot\gamma \cot R + \cos\gamma \, H^{\pr\Sigma}_S.\]
\end{corollary}

\subsection{Conformal maps and submanifolds}

Consider a submanifold $\Sigma$ immersed by the map $x:\Sigma \looparrowright M$. Given a diffeomorphism $\Phi : M\to N$, we denote by $\Psi_*\Sigma$ the submanifold immersed by $\Psi\circ x: \Sigma \looparrowright M$. If $\Sigma$ contacts $S$ at angle $\gamma$, and $\Psi$ is a conformal map, then it is clear that $\Psi_*\Sigma$ contacts $\Psi(S)$ at angle $\gamma$. 

It will also be helpful to recall the transformation formula for the second fundamental form under conformal diffeomorphisms: if $\Psi^* g_N = e^{2\psi} g_M$ and $k=\dim \Sigma$, then

\begin{equation}
\label{eq:2ff-conf}
\mathbf{A}_{\Psi_*\Sigma}(\Psi(x)) = (d\Psi)_x \left(\mathbf{A}_\Sigma(x) - g \otimes \pi^{\perp}(\bar{\nabla} \psi(x)) \right),
\end{equation}

\begin{equation}
\label{eq:mc-conf}
\mathbf{H}_{\Psi_*\Sigma}(\Psi(x)) = e^{-2\psi(x)} (d\Psi)_x \left((\mathbf{H}_\Sigma(x) - k \pi^{\perp}(\bar{\nabla} \psi(x)) \right).
\end{equation}

To illustrate these conventions we check that in dimension $k=2$, the norm of the traceless second fundamental form has a scaling property under conformal diffeomorphisms:

\begin{lemma}
\label{lem:traceless-2ff}
Let $\Sigma_* = \Psi_*\Sigma$ as above. Then
\[|\mathring{\mathbf{A}}_{\Sigma_*}|^2(\Psi(x)) = e^{-2\psi(x)} |\mathring{\mathbf{A}}_\Sigma|^2(x)  ,\]
\end{lemma}
\begin{proof}
The traceless second fundmental form is $\mathring{\mathbf{A}} = \mathbf{A}-\frac{1}{2}\mathbf{H} g$, and has norm \[|\mathring{\mathbf{A}}|^2 =|\mathbf{A}|^2 + \frac{2}{4}|\mathbf{H}|^2 - |\mathbf{H}|^2 = |\mathbf{A}|^2 - \frac{1}{2}|\mathbf{H}|^2.\] Contracting the above formulae we have 

\[|\mathbf{A}_{\Sigma_*}|^2(\Psi(x)) = e^{-4\psi(x)}e^{2\psi(x)} \left( |\mathbf{A}_\Sigma|^2(x) + 2 |\bar{\nabla}^{\perp} \psi|^2(x) - 2g(\mathbf{H}_\Sigma(x), \bar{\nabla}^{\perp}\psi(x)) \right),\]

\[|\mathbf{H}_{\Sigma_*}|^2(\Psi(x)) = e^{-4\psi(x)}e^{2\psi(x)} \left( |\mathbf{H}_\Sigma|^2(x) +4 |\bar{\nabla}^{\perp} \psi|^2(x) - 4g(\mathbf{H}_\Sigma(x), \bar{\nabla}^{\perp}\psi(x)) \right),\]

and subtracting and simplifying gives the result. 
\end{proof}

\section{Conformal group of round balls}
\label{sec:conf-ball}

In this section, we describe the conformal group of Euclidean balls $\mathbb{B}^n$, the round sphere $\mathbb{S}^n$ as well as the geodesic balls $B^n_R$. This section is comprised of rather well-known results, so may be skipped by expert readers, but we have included it to make our presentation of the conformal group completely concrete and highlight precisely what information is needed. 

In our presentation, we will take the description of $\Conf(\mathbb{B}^{n+1}) \simeq \Conf(\mathbb{S}^n)$ in terms of fractional linear transformations for granted, and deduce realisations of $\Conf(B^n_R)$ as a subgroup of $\Conf(\mathbb{B}^{n+1})$. We will also describe $\Conf(\mathbb{S}^n)$ as flows of conformal vector fields, and the relation between these presentations. 

\subsection{Fractional linear presentation}
\label{sec:conf-euc}

For a given $y\in \mathbb{B}^{n+1}$, the fractional linear transformation
\begin{equation}
\label{eq:Phi-y}
\Phi_y(x) = \frac{(1-|y|^2)x + (1+2\langle x,y\rangle+|x|^2)y}{1+2\langle x,y\rangle + |x|^2|y|^2}
\end{equation}
defines a conformal diffeomorphism $\mathbb{B}^{n+1}\to \mathbb{B}^{n+1}$, which satisfies $\Phi_y(0)=y$. It may be readily verified that its inverse map is $\Phi_y^{-1} = \Phi_{-y}$. 

In fact, the conformal group $\Conf(\mathbb{B}^{n+1})$ consists precisely of maps $\Theta \circ \Phi_y$, for some $\Theta \in \mathrm{SO}(n+1)$ and $y\in\mathbb{B}^{n+1}$. Moreover, the restriction map gives an identification $\Conf(\mathbb{B}^{n+1}) \simeq \Conf(\mathbb{S}^n)$.

A standard calculation gives that the conformal factor $\Phi_y^* \delta = e^{\phi} \delta$ is given by 

\begin{equation}
\label{eq:conf-factor-phi}
e^\phi = \frac{1-|y|^2}{1+2\langle x,y\rangle + |x|^2|y|^2}.
\end{equation}

(See also Lemma \ref{lem:conf-factor}, in the case of the restriction to $\mathbb{S}^n=\pr \mathbb{B}^{n+1}$.)

\begin{remark}
Restricted to $\mathbb{S}^n$, the map $\Phi_y$ acts as
\begin{equation}
\label{eq:Phi-y-Sn}
\Phi_y(x) = \frac{(1-|y|^2)x + 2(1+\langle x,y\rangle)y}{1+2\langle x,y\rangle + |y|^2}
\end{equation}
Montiel-Ros \cite{MR86} give a slightly different description, namely 
\[ x \mapsto \frac{x+ (\lambda + \mu\langle x, z\rangle)z}{\lambda(1+\langle x,z\rangle)},\]
where $\lambda = (1-|z|^2)^{-1/2}$ and $\mu=(\lambda -1)|z|^{-2}$. 

We note that these maps coincide when $z= \frac{2y}{1-|y|^2}$. 
\end{remark}

We will need the following fact about fractional linear transformations (cf. \cite[Theorem 3.2.1]{Bear83}):

\begin{proposition}
\label{prop:balls-to-balls}
$\Phi_y$ maps round balls to round balls. 
\end{proposition}

\begin{lemma}
\label{lem:conf-translation}
Let $a\in\mathbb{S}^n$, $s\in(-1,1)$. Then 
\begin{equation} \langle \Phi_{sa}(x), a\rangle =  \frac{2s+(1+s^2) \langle x,a\rangle}{1+s^2 + 2s \langle x,a\rangle}\end{equation}
\end{lemma}

\subsection{Conformal group of spherical caps}
\label{sec:conf-sph}

In this subsection, we emphasise the explicit presentation of the conformal group $\Conf(B^n_R)$ as the subgroup of $\Conf(\mathbb{B}^{n+1})$ that preserves $B^n_R\subset \mathbb{S}^n$ (see Proposition \ref{prop:conf-description} below). Of course, $\Conf(B^n_R)$ is isomorphic to $\Conf(\mathbb{B}^n)$, and many of the properties we will review may be verified by direct computation of this isomorphism. In our presentation, we find it efficient and hopefully enlightening to deduce these properties from properties of $\Conf(\mathbb{B}^{n+1})$. 

Given $R\in(0,\pi)$, we define the number \begin{equation}s_R = \tan \frac{\pit-R}{2}.\end{equation} (In particular, by the tangent half-angle formulae $\frac{2s_R}{1+s_R^2}=\cos R$, $\frac{1-s_R^2}{1+s_R^2} = \sin R$.) With coordinates so that $o=e_0$ is the centre of $B^n_R$, it follows from Lemma \ref{lem:conf-translation} that
\begin{equation}
\label{eq:e0-to-e0}
 \langle\Phi_{s_R e_0}(x) ,e_0\rangle = \frac{\langle x,e_0\rangle +\cos R}{1 + \langle x, e_0\rangle \cos R}.
 \end{equation}

Consequently, $\Phi_{s_R e_0}$ provides a conformal diffeomorphism $B^n_\pit \to B^n_R$, with inverse $\Phi_{-s_R e_0}$. Moreover, this conformal equivalence commutes with the subgroup $\mathrm{SO}(n+1)^{o}$ of rotations fixing $o$. 

Inverse stereographic projection gives a conformal equivalence $\Xi: \mathbb{R}^n \to \mathbb{S}^n\setminus\{-o\}$ which maps $\mathbb{B}^n$ to $B^n_\pit$ and $0\to o=e_0$; this correspondence clearly sends  $\mathrm{Isom}(\mathbb{B}^n)=\mathrm{SO}(n)$ to the subgroup $\mathrm{SO}(n+1)^{o}$. If we realise $\mathbb{B}^n$ as the subset $\{\langle x,e_0\rangle =0\}\subset \mathbb{B}^{n+1}$, then for $z\in \mathbb{B}^n$ we have the explicit formula 
\begin{equation}
\label{eq:stereo}
 \Xi(z) = \frac{1-|z|^2}{1+|z|^2}e_0 +  \frac{2z}{1+|z|^2}
\end{equation}

\begin{definition}
It will be convenient to define the sets \[\mathcal{S}_c = \{ x\in\mathbb{B}^{n+1} | \langle x,e_0\rangle = |x|^2 c \}.\] 
For example, the above discussion identifies $\mathbb{B}^n \simeq \mathcal{S}_0$. We remark that each $\mathcal{S}_c$ is the intersection of $\mathbb{B}^{n+1}$ with a particular sphere or hyperplane. (See also Figure \ref{fig:conf-BR}.)
\end{definition}

We first note that elements of $\Conf(B^n_R)$ are determined, up to rotation, by where they send the centre $o=e_0$. 

\begin{lemma}
\label{lem:rotation}
Suppose that $\Psi\in \Conf(B^n_R)$ fixes $o$, that is, $\Psi(o)=o$. Then $\Psi$ is given by a rotation in $\mathrm{SO}(n+1)^o$. 

If $\Psi, \Psi' \in \Conf(B^n_R)$ satisfy $\Psi(o) = \Psi'(o)$, then there is $\Theta \in \mathrm{SO}(n+1)^o$ so that (acting on $B^n_R$) we have $\Psi' = \Theta \circ \Psi$.
\end{lemma}
\begin{proof}
First consider $R=\pit$. If $\Psi \in \Conf(B^n_\pit)$ fixes $o$, then its conjugate by $\Xi$ is an element of $\Conf(\mathbb{B}^n)$ that fixes the origin. By the description of $\Conf(\mathbb{B}^n)$ above, the latter must be rotation in $\mathrm{SO}(n)$, and so $\Psi$ must be a rotation in $\mathrm{SO}(n+1)^o$. 

Similarly, note that $\Phi_{s_R e_0}(e_0)=e_0$. So if $\Psi \in \Conf(B^n_R)$ fixes $o=e_0$, then its conjugate by $\Phi_{s_R e_0}$ is an element of $\Conf(B^n_\pit)$ that still fixes $o$. Again by the above discussion, $\Psi$ and its conjugate must be induced by a rotation in $\mathrm{SO}(n+1)^o$.

For the last statement, simply note that $\Psi' \circ  \Psi^{-1}$ fixes $o$ and hence must be a rotation. 
\end{proof}

We now give an explicit realisation of $\Conf(B^n_R)$ as a subgroup of $\Conf(\mathbb{B}^{n+1})$. 

\begin{proposition}
\label{prop:conf-description}
Let $\Theta \in \mathrm{SO}(n+1)^o$ be a rotation fixing $o=e_0$ and let $y\in \mathcal{S}_0$ be a point in $\mathbb{B}^{n+1}$ orthogonal to $e_0$. The map $\Phi_{s_R e_0} \circ \Phi_y \circ \Phi_{-s_R e_0}$ restricts to a conformal diffeomorphism of $B^n_R$. Consequently, we may identify 
\begin{equation}
\label{eq:conf-B-R-presentation}
\begin{split}
 \Conf(B^n_R)  &=\{ \Theta \circ \Phi_{s_R e_0} \circ \Phi_y \circ \Phi_{-s_R e_0} | \Theta \in \mathrm{SO}(n+1)^o, y\in \mathcal{S}_0\}
 \\&= \{ \Psi \in \Conf(\mathbb{B}^{n+1}) | \Psi(B^n_R)=B^n_R\} 
 \end{split}
 \end{equation}
\end{proposition}
\begin{proof}
We first consider the case $R=\pit$, in which case $\Phi_{s_R e_0} = \Phi_0 = \id$. It is easily verified that if $\langle y,e_0\rangle=0$, then $\Phi_y$ preserves the hemisphere $B^n_\pit$. 

We claim that for any $z\in B^n_\pit$, there is such a $y$ for which $\Phi_y(o)=z$. Indeed, as $\langle y,e_0\rangle=0$, we have 
\begin{equation}
\label{eq:Phi-y-o}
\Phi_y(e_0) = \frac{(1-|y|^2)e_0 + 2y}{1+|y|^2}.
\end{equation}
 Choosing coordinates so that $z= (\cos \zeta)e_0 + (\sin \zeta)e_1$ and taking $y= (\tan \frac{\zeta}{2})e_1$ establishes the claim. (Note that $\zeta \in [0,\pit)$, so that $|\tan \frac{\zeta}{2}|<1$.)
 
Given the claim, Lemma \ref{lem:rotation} now implies that any $\Psi \in \Conf(B^n_\pit)$ is, up to rotation, induced by $\Phi_y$ for some $y\in\mathcal{S}_0$. This establishes the first equality in (\ref{eq:conf-B-R-presentation}) for $R=\pit$. Now recall that $\Phi_{s_R e_0}$ gives a conformal equivalence $B^n_\pit \to B^n_R$ that commutes with $\mathrm{SO}(n+1)^o$. In particular, conjugation by $\Phi_{s_R e_0}$ establishes the first equality in (\ref{eq:conf-B-R-presentation}) for general $R$.
 
%
\end{proof}

Henceforth, we will always regard $\Conf(B^n_R)$ as a subgroup of $\Conf(\mathbb{B}^{n+1})$.

\subsection{Flow presentation and relations}
\label{sec:conf-flow}

Given a unit vector $a\in \mathbb{S}^n$, define $u_a:\mathbb{S}^n\to\mathbb{R}$ by $u_a(x) = \langle x,a\rangle$. The gradient
\begin{equation}V_a(x) = \bar{\nabla} u_a(x)= \pi_{T_x\mathbb{S}^n}(a) = a-\langle x,a\rangle x\end{equation}
defines a conformal vector field on $\mathbb{S}^n$. 

We note that 
\begin{equation}
\label{eq:u-grad}
|\bar{\nabla} u_a(x)|^2 = 1-\langle x,a\rangle^2 = 1-u_a(x)^2,
\end{equation}
\begin{equation}
\label{eq:u-hess}
\bar{\nabla}^2 u_a = -\langle x,a\rangle \bar{g} = - u_a\bar{g}.
\end{equation}

 We can explicitly identify the gradient flow of $u_a$ with (restrictions of) the fractional linear presentation above:

\begin{lemma}
\label{lem:flow-relation}
Given $a\in\mathbb{S}^n$, let $\Psi^a_t$ be the 1-parameter family of conformal diffeomorphisms generated by $V_a$. Then (acting on $\mathbb{S}^n$)
\[ \Psi^a_t = \Phi_{sa},\] where $s=\tanh \frac{t}{2}$. 
\end{lemma}
\begin{proof}
We need only check the flow equation, that is, $\pr_t \Phi_{sa}(x) = V_a(\Phi_{sa}(x))$ for $x\in \mathbb{S}^n$. Indeed, using (\ref{eq:Phi-y-Sn}) we write 

\[\Phi_{sa}(x) = \frac{(1-s^2)x + 2s(1+s u_a)a}{1+2s u_a + s^2},\]

in particular 

\[\langle \Phi_{sa}(x), a\rangle = \frac{(1-s^2)u_a + 2s(1+s u_a)}{1+2s u_a + s^2}= \frac{u_a+2s+s^2u_a}{1+2s u_a + s^2}.\]

It follows that \[ 
\begin{split}
V_a(\Phi_{sa}(x))&=a- \langle \Phi_{sa}(x), a\rangle \Phi_{sa}(x) 
\\&= a-\frac{(u_a+2s+s^2u_a )( (1-s^2)x + 2s(1+s u_a)a)}{(1+2s u_a + s^2)^2}
\\&= \frac{1-s^2}{(1+2s u_a + s^2)^2} \left(  -(u_a+2s+s^2u_a )x + (1+2su_a +s^2(-1+2u_a^2))a\right)
\end{split}\]

Differentiating gives

\[\begin{split}
\pr_s \Phi_{sa}(x) &= \frac{  2 (1+2s u_a + s^2)(-sx +a + 2s u_a a) - 2((1-s^2)x + 2s(1+s u_a)a)(u_a+s)  }{(1+2s u_a + s^2)^2}
\\&= \frac{2}{(1+2s u_a + s^2)^2} \left( -(u_a+2s+s^2u_a)x +(1+2su_a +s^2 (-1+2u_a^2))a\right)
\end{split}\]

Thus we have shown $\pr_s \Phi_{sa}(x)  = \frac{2}{1-s^2} V_a(\Phi_{sa}(x)).$ As $s=\tanh\frac{t}{2}$ we indeed have $\frac{ds}{dt} = \frac{1-s^2}{2}$, which completes the proof. 
\end{proof}

The following lemma is contained in \cite[Proof of Theorem 1.1]{eSI}, but for the reader's convenience we have translated it here:

\begin{lemma}
\label{lem:conf-factor}
Fix $a\in\mathbb{S}^n$ and let $\Psi^a_t$ be the gradient flow of $u_a:\mathbb{S}^n\to \mathbb{S}^n$ as above. For any $t$, we have $(\Psi^t_a)^*\bar{g} = e^{2\psi_t} \bar{g}$, where $\psi_t:\mathbb{S}^n\to\mathbb{S}^n$ satisfies
\begin{equation}
\label{eq:conf-factor}
e^{2\psi(x)} =   \frac{1-(u_a\circ \Psi^a_t)(x)^2}{1-u_a(x)^2}.
\end{equation}

\begin{equation} 
\label{eq:conf-factor-grad}
\bar{\nabla} \psi_t = \frac{u_a-u_a\circ \Psi^a_t}{1-u_a^2}\bar{\nabla} u_a,
\end{equation}
\end{lemma}
\begin{proof}
For convenience, write $u=u_a$ and $\Psi=\Psi^a_t$. As $\Psi$ is a conformal diffeomorphism, we certainly have $\Psi^* \bar{g} = e^{2\psi} \bar{g}$ for some $\psi=\psi_t$. Moreover, since $\Psi^a_t$ was the gradient flow of $u$, we have that $d\Psi_x(V_a(x)) = V_a(\Psi(x))$, where as above $V_a = \bar{\nabla} u$. 

Therefore, we can determine $\psi$ by measuring the length of $V_a$: For any $x\in\mathbb{S}^n$ we have \[e^{2\psi(x)} = \frac{|d\Psi_x(V_a(x))|^2}{|V_a(x)|^2} =  \frac{|V_a(\Psi(x))|^2}{|V_a(x)|^2} = \frac{1-(u\circ \Psi)(x)^2}{1-u(x)^2}.\] 

Differentiating this relation, it follows that
\[ d\psi = -\frac{u\circ \Psi}{1-(u\circ \Psi)^2} du \circ d\Psi + \frac{u}{1-u^2} du.\] 
But now using again that $d\Psi_x(V_a(x)) = V_a(\Psi(x))$, for any $X\in T_x\mathbb{S}^n$ we have
\[ (du\circ d\Psi)_x(X) = \langle V_a(\Psi(x)), d\Psi_x(X)\rangle = \langle d\Psi_x(V_a(x)), d\Psi_x(X)\rangle = e^{2\psi(x)} \langle V_a(x), X\rangle. \]

Cancelling the factors of $1-(u\circ \Psi)^2$ yields 
\[ d\psi = \frac{u-u\circ \Psi}{1-u^2}du,\]
which implies (\ref{eq:conf-factor-grad}) as desired. 
\end{proof}

We may also precisely describe $\Conf(B^n_R)$ in terms of these conformal flows; we have included such a description but confined it to Appendix \ref{sec:conf-mono-overflow}, as it is not strictly necessary to prove our results.

\section{Conformal maximisation via the Willmore energy}
\label{sec:conf-willmore}

In this subsection, we establish our conformal maximisation results for minimal surfaces using the conformal invariance of the Willmore energy. We begin with this approach as it gives some of the shorter proofs, and is also enlightening as to the definition of the relevant quantities. The key will be that the norm of the traceless second fundamental form has a \textit{pointwise} conformal scaling (cf. Lemma \ref{lem:traceless-2ff}). 

We first consider immersed surfaces $(\Sigma, \pr\Sigma)$ in $(\mathbb{S}^n, \pr B_R)$. By the Gauss equations, the Gaussian curvature satisfies \[ 2K_\Sigma = 2(n-1)-2(n-2) + |\mathbf{H}_\Sigma|^2 - |\mathbf{A}_\Sigma|^2 = 2  +\frac{1}{2}|\mathbf{H}|^2 - |\mathring{\mathbf{A}}_\Sigma|^2.\] Note that the second fundamental form of the barrier is given by $k^S = (\cot R) g|_S$. 

Recall that $\mathcal{E}^{R}(\Sigma) = |\Sigma| + \cot R\, |\pr\Sigma|$. 

\begin{proof}[Proof of Theorem \ref{thm:conf-max-sph-intro} ($k=2$) via Willmore energy]
Let $\Psi \in \Conf(\mathbb{S}^n)$ and let $R_*$ be such that $\Psi$ maps $B^n_R$ to a geodesic ball of radius $R_*$. 

Consider any (not necessarily minimal) immersed surface $(\Sigma, \pr\Sigma)$ in $(\mathbb{S}^n, \pr B_R)$ that contacts $\pr B_R$ at angle $\gamma =\pit$ along $\pr\Sigma$, which has wet surface $S^-$. Recall that $\sin\gamma \, \bar{\eta} = \eta + \cos\gamma \, \bar{\nu}$ on $\pr S^-=\pr\Sigma$, where $\bar{\nu}$ is the outward normal on $\pr \Sigma$ with respect to $S^-$. By Corollary \ref{cor:bdry-mc}, the geodesic curvature of $\pr\Sigma$ in $\Sigma$ is given by $H^{\pr\Sigma}_\Sigma = \cot R$. Applying Gauss-Bonnet to $\Sigma$ gives 
\[ 2\pi \chi(\Sigma) = \int_\Sigma (1+\frac{1}{4}|\mathbf{H}_\Sigma|^2-\frac{1}{2}|\mathring{\mathbf{A}}_\Sigma|^2) + \int_{\pr\Sigma} \cot R.\]

As $\Psi$ is a conformal diffeomorphism, the above also applies to $\Sigma_*=\Psi_*\Sigma$ (with $R$ replaced by $R_*$). By the conformal scaling for $|\mathring{\mathbf{A}}|^2$ (Lemma \ref{lem:traceless-2ff}) we have $\int_{\Sigma_*}|\mathring{\mathbf{A}}_{\Sigma_*}|^2 =\int_\Sigma |\mathring{\mathbf{A}}_\Sigma|^2$. If $\Sigma$ is also minimal, then subtracting the resulting equations gives 

\[0 = |\Sigma|  - |\Sigma_*| - \frac{1}{4}\int_{\Sigma_*} |\mathbf{H}_{\Sigma_*}|^2 + \cot R \,|\pr\Sigma| - \cot R_*\, |\pr\Sigma_*|.\] 

This immediately implies that $\mathcal{E}^{R}(\Sigma) \geq \mathcal{E}^{R_*}(\Sigma_*)$; if $\Psi \in \Conf(B^n_R)$ then $R_*=R$ which completes the proof. 
\end{proof}

Recall that $\mathcal{E}^{R,\gamma}(\Sigma,S^-) = |\Sigma| +\cos\gamma \csc^2 R\, |S^-| + \sin\gamma \cot R\, |\pr\Sigma|$. 

This method also gives the conformal maximisation for capillary surfaces:

\begin{proposition}
Let $(\Sigma,\pr\Sigma)$ be a capillary minimal surface in $(B^n_R, \pr B^n_R)$ with contact angle $\gamma \in (0,\pit]$ and wet surface $S^-$. 

Given $\Psi \in \Conf(\mathbb{S}^n)$, let $\Sigma_* = \Psi_* \Sigma$, $S^-_* = \Psi_* S^-$, and let $R_*$ be such that $\Psi$ maps $B^n_R$ to a geodesic ball of radius $R_*$. Then $\mathcal{E}^{R_*, \gamma}(\Sigma_*, S^-_*) \leq \mathcal{E}^{R,\gamma}(\Sigma,S^-)$. 
\end{proposition}
\begin{proof}
Consider any (not necessarily minimal) immersed hypersurface $(\Sigma, \pr\Sigma)$ in $(\mathbb{S}^3, \pr B_R)$ that contacts $\pr B_R$ at angle $\gamma$ along $\pr\Sigma$ with wet surface $S^-$. Again recall $\sin\gamma \, \bar{\eta} = \eta + \cos\gamma \, \bar{\nu}$, so by Corollary \ref{cor:bdry-mc}, the geodesic curvatures of $\pr\Sigma$ in $\Sigma$ and $S^-$ are related by 
\[H^{\pr\Sigma}_\Sigma +\cos\gamma\, H^{\pr\Sigma}_{S^-} =\sin\gamma \cot R.\] 

Note that the barrier $\pr B_R$ (hence $S^-$) has Gaussian curvature $\csc^2 R$. (One may use the Gauss equations as above, or note that $\pr B_R$ is isometric to a round 2-sphere of radius $\sin R$.)

Applying Gauss-Bonnet to $\Sigma$ and $S^-$ then gives
\[ 2\pi( \chi(\Sigma) + \cos\gamma \, \chi(S^-)) = \int_\Sigma (1+\frac{1}{4}|\mathbf{H}_\Sigma|^2-\frac{1}{2}|\mathring{\mathbf{A}}_\Sigma|^2)  + \cos\gamma \csc^2 R\, |S^-| + \sin\gamma \cot R\, |\pr\Sigma|.\]

As $\Psi$ is a conformal diffeomorphism, the above also applies to $\Sigma_*, S^-_*, R_*$, and by the conformal invariance of $|\mathring{\mathbf{A}}|^2$ we still have $\int_{\Sigma_*}|\mathring{\mathbf{A}}_{\Sigma_*}|^2 =\int_\Sigma |\mathring{\mathbf{A}}_\Sigma|^2$. If $\Sigma$ is also minimal, then subtracting the resulting equations gives 

\[\begin{split} 0 = &|\Sigma|  - |\Sigma_*| - \frac{1}{4}\int_{\Sigma_*} |\mathbf{H}_{\Sigma_*}|^2 +\cos\gamma \csc^2 R\, |S^-| - \cos\gamma \csc^2 R_*\,|S^-_*| \\&+ \sin\gamma \cot R \,|\pr\Sigma| -\sin\gamma \cot R_* \, |\pr\Sigma_*|.\end{split}\] 
\end{proof}

By essentially the same proof, we have the following result for capillary minimal surfaces in the Euclidean ball. We omit the details for the sake of brevity.

\begin{proposition}
Let $(\Sigma,\pr\Sigma)$ be a capillary minimal surface in $(\mathbb{B}^n, \pr \mathbb{B}^n)$ with contact angle $\gamma \in (0,\pit]$ and wet surface $S^-$. 

Given $\Psi \in \Conf(\mathbb{B}^n)$, let $\Sigma_* = \Psi_* \Sigma$, $S^-_* = \Psi_* S^-$. Then \[ \cos\gamma \,|S^-_*|+ \sin\gamma\, |\pr\Sigma_*| \leq \cos\gamma \,|S^-|+ \sin\gamma\, |\pr\Sigma|.\] 
\end{proposition}

\section{Conformal maximisation via monotonicity}
\label{sec:conf-mono}

In this section, we will show that our conformal maximisation results for minimal surfaces also follow from certain monotonicity formulae. These monotonicity methods also allow for generalisations to higher dimensions, in certain cases. We first investigate how the conformal flow moves the cap $B^n_R$, then we detail our proof of monotonicity in the free boundary setting and finally discuss monotonicity in the capillary setting. We will perform our calculations on the flowing surface, but we remark that one could also calculate on the base surface.

\subsection{Evolution of the cap}


In this subsection, we prove some quantitative facts about the conformal flows in Section \ref{sec:conf-flow}. Some of these calculations may also be performed directly using the fractional linear presentation, but we find it thematic to analyse using the flow. 

\begin{lemma}
\label{lem:u-evo}
Fix $x_0, a\in\mathbb{S}^n$. Let $x_t = \Psi^a_t(x_0)$ be the trajectory under the flow by $V_a = \bar{\nabla} u_a$. Then setting $s=\tanh \frac{t}{2}$, we have
\begin{equation}
\label{eq:u-a-s}
u_a(x_t) =\tanh(t + \artanh(u_a(x_0)))=\frac{2s + (1+s^2)u_a(x_0)}{1+s^2+2s u_a(x_0)}. 
\end{equation}\end{lemma}
\begin{proof}
Write $u=u_a$. By the chain rule we have 

\begin{equation}
\label{eq:u-monotone}
\frac{d}{dt} u(x_t) = |\bar{\nabla} u|^2(x_t) = 1-u(x_t)^2.
\end{equation}
 Integrating gives 
\[
u(x_t) = \tanh(t + \artanh(u(x_0))),
\]
and using the addition (and half-angle) formula for $\tanh$ gives the formula in terms of $s$. 
\end{proof}

\begin{proposition}
\label{prop:flow-ball}
Fix $a\in\mathbb{S}^n$ and let $\alpha \in [0,\pi]$ be such that $\cos\alpha = -\langle e_0,a\rangle$. Under the flow $\Psi^a_t$, the ball $B^n_R$ is mapped to a ball $\Psi^a_t(B^n_R) = B^n_{R_t}(o_t)$, where the radius satisfies 
\begin{equation}
\label{eq:moving-radius}
\cot R_t = \cot R\cosh t - \cos\alpha \csc R \sinh t = \frac{-2s\cos\alpha \csc R + (1+s^2)\cot R}{1-s^2}.
\end{equation}

Moreover, let $\bar{\eta}_t$ be the outer unit normal of $\pr B^n_{R_t}(o_t)$ and suppose $\cos R_t\neq 0$. Then 
\begin{equation}
\label{eq:conormal-V}
\langle\bar{\eta}_t, V_a\rangle = -\frac{d}{dt}\cot R_t + u\cot R_t.
\end{equation}
\end{proposition}
\begin{proof}
By Proposition \ref{prop:balls-to-balls}, we know that $B^n_R$ is mapped to some ball $B^n_{R_t}(o_t)$, where clearly $R_t, o_t$ depend smoothly on $t$. 

Choose coordinates on $\mathbb{R}^{n+1}$ so that $a= -(\cos\alpha)e_0 + (\sin\alpha)e_1$. Consider the two points $z^\pm = (\cos R)e_0 \pm (\sin R)e_1 \in \pr B^n_R$. Then $z^\pm$ are endpoints of a diameter of $B^n_R$, and by symmetry in the remaining variables, for each $t$ their images $z^\pm_t = \Psi^a_t(z^\pm)$ must be endpoints of a diameter of $B^n_{R_t}(o_t)$. In particular, the length of the diameter is related to the spherical distance by $\cos 2R_t = \langle z^+_t, z^-_t\rangle.$ From this we obtain the differential equation
\begin{equation}
\label{eq:radius-evo}
 \frac{d}{dt}\cos 2R_t = \langle V_a(z^-_t), z^+_t\rangle + \langle z^-_t, V_a(z^+_t)\rangle = (u_a(z^+_t) + u_a(z^-_t))(1-\cos 2R_t).
\end{equation}

Using Lemma \ref{lem:u-evo}, we find that 

\begin{equation}
 \frac{d}{dt}\cos 2R_t = (1-\cos 2R_t)(\tanh(t+\artanh(u_a(z^+))+ \tanh(t+\artanh(u_a(z^-)) ).
\end{equation}

Integrating, we have 
\[ -\log\left(\frac{1-\cos 2R_t}{1-\cos 2R}\right)  = \log \left( \frac{\cosh(t+ \artanh(u_a(z^+)))}{\cosh(\artanh(u_a(z^+)))}\right) + \log \left( \frac{\cosh(t+ \artanh(u_a(z^-)))}{\cosh(\artanh(u_a(z^-)))}\right) \]

Using the identity $\frac{\cosh(t+\artanh(u))}{\cosh(\artanh(u))} = \cosh t + u \sinh t$, and noting that the initial values are 
\begin{equation}
\label{eq:endpt-initial}
u_a(z^\pm)  = \langle z^\pm,a\rangle= -\cos R \cos \alpha \pm \sin R \sin \alpha = -\cos(\alpha\pm R),
\end{equation}

we find the desired equation

\[
\frac{1-\cos 2R}{1-\cos 2R_t} = ( \cosh t - \cos(\alpha +R)\sinh t)( \cosh t - \cos(\alpha -R) \sinh t)
\]

Using $\frac{1-\cos 2R}{2}=\csc^2 R = 1+\cot^2 R$ and substituting $s=\tanh \frac{t}{2}$, after some algebraic manipulations we find that 
\[
\cot^2 R_t =\left(\cot R\cosh t - \cos\alpha \csc R \sinh t\right)^2.
\]
We then take the square root of both sides, noting that (as $R_t$ varies smoothly) the signs are fixed by the signs at $s=t=0$. This yields (\ref{eq:moving-radius}), where we use the hyperbolic tangent half-angle formulae for the expression in terms of $s$.

In the remainder of the proof, we consider the case $R_t\neq \pit$. Then $z^+_t + z^-_t \neq 0$, and the centre $o_t$ may be found as
\begin{equation}
\label{eq:moving-centre} 
o_t = \frac{z^+_t + z^-_t}{|z^+_t + z^-_t|},
\end{equation}
and the radius is given by 

\begin{equation}
\label{eq:moving-centre-norm}
\cos R_t = \langle o_t, z^\pm_t\rangle = \frac{1+\langle z^+_t, z^-_t\rangle}{|z^+_t + z^-_t|} = \frac{1}{2}|z^+_t + z^-_t|. 
\end{equation}

The outer unit normal of $\pr B^n_{R_t}(o_t)$ is $\bar{\eta}_t = \bar{\nabla} \rho_{t}$, where $\rho_t$ is the spherical distance to $o_t$. Note that $\cos\rho_{t} = \langle x, o_t\rangle = u_{o_t}$, and therefore 
\begin{equation}
\label{eq:moving-normal}
\bar{\eta}_t =\bar{\nabla} \rho_{t}=  - \csc R_t\, V_{o_t}.
\end{equation}

In particular, we have that 
\begin{equation}
\label{eq:conormal-1}
\langle\bar{\eta}_t, V_a\rangle  = -  \csc R_t\, \langle o_t , V_a\rangle= - \csc R_t\, (\langle o_t, a\rangle -u_a \cos R_t) .
\end{equation}

where we have used that $\langle o_t,x\rangle = \cos R_t$ for $x\in \pr B^n_{R_t}(o_t)$. Using (\ref{eq:moving-centre}) and (\ref{eq:moving-centre-norm}) to find $\langle o_t,a\rangle$, we then have

\[ \langle\bar{\eta}_t, V_a\rangle= - \csc R_t\, \frac{(u_a(z^+_t) + u_a(z^-_t))}{2\cos R_t} + u \cot R_t.\]

But the numerator above is related to the evolution of $R_t$ by (\ref{eq:radius-evo}), and so
\[
\begin{split}
u_a(z^+_t) + u_a(z^-_t) &= \frac{1}{1-\cos 2 R_t} \frac{d}{dt} \cos 2R_t = \frac{1}{2\sin^2 R_t}(-2\sin 2R_t) \frac{d}{dt}R_t
\\&= -2\cot R_t \,\frac{d}{dt} R_t
\end{split}
\]

Substituting in to (\ref{eq:conormal-1}) we have that, as desired,

\[
\langle\bar{\eta}_t, V_a\rangle =  \csc^2 R_t\,  \frac{d}{dt}R_t + u \cot R_t = -\frac{d}{dt}\cot R_t + u\cot R_t.
\]

\end{proof}

\subsection{Conformal monotonicity in the free boundary setting}
\label{sec:conf-mono-image}

We now establish our monotonicity formulae for minimal submanifolds in $(\mathbb{S}^n , \pr B^n_R)$ under conformal translation. In this subsection, we will focus on the free boundary setting, that is, constant contact angle $\pit$. We state our results under the more general assumption of bounded mean curvature. 

Recall again that $\mathcal{E}^{R}(\Sigma) = |\Sigma| + \cot R\, |\pr\Sigma|$. 

\begin{theorem}
\label{thm:mono-sph}
Let $(\Sigma^k,\pr\Sigma)$ be an immersed submanifold in $(\mathbb{S}^n, \pr B^n_R)$ that contacts $\pr B^n_R$ at angle $\pit$ along $\pr \Sigma$, and suppose that $\sup_\Sigma |\mathbf{H}_\Sigma| \leq C_H$. 

Fix $a\in\mathbb{S}^n$ and let $\Psi^a_t$ be the flow of the conformal vector field $V_a$. Let $\Sigma_t = (\Psi^a_t)_*\Sigma$, which is a free boundary submanifold in $(\mathbb{S}^n, \pr B^n_{R_t}(o_t))$, where $B^n_{R_t}(o_t) = \Psi^a_t(B^n_R)$. 

Assume either: $k=2$; or $R=\pit$ and $a\perp e_0$. Then for $t\geq 0$ we have 
\begin{equation}
\label{eq:mono-sph}
\frac{d}{dt} \mathcal{E}^{R_t}(\Sigma_t) \leq C_H \mathcal{E}^{R_t}(\Sigma_t), 
\end{equation}

and hence $e^{-C_H t}\mathcal{E}^{R_t}(\Sigma_t)$ is non-increasing for $t\geq 0$.
\end{theorem}

\begin{proof}[Proof of Theorem \ref{thm:mono-sph}]
For convenience we write $u=u_a$. We write $\d\mu, \d\sigma$ for the respective measures on $\Sigma, \pr\Sigma$, and similarly $\d\mu_t, \d\sigma_t$ for the respective measures on $\Sigma_t, \pr\Sigma_t$.

By the first variation formula for submanifolds, we have that:

\begin{equation}
\label{eq:mono-sph-dB}
 \frac{d}{dt}|\pr \Sigma_t| = \int_{\pr\Sigma_t}\div_{\pr\Sigma_t} V_a \,\d\sigma_t = -(k-1) \int_{\pr\Sigma_t}u \,\d\sigma_t,
 \end{equation}

where we have used that $\bar{\nabla}^2 u =-u \bar{g}$ as in (\ref{eq:u-hess}). Similarly, by the divergence theorem,

\begin{equation}
\label{eq:mono-sph-dA-0}
 \frac{d}{dt}|\Sigma_t| = \int_{\Sigma_t} \div_{\Sigma_t} V_a \, \d\mu_t=  - \int_{\Sigma_t} \langle \mathbf{H}_{\Sigma_t}, V_a\rangle\,\d\mu_t+ \int_{\pr \Sigma_t} \langle \eta_t, V_a\rangle\,\d\mu_t.
 \end{equation}

As $\Psi^a_t$ is a conformal diffeomorphism, $\Sigma_t$ contacts $\pr B^n_{R_t}(o_t)$ at angle $\pit$ along $\pr\Sigma_t \subset \pr B^n_{R_t}(o_t)$. In particular, the unit conormal along $\pr\Sigma_t$ coincides with the outer unit normal of $\pr B^n_{R_t}(o_t)$. So by (\ref{eq:conormal-V}), we have 
\begin{equation}
\label{eq:mono-sph-conormal-V}
\langle \eta_t, V_a\rangle =\langle\bar{\eta}_t, V_a\rangle =   -\frac{d}{dt} \cot R_t + u\cot R_t.
\end{equation}

To deal with the mean curvature term, we proceed as in el Soufi-Ilias \cite{eSI}. For the remainder of the proof, we will write $\Psi=\Psi^a_t$ where convenient. 

Recall that by Lemma \ref{lem:conf-factor}, we have $\Psi^*\bar{g}= e^{2\psi} \bar{g}$, where $\psi$ satisfies
\begin{equation}
\label{eq:mono-sph-2}
\bar{\nabla} \psi = \frac{u-u\circ \Psi}{1-u^2}\bar{\nabla} u = \frac{u-u\circ \Psi}{1-u^2} V_a.
\end{equation}

As $\Sigma_t= \Psi_*\Sigma$, by the conformal transformation formula (\ref{eq:mc-conf}), we have 
\begin{equation}
\begin{split}
 \mathbf{H}_{\Sigma_t}(\Psi(x)) &= e^{-2\psi(x)} (d\Psi)_x\left( \mathbf{H}_{\Sigma}(x) - k \pi^{\perp}(\bar{\nabla} \psi(x))\right)
 \end{split}
  \end{equation}

Then by change of variables, we have 
 \[
 \begin{split}
  - \int_{\Sigma_t} &\langle \mathbf{H}_{\Sigma_t}(x), V_a(x)\rangle \,\d\mu_t(x) 
\\&= -\int_\Sigma \langle \mathbf{H}_{\Sigma_t}(\Psi(x)), V_a(\Psi(x))\rangle e^{k\psi(x)} \,\d\mu(x)
\\&= -\int_\Sigma \left\langle(d\Psi)_x\left( \mathbf{H}_{\Sigma}(x) - k \pi^{\perp}(\bar{\nabla} \psi(x))\right), (d\Psi)_x(V_a(x))\right\rangle e^{(k-2)\psi(x)} \,\d\mu(x) 
\\&= -\int_\Sigma \langle \mathbf{H}_{\Sigma}(x) - k \pi^{\perp}(\bar{\nabla} \psi(x)), V_a(x)\rangle e^{k\psi(x)} \,\d\mu(x) 
\\&= -\int_\Sigma \langle \mathbf{H}_{\Sigma} , V_a\rangle e^{k\psi}\d\mu + k\int_\Sigma \frac{u - u\circ \Psi}{1-u^2}|\pi^{\perp}(V_a)|^2  e^{k\psi}\,\d\mu ,
\end{split}
\]
where in the third equality we have used that $\Psi$ is a conformal map, and in the last step we have finally used (\ref{eq:mono-sph-2}). 

 As $\Psi^a_t$ was the gradient flow for $u$, for $t\geq 0$ we certainly have $u\circ \Psi \geq u$, that is, the last term above is nonpositive.
  
For the remaining mean curvature term, we use $|V_a| \leq \sqrt{1-u^2}\leq 1$, $|\mathbf{H}_\Sigma|\leq C_H$, so that 
\begin{equation}
\label{eq:mono-sph-H}
 - \int_{\Sigma_t} \langle \mathbf{H}_{\Sigma_t}(x), V_a(x)\rangle \,\d\mu_t(x) \leq -\int_\Sigma \langle \mathbf{H}_{\Sigma} , V_a\rangle e^{k\psi}\d\mu \leq C_H \int_\Sigma e^{k\psi} = C_H |\Sigma_t|.
 \end{equation}

Combining (\ref{eq:mono-sph-dB}), (\ref{eq:mono-sph-dA-0}), (\ref{eq:mono-sph-conormal-V}) and (\ref{eq:mono-sph-H}), we find that variation of area satisfies

\begin{equation}
\label{eq:mono-sph-dA}
 \frac{d}{dt}|\Sigma_t| \leq C_H|\Sigma_t| - |\pr\Sigma_t| \frac{d}{dt}\cot R_t - \frac{1}{k-1} \cot R_t \frac{d}{dt} |\pr\Sigma_t| . 
 \end{equation}

In the case $k=2$, it follows immediately from the definition of $\mathcal{E}^{R}$ that $\frac{d}{dt} \mathcal{E}^{R_t}(\Sigma_t) \leq C_H |\Sigma_t|\leq C_H \mathcal{E}^{R_t}(\Sigma_t)$.

In the case $R=\pit$, we also assume $a\perp e_0$. Then $\Psi^a_t$ preserves $B^n_\pit$ (cf. Proposition \ref{prop:flow-correspondence}), so $R_t\equiv \pit$ for all $t$. In particular, $\mathcal{E}^{R_t}(\Sigma_t) = |\Sigma_t|$, and (\ref{eq:mono-sph-dA}) gives $\frac{d}{dt}|\Sigma_t| \leq C_H |\Sigma_t|$. 
 
%
%
\end{proof}

Note that Theorem \ref{thm:conf-max-sph-intro-2} certainly follows from Theorem \ref{thm:mono-sph} by taking $C_H=0$. 

\begin{proof}[Proof of Theorem \ref{thm:conf-max-sph-intro}]
By Lemma \ref{lem:flow-relation}, up to isometries of $\mathbb{S}^n$, every element of $\Conf(\mathbb{S}^n)$ is given by $\Psi=\Psi^a_t$ for some $a\in\mathbb{S}^n$ and some $t = 2\artanh|\Psi(0)|\geq 0$. 

In the case $k=2$, then Theorem \ref{thm:mono-sph} in fact implies that $\mathcal{E}^{R_*}(\Sigma_*) \leq \mathcal{E}^{R}(\Sigma)$, where $\Psi$ maps $B^n_R$ to a geodesic ball of radius $R_*$. If $\Psi \in \Conf(B^n_R)$ then $R_*=R$, which completes the proof. 

In the case $R=\pit$, then as in Lemma \ref{lem:conf-Y}, the subgroup $\Conf(B^n_\pit)$ is realised, up to rotations, by $\Psi=\Psi^a_t$, where $a\perp e_0$, $t\geq 0$. So the result again follows immediately from Theorem \ref{thm:mono-sph}. 
\end{proof}

\begin{remark}
\label{rmk:higher-dim}
For higher dimensions $k>2$, an issue seems to arise from the mismatch of coefficients of $u$ in the boundary integrands of (\ref{eq:mono-sph-dB}) and (\ref{eq:mono-sph-conormal-V}). One may recover a monotonicity formula by replacing the boundary term $\cot R_t \, |\pr\Sigma_t|$ by the boundary integral \[\mathcal{B}(t) = \int_{\pr\Sigma_t} F(t,u_a(x))\d\sigma_t(x),\] where (setting $c(t)=\cos R_t$; note that $c''(t)=c(t)$ by (\ref{eq:moving-radius}))
\[ F(t,u) := \beta(u)\frac{c(t)+c'(t)}{2} + \bar{\beta}(u)\frac{c(t)-c'(t)}{2} = \frac{\beta(u) + \bar{\beta}(u)}{2} c(t) + \frac{\beta(u) - \bar{\beta}(u)}{2} c'(t),\] 
and $\beta(u),\bar{\beta}(u)$ solve the ODEs
\[ \beta -(k-1)u\beta + (1-u^2)\beta' = 1-u,\qquad \bar{\beta} +(k-1)u\bar{\beta} -(1-u^2)\bar{\beta}' =1+u.\]
(Such solutions may be given in terms of the incomplete beta function.) However, it is unclear to us whether this quantity yields any natural consequences. (See also Remark \ref{rmk:higher-dim-euc}.) 
\end{remark}

\subsection{Conformal monotonicity in the capillary setting}
\label{sec:conf-mono-cap}

In this section, we investigate conformal monotonicity formulae for capillary hypersurfaces $(\Sigma^{n-1},\pr\Sigma) \looparrowright (\mathbb{S}^n, S)$, where $S=\pr B^n_R$. These formulae hold and have satisfying interpretations for somewhat weaker notions of capillary hypersurface as outlined in Definition \ref{def:capillary}. 

We first consider the case $n-1=2$, for which we assume that $\Sigma$ has an associated wet surface, namely an \textit{immersed} $(n-1)$-manifold with boundary $S^-\looparrowright S$ so that $\pr S^- = \pr \Sigma$, and the outer conormal $\bar{\nu}$ of $\pr\Sigma$ in $S^-$ satisfies 
\begin{equation}\tag{$\ddagger$}\sin\gamma \, \bar{\eta} = \eta + \cos\gamma \, \bar{\nu},\end{equation}

Recall that

\[
\mathcal{E}^{R,\gamma}(\Sigma, S^-) = |\Sigma| +\cos\gamma \csc^2 R\, |S^-| + \sin\gamma \cot R\, |\pr\Sigma|.
\]

For simplicity, in this subsection we will state results for which the base $\Sigma$ is minimal. We begin with the $k=2$ case of Theorem \ref{thm:conf-max-sph-intro-cap}, restated precisely for capillary minimal surfaces as in Definition \ref{def:capillary}:

\begin{theorem}
\label{thm:mono-sph-capillary}
Let $(\Sigma^{2},\pr\Sigma)$ be a capillary minimal surface in $(\mathbb{S}^3 , \pr B_R)$ with contact angle $\gamma \in (0,\pit]$, and wet surface $S^-$. 

Fix $a\in\mathbb{S}^3$ and let $\Psi^a_t$ be the flow of the conformal vector field $V_a$. Let $\Sigma_t = (\Psi^a_t)_*\Sigma$, which is a capillary surface with contact angle $\gamma$ in $(B^n_{R_t}(o_t), \pr B^n_{R_t}(o_t))$, where $B^n_{R_t}(o_t) = \Psi^a_t(B^n_R)$. Similarly let $S^-_t = (\Psi^a_t)_* S^-$. 

Then $\mathcal{E}^{R_t,\gamma}(\Sigma_t , S^-_t)$ is non-increasing for $t\geq 0$.
\end{theorem}

\begin{proof}[Proof of Theorem \ref{thm:mono-sph-capillary}]
As in the proof of Theorem \ref{thm:mono-sph}, for convenience we write $u=u_a$, and $\d\mu, \d\mu_t$ for interior measures and $\d\sigma, \d\sigma_t$ for boundary measures. 

Again let $\bar{\eta}_t$ be the outward unit normal on $\pr B_{R_t}(o_t)$, and by (\ref{eq:conormal-V}) we have
\begin{equation}
\label{eq:mono-sph-cap-1}
\langle\bar{\eta}_t, V_a\rangle = -\frac{d}{dt}\cot R_t + u\cot R_t.
\end{equation}

As $S^-_t$ is an immersed surface in $\pr B_{R_t}(o_t)$, we note that $\mathbf{H}_{S^-_t}= -2\cot R_t \, \bar{\eta}_t$. Then by the divergence theorem, we have 

\[
\begin{split}
 \frac{d}{dt}|S^-_t| &= \int_{S^-_t} \div_{S^-_t}V_a = -2 \int_{S^-_t} u \d\mu_t
\\& = -\int_{S^-_t} \langle \mathbf{H}_{S^-_t}, V_a\rangle\d\mu_t  + \int_{\pr\Sigma_t} \langle \bar{\nu}_t, V_a\rangle \d\sigma_t
\\&= 2\cot R_t  \left(-|S^-_t|\frac{d}{dt}\cot R_t + \cot R_t\int_{S^-_t} u\d\mu_t\right) + \int_{\pr\Sigma_t} \langle \bar{\nu}_t, V_a\rangle \d\sigma_t
 \end{split}
\]

Comparing the last line to the first, we can eliminate the $u$ term. As $1+\cot^2 = \csc^2$, we find that

\[
\csc^2 R_t\, \frac{d}{dt}|S^-_t| = - |S^-_t| \frac{d}{dt} \csc^2 R_t + \int_{\pr\Sigma_t} \langle \bar{\nu}_t, V_a\rangle\d\sigma_t. 
\]

Adding this to the divergence formula on $\Sigma_t$, we have

\begin{equation}
\label{eq:mono-sph-cap-2}
\frac{d}{dt}(|\Sigma_t| + \cos\gamma \csc^2 R_t \,|S^-_t|) =  - \int_{\Sigma_t} \langle \mathbf{H}_{\Sigma_t}, V_a\rangle\,\d\mu_t +  \int_{\pr \Sigma_t} \langle \eta_t +\cos\gamma\, \bar{\nu}_t, V_a\rangle\,\d\mu_t.
\end{equation}

As in the proof of Theorem \ref{thm:mono-sph}, and using that $\mathbf{H}_\Sigma=0$ on the base $\Sigma$, we have that $- \int_{\Sigma_t} \langle \mathbf{H}_{\Sigma_t}, V_a\rangle\,\d\mu_t \leq 0$. Also, by (\ref{eq:angle-balance}) and (\ref{eq:mono-sph-cap-1}), we have

\[\begin{split}
\int_{\pr \Sigma_t} \langle \eta_t +\cos\gamma\, \bar{\nu}_t, V_a\rangle\,\d\mu_t &= \sin\gamma \int_{\pr\Sigma_t} \langle \bar{\eta}_t, V_a\rangle \d\sigma_t
\\&= \sin\gamma \left( -|\pr\Sigma_t| \frac{d}{dt}\cot R_t + \cot R_t \int_{\pr\Sigma_t} u \d\sigma_t\right)
\end{split}\]

But for the boundary variation, it still holds that (as $n-1=2$)

\[ \frac{d}{dt}|\pr \Sigma_t| = \int_{\pr\Sigma_t}\div_{\pr\Sigma_t} V_a \,\d\sigma_t = - \int_{\pr\Sigma_t}u \,\d\sigma_t.\]

Thus 

\[ \frac{d}{dt}(|\Sigma_t| + \cos\gamma \csc^2 R_t \,|S^-_t|) \leq \sin\gamma \left( - |\pr\Sigma_t| \frac{d}{dt} \cot R_t - \cot R_t\, \frac{d}{dt} |\pr\Sigma_t|\right),\]

which implies $\frac{d}{dt}\mathcal{E}^{R_t,\gamma}\leq 0$ by its definition. 
\end{proof}

We now turn to the hemisphere case. In this case, for $a\perp e_0$ the flow $\Psi^a_t$ preserves the hemisphere, so as in Definition \ref{def:capillary} we may define a local wetting energy to be the area swept out by $\pr\Sigma$ (in $\pr B^n_\pit$). We then have the following monotonicity, which is the $R=\pit$ case of Theorem \ref{thm:conf-max-sph-intro-cap}, restated more generally for minimal hypersurfaces with constant contact angle $\gamma$. (Note that if $\Sigma$ is a \textit{capillary} minimal hypersurface as in Definition \ref{def:capillary}, that is, $\Sigma$ admits a wet surface $S^-\looparrowright \pr B^n_\pit$, then the local wetting energy is precisely $W(t) = |S^-_t| - |S^-|$.)

\begin{theorem}
\label{thm:mono-sph-capillary-hemi}
Let $(\Sigma^{n-1},\pr\Sigma)$ be a minimal hypersurface in $(\mathbb{S}^n , \pr B^n_\pit)$ which contacts $S=\pr B^n_\pit$ at angle $\gamma \in (0,\pit]$. 

Fix $a\in\mathbb{S}^n$ so that $a\perp e_0$, and let $\Psi^a_t$ be the flow of the conformal vector field $V_a$. Let $\Sigma_t = (\Psi^a_t)_*\Sigma$, which also contacts $\pr B^n_\pit$ at angle $\gamma \in (0,\pit]$. Let $W(t)=\int_{\pr\Sigma\times [0,t]} X^* \d\mu_S$ be the local wetting energy for the variation $X(\cdot,t) = \Psi^a_t$, as in Definition \ref{def:capillary}. 

Then $|\Sigma_t| + \cos\gamma\, W(t)$ is non-increasing for $t\geq 0$.
\end{theorem}

\begin{proof}
If $\gamma=\pit$ then the result follows from the free boundary monotonicity Theorem \ref{thm:mono-sph}. So we may assume $\gamma \in (0,\pit)$. 

Recall that there is still a well-defined unit normal $\bar{\nu}$ of $\pr\Sigma$ in $S$, satisfying 
\[\sin\gamma \, \bar{\eta} = \eta + \cos\gamma \, \bar{\nu},\] 
and that the wetting energy $W(t)$ is defined with respect to the induced orientation of $\pr\Sigma$. 
In particular,
\[
\begin{split}
 W(t) &= \int_{\pr\Sigma\times [0,t]} X^* \d\mu_S \\&=\int_0^t \int_{\pr\Sigma_\tau} \langle \bar{\nu}_\tau,V_a\rangle \d\sigma_\tau \d \tau. 
 \end{split}
 \]

With this in hand, arguing as in the proof of Theorem \ref{thm:mono-sph-capillary}, we have 
\[
\begin{split}
\frac{d}{dt}(|\Sigma_t| + \cos\gamma \, W(t)) &=  - \int_{\Sigma_t} \langle \mathbf{H}_{\Sigma_t}, V_a\rangle\,\d\mu_t +  \int_{\pr \Sigma_t} \langle \eta_t +\cos\gamma\, \bar{\nu}_t, V_a\rangle\,\d\mu_t
\\&\leq \int_{\pr \Sigma_t} \langle \bar{\eta}_t, V_a\rangle\,\d\mu_t. 
\end{split}
\]

In this case, as $R_t\equiv \pit$, by (\ref{eq:conormal-V}) we have
\[
\langle\bar{\eta}_t, V_a\rangle = -\frac{d}{dt}\cot R_t + u\cot R_t =0,
\]
which completes the proof. 

\end{proof}

\subsection{Conformal blowup}
\label{sec:blowup}

In this subsection, we briefly record corollaries of our conformal maximisation results under conformal blowup - in essence, by taking a limit as $t\to \infty$. This blowup proceeds as in \cite{LY82} and \cite[Remark 5.2]{FS11} (see also \cite[Proposition 3.2]{RS97}, which essentially performs this process for capillary surfaces in the Euclidean ball).

First, we consider capillary minimal hypersurfaces in the hemisphere. 

\begin{corollary}
Let $(\Sigma^{n-1},\pr\Sigma)$ be a capillary minimal hypersurface in $(\mathbb{S}^n , \pr B_\pit)$ with contact angle $\gamma \in (0,\pit]$, and wet surface $S^-$. We have \[\mathcal{E}^{\pit,\gamma}(\Sigma, S^-) \geq \frac{1+\cos\gamma}{2}|\mathbb{S}^n|.\] 
\end{corollary}
\begin{proof}
Fix $y\in \pr\Sigma$, let $a$ be the antipodal point of (the image of) $y$ and consider the one-parameter family of conformal translations $\Phi_{sa} \in \Conf(B_\pit)$. This family fixes the points $a$ and $-a$, and for any other point $x\in \mathbb{S}^n$ the trajectory $\Phi_{sa}(x)$ is a geodesic connecting $-a = \lim_{s\to -1} \Phi_{sa}(x)$ to $a= \lim_{s\to1} \Phi_{sa}(x)$. 

In particular, as $s\to 1$ the image $(\Psi_{sa})_*\Sigma$ converges to a totally geodesic half-equator $\Sigma'$, whose tangent half-plane $T_y \Sigma' = T_y\Sigma$ (possibly with multiplicity). Similarly, the wet surface $(\Psi_{sa})_* S^-$ converges to half of the barrier $\pr B_\pit$. Both limiting surfaces have area $\frac{1}{2}|\mathbb{S}^n|$, so the result follows from Corollary \ref{cor:conf-max-sph-cap-intro}.
\end{proof}

Note that the totally geodesic half-equator (with the correct contact angle) is indeed a capillary minimal hypersurface, so the above bound is sharp. For capillary surfaces in other spherical caps, it seems that the analogous blowup process in $B_R$ should not give a sharp bound (similar to the Euclidean ball, wherein the blowup may not be minimal), so we leave the following discussion:

\begin{remark}
Consider a capillary minimal surface $(\Sigma^{2},\pr\Sigma)\looparrowright(\mathbb{S}^3 , \pr B_R)$ with contact angle $\gamma \in (0,\pit]$, and wet surface $S^-$. Considering $y\in \pr\Sigma$ as before, we may again flow by $\Phi_{-sy}$. Since $y\in \pr B_R$, this will flow $B_R$ to a (rotated) hemisphere, and the blowup should again be a half-equator. This would give $\mathcal{E}^{R,\gamma}(\Sigma) \geq 2\pi(1+\cos\gamma)$ in the limit. 

One may also consider a conformal blowup by transformations in $\Conf(B_R)$. In this way, the conformal blowup should be the image in $B_R$ of the half-equator in $B_\pit$, under the map $\Phi_{s_R e_0}$. It seems that this gives a weaker bound than above. 
\end{remark}

\subsection{A global view, and the Euclidean limit as $R\to 0$}
\label{sec:euclid-limit}
In this subsection, we sketch some remarks about what information may be obtained for minimal submanifolds in the Euclidean ball $\mathbb{B}^n$, by taking suitable limits as $R\to 0$. We will focus on free boundary submanifolds, for which we recall the definition
\[
\mathcal{E}^{R}(\Sigma) = |\Sigma|  +  \cot R\, |\pr\Sigma|.
\]

To begin our discussion, we first give a unified overview of some monotonicity results we have proven so far. Consider a free boundary submanifold $(\Sigma_\pit^k, \pr \Sigma_\pit)$ in the hemisphere $(B^n_\pit, \pr B^\pit)$, and for $R\in(0,\pi)$, $y\in \mathcal{S}_0 = \mathbb{B}^n$ define the submanifolds in $(B^n_R, \pr B^n_R)$ given by
\[\Sigma_R = (\Phi^{n+1}_{s_R e_0})_* \Sigma_\pit, \qquad \Sigma^y_R = (\Phi^{n+1}_{s_R e_0})_* (\Phi^{n+1}_y)_* \Sigma_\pit.\]

(Recall that $\Phi^{n+1}_y$ will preserve $B^n_\pit$, and $\Phi^{n+1}_{s_R e_0}$ gives a conformal equivalence $B^n_\pit \to B^n_R$.) 

For such $y\in \mathbb{B}^n$ we may also consider the corresponding conformal transformations in $\Conf(B^n_R)\subset \Conf(\mathbb{B}^{n+1})$ given by \[\Psi := \Phi^{n+1}_{s_R e_0} \circ \Phi^{n+1}_y \circ \Phi^{n+1}_{-s_R e_0},\]
so that \[ \Sigma^y_R  = \Psi_* \Sigma_R.\]

In this way, $\{\Sigma^y_R\}$ comprises all conformal images (up to rotation) of $\Sigma_R$ (which are free boundary in $B^n_R$). This is schematically depicted in Figure \ref{fig:euc}, using (see Appendix \ref{sec:conf-mono-overflow}) \[Y=\frac{|y|^2(\cos R) e_0 + (\sin R) y}{\sin^2 R + |y|^2 \cos^2 R}\in\mathbb{B}^{n+1}\] to parametrise all $\Sigma^y_R$ at once. In dimension $k=2$, if $\Sigma_\pit$ is also minimal, then we showed that the quantity $\mathcal{E}^R(\Sigma^y_R)$ is monotone along rays from the origin (in any direction). Similarly, if $\Sigma_{R_0}$ is minimal, then we will have that $\mathcal{E}^R(\Sigma^y_R)$ is monotone along circular arcs beginning at $(\cos R_0) e_0$ (which intersects $\pr\mathbb{B}^{n+1}$ orthogonally; indeed, these are the images of ray from the origin under the conformal transformation $\Phi^{n+1}_{s_R e_0}$). As in Figure \ref{fig:euc}, for any $R,y$ there is always such a path connecting $\Sigma_R$ to $\Sigma^y_R$. 

\begin{figure}
\includegraphics[scale=0.6]{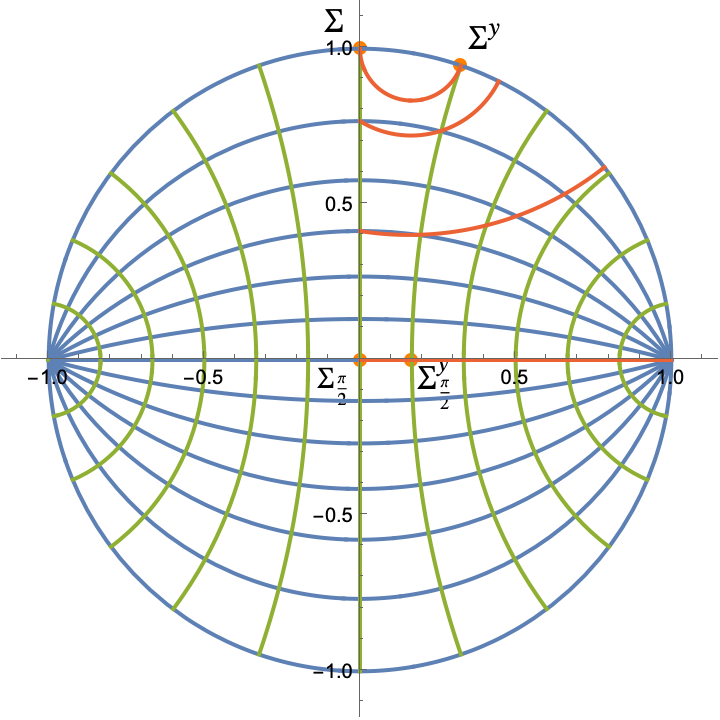}
\caption{Schematic diagram, wherein each point $Y\in \mathbb{B}^{n+1}$ corresponds to one of the surfaces $\Sigma^y_R$. The horizontal `meridians' correspond to fixed values of $R$. The vertical `parallels' correspond to fixed values of $y$, so that points on the (upper) boundary will correspond to $\Sigma^y$. Additionally, four paths are depicted that begin at $\Sigma^0_R$ and pass through $\Sigma^y_R$, for $R=\pit, \frac{\pi}{4}, \frac{\pi}{6},0$ respectively. These paths are hyperbolic geodesics if $\mathbb{B}^{n+1}$ is treated as the Poincar\'{e} model. 
}
\label{fig:euc}
\end{figure}

In this scheme, points on the boundary $\pr \mathbb{B}^{n+1}$ represent limits of $\Sigma^y_R$ as $R\to 0$, which correspond to submanifolds $\Sigma^y$ of the Euclidean ball $(\mathbb{B}^n,\pr\mathbb{B}^n)$. To make this precise, we recall some notation: $\delta$ refers to the Euclidean metric, $\bar{g}$ the unit round metric on $\mathbb{S}^n$; we defined $s_R = \tan \frac{\pit-R}{2}$. Also recall that $\Xi:\mathbb{B}^n\to B^n_\pit$ is the (inverse) stereographic projection, and so the composition $\Phi_{s_R e_0}\circ \Xi$ defines a conformal equivalence $\Xi_R:\mathbb{B}^n \to B^n_R$. As in Lemma \ref{lem:h-R}, we can compute the pullback factor $\Xi_R^*\bar{g} = f_R^2 \delta$, where $f_R:\mathbb{B}^n\to\mathbb{R}$ is defined by \[f_R(z)=\frac{2\sin R}{1+|z|^2 + (1-|z|^2)\cos R}.\]

Given a free boundary minimal submanifold $(\Sigma^k,\pr\Sigma)$ in $(\mathbb{B}^n,\pr\mathbb{B}^n)$, for each $R$ we define the (inversely) projected submanifold \[\Sigma_\pit = \Xi_*\Sigma,\] so that $\Sigma_R = (\Xi_R)_*\Sigma$ as above. We will be interested in the orbit of $\Sigma$ under the action of $\Conf(\mathbb{B}^n)$, so for $y\in \mathbb{B}^n$ we consider \[\Sigma^y = (\Phi^n_y)_*\Sigma;\] 
note that (by Lemma \ref{lem:conf-Y}, for instance) we also have $\Sigma^y_R= (\Xi_R)_* \Sigma^y$. In this way, as $f_R \sim \sin R$ for small $R$, after rescaling by $\frac{1}{\sin R}$ the submanifolds $\Sigma^y_R$ will converge smoothly to $\Sigma^y$. The $\Sigma_R$ will not be minimal, but instead have small mean curvature as $R\to 0$. Indeed, using the conformal transformation formula (\ref{eq:mc-conf}), one may verify that \[|\mathbf{H}_{\Sigma_R}| \leq k \frac{1-\cos R}{\sin R}.\] 

Henceforth we will restrict to the case $k=2$. In the scheme depicted in Figure \ref{fig:euc}, for any fixed $y\in \mathbb{B}^n$ we will now have an \textit{almost}-monotone path connecting $\Sigma_R$ to $\Sigma^y_R$, and we expect a genuine monotonicity along the limiting path, which is a circular arc orthogonal to $\pr\mathbb{B}^n$ at the points corresponding to $\Sigma= \Sigma^0$ and $\Sigma^y$. Here we will not explicitly prove this limiting monotonicity, but instead use the scheme above to recover the boundary maximisation result prove by Fraser-Schoen \cite{FS11}, namely that for any $y\in\mathbb{B}^n$, 
\begin{equation}
\label{eq:FS-result} |\pr\Sigma^y| \leq |\pr\Sigma|.
\end{equation}

By Lemma \ref{lem:conf-Y}, (up to ambient rotation) $\Sigma^y_R=\Psi_* \Sigma_R$ may be reached by a conformal flow of $\Sigma_R$ for time $t= 2\artanh |Y|$, where $Y=\Psi(0)=\frac{|y|^2(\cos R) e_0 + (\sin R) y}{\sin^2 R + |y|^2 \cos^2 R}$ as above. (The interested reader may consult Appendix \ref{sec:conf-mono-overflow}, in particular Proposition \ref{prop:flow-correspondence}, for the flow direction.)

By the monotonicity Theorem \ref{thm:mono-sph} we have that 

\[\mathcal{E}^{R}(\Sigma^y_R) \leq \exp\left(2 \frac{1-\cos R}{\sin R} \artanh|Y|\right) \mathcal{E}^{R}(\Sigma_R).\]

Since $\lim_{R\to 0}\frac{f_R}{\sin R}\to1$, taking into account the scaling properties of area and length (see Corollary \ref{cor:area-limit}) we have $\mathcal{E}^{R}(\Sigma_R) \to |\pr\Sigma|$, and similarly $ \mathcal{E}^{R}(\Sigma^y_R) \to |\pr\Sigma^y|$. Thus we will recover the Euclidean result (\ref{eq:FS-result}) so long as 

\begin{equation}\label{eq:euc-limit} \lim_{R\to 0} \frac{1-\cos R}{\sin R}\artanh|Y| =0.\end{equation}

But \[|Y| = \frac{|y|}{\sqrt{\sin^2 R + |y|^2 \cos^2 R}},\] and using $\artanh s= \frac{1}{2} \log \frac{1+s}{1-s}$, it follows that $\artanh\left(\frac{|y|}{\sqrt{\sin^2 R + |y|^2 \cos^2 R}}\right)  \sim -\log R$. Then (\ref{eq:euc-limit}) follows as $\frac{1-\cos R}{\sin R} \sim \frac{1}{2}R$ and $R\log R\to 0$. 

\begin{remark}
We have seen above that, when $k=2$, the conformal maximisation in spherical caps (with mean curvature term) indeed recovers the Fraser-Schoen conformal maximisation in Euclidean balls. However, it seems that we may \textit{not} obtain a limiting monotonicity formula for the Euclidean setting, at least by the naive perturbation method above. Indeed, as depicted in Figure \ref{fig:euc}, the natural limiting monotonicity formula is along a circular arc orthogonal to to the (ideal) boundary - that is, naturally involving conformal images \textit{in spherical caps}, and only recovering information on the Euclidean setting at its endpoints.
\end{remark}

\begin{remark}
\label{rmk:higher-dim-euc}
For higher dimensions $k>2$, one may consider a limiting monotonicity corresponding to the monotone quantity discussed in Remark \ref{rmk:higher-dim}. However, due to the form of that monotone quantity, it seems that on the Euclidean limits $\Sigma$ and $\Sigma^y$, the rescaled values may not converge to the boundary areas $|\pr\Sigma|, |\pr\Sigma^y|$, but instead to boundary integrals that depend on the coordinate function $\langle \cdot, e_1\rangle$ (on $\mathbb{B}^n$), where $e_1 = \frac{y}{|y|}$. These boundary integrals do not appear to give a simple relation between intrinsic quantities on $(\Sigma,\pr\Sigma)$ and $(\Sigma^y,\pr\Sigma^y)$. 
\end{remark}

\section{Quadratic forms and index problems}
\label{sec:quadratic}

The purpose of this section is to concretely review the index of quadratic forms on compact manifolds with boundary, as well as associated Robin, Dirichlet and (modified) Steklov eigenvalue problems. The index may be calculated either just by counting Robin eigenvalues, or counting Dirichlet and Steklov eigenvalues (Proposition \ref{prop:index-sum}). We will also record some important variational properties of those eigenvalues in Lemma \ref{lem:comparison}. We emphasise that these notions are stated for general forms $Q$, so may certainly differ from more familiar eigenvalue problems - depending upon the choice of $Q$. 

\subsection{Forms on Hilbert spaces}

For any quadratic form $Q$ defined on a real vector space $\mathcal{V}$, there is an associated symmetric bilinear form (by polarisation) that we continue to denote $Q$; explicitly $Q(u,v) = \frac{1}{2}(Q(u+v)-Q(u)-Q(v))$. 

\begin{definition}
The index $\ind(Q)$ is the dimension of a maximal subspace of $\mathcal{V}$ on which $Q$ is negative definite. The kernel is $\ker(Q)= \{ u\in \mathcal{V} | \forall v\in \mathcal{V}: Q(u,v)=0\}$. 

We will write $\ind_0(Q)$ for the dimension of a maximal subspace of $\mathcal{V}$ on which $Q$ is negative semidefinite; we have $\ind_0(Q) = \ind(Q) + \nul(Q)$, where $\nul(Q) = \dim \ker(Q)$.
\end{definition}

\subsection{Forms on compact manifolds}

Let $(\Sigma, \pr\Sigma)$ be a smooth compact Riemannian manifold with boundary. Let $\eta$ be the outer conormal of $\pr\Sigma$ in $\Sigma$. 


\begin{definition}
The index $\ind(Q)$ is the dimension of a maximal subspace of $C^\infty(\bar{\Sigma})$ on which $Q$ is negative definite. We will write $\ind_0(Q)$ for the dimension of a maximal subspace of $C^\infty(\bar{\Sigma})$ on which $Q$ is negative semidefinite.

If $C^\infty(\bar{\Sigma})$ is dense in a Hilbert space $\mathcal{H}$, and $Q$ extends continuously to $\mathcal{H}$, then (see Tran-Zhou \cite[Lemma 2.7]{TZ23}) these definitions agree with the definitions via $\mathcal{H}$. 
\end{definition}

We are most interested in $H^1(\Sigma)$, which contains $C^\infty(\bar{\Sigma})$ as a dense subspace. We also have the natural inclusion map $j: H^1(\Sigma) \to L^2(\Sigma)$ and the trace map $\tau: H^1(\Sigma) \to L^2(\pr\Sigma)$, both of which are compact operators (recall the latter factors through $H^{1/2}(\pr\Sigma)$). As usual we define $H^1_0(\Sigma) = \ker \tau$ to be the subspace of $H^1(\Sigma)$ with vanishing boundary values. 

For smooth functions $p \in C^\infty(\bar{\Sigma})$ and $q\in C^\infty(\pr\Sigma)$, we consider the quadratic form defined on $H^1(\Sigma)$ by

\[Q_{p,q}(u) = \int_\Sigma (|\nabla u|^2 -pu^2)  - \int_{\pr \Sigma} qu^2.\]

If $u,v\in C^\infty(\bar{\Sigma})$, then integration by parts gives

\[Q_{p,q}(u,v) = \int_\Sigma (\langle \nabla u, \nabla v\rangle -puv)  - \int_{\pr \Sigma} quv =  -\int_\Sigma u(\Lap+p)v  + \int_{\pr \Sigma} u(\pr_\eta-q)v .\]

When $p,q$ are contextually unambiguous we will drop the subscripts, writing $Q=Q_{p,q}$.

We say that $u \in H^1(\Sigma)$ is a weak solution of $(\Lap+p)=0$ if 
\[Q(u,\phi)=0\text{ for all }\phi \in H^1_0(\Sigma).\] We will denote the space of weak solutions by $V(Q)$; note of course that it does not depend on the boundary term $q$. By elliptic regularity, any $u\in V(Q)$ is smooth on the interior of $\Sigma$. We will also refer to such weak solutions as $(\Lap+p)$-harmonic.

As $p,q$ are uniformly bounded, we have the following `almost-coercivity' estimate:

\begin{lemma}
\label{lem:coercive}
There exist $\alpha,C$ so that
\[\|u\|_{H^1(\Sigma)}^2-\alpha(\|u\|_{L^2(\Sigma)}^2 + \|u\|_{L^2(\pr\Sigma)}^2) \leq Q_{p,q}(u) \leq C\|u\|_{H^1(\Sigma)}^2\]
\end{lemma}

In the language of \cite{AEKS}, Lemma \ref{lem:coercive} says that $Q_{p,q}$ is `compactly elliptic' with respect to the compact operator $j\oplus \tau: H^1(\Sigma)\to L^2(\Sigma)\oplus L^2(\pr\Sigma)$.

\subsection{Robin spectrum}

The Robin eigenvalue problem associated to $Q=Q_{p,q}$ is stated as
\begin{equation}
\begin{cases}
(\Lap+p)u= - \lambda u &, \text{ on } \Sigma, \\
(\pr_\eta -q)u=0 &, \text{ on } \pr\Sigma. 
\end{cases}
\end{equation}

By standard spectral theory, $H^1(\Sigma)$ admits an $L^2(\Sigma)$-orthonormal basis of eigenfunctions $\varphi^{\mathrm{R}}_k$, $k=0,1,2,\cdots$, and the corresponding Robin eigenvalues $\lambda^{\mathrm{R}}_k$ are discrete, accumulating only at $\infty$. (See \cite{CFP14, MM25}.) Moreover, we have the variational characterisation 

\[ \lambda^{\mathrm{R}}_k 
= \inf_{\substack{u\in H^1(\Sigma) \\ \forall j<k, \langle u,\varphi^{\mathrm{R}}_j\rangle_{L^2(\Sigma)}=0} } \frac{Q(u,u)}{\|u\|_{L^2(\Sigma)}^2} .\] 

By the variational characterisation, it is easy to see that $\ind(Q)$ (resp. $\ind_0(Q)$) is always given by the number of negative (resp. nonnegative) Robin eigenvalues. However, for analysis it is convenient to quantify the interior and boundary influence separately. This will be achieved by the following Dirichlet and Steklov-type problems.

\subsection{Dirichlet spectrum}

The Dirichlet eigenvalue problem associated to $Q=Q_{p,q}$ is stated as
\begin{equation}
\label{eq:Dir-spec}
\begin{cases}
(\Lap+p)u= - \lambda u &, \text{ on } \Sigma, \\
u|_{\pr\Sigma}=0 &, \text{ on } \pr\Sigma. 
\end{cases}
\end{equation}

Consider the restriction of $Q$ to $H^1_0(\Sigma) = \ker \tau$. For $u\in H^1_0(\Sigma)$, the form acts as 
\[Q(u) = \int_{\Sigma} (|\nabla u|^2 -pu^2).\]

By standard spectral theory, there is an $L^2(\Sigma)$-orthonormal basis of $H^1_0(\Sigma)$ given by smooth eigenfunctions $\varphi^{\mathrm{D}}_k \in C^\infty(\bar{\Sigma})$, $k=0,1,2,\cdots$, and the corresponding Dirichlet eigenvalues $\lambda^{\mathrm{D}}_k$ are discrete, accumulating only at $\infty$. Moreover, we have the variational characterisation

\begin{equation}
\label{eq:dir-variational}
\lambda^{\mathrm{D}}_k 
=  \inf_{\substack{u\in H^1_0(\Sigma) \\ \forall j<k, \langle u,\varphi^{\mathrm{D}}_j\rangle_{L^2(\Sigma)}=0} } \frac{Q(u,u)}{\|u\|_{L^2(\Sigma)}^2}. 
\end{equation}

If $\Sigma$ is connected, by a standard argument (see also Lemma \ref{lem:steklov-bottom} below for the Steklov setting) using the variational characterisation and the regularity of eigenfunctions, the first eigenvalue $\lambda^{\mathrm{D}}_0$ is simple and has an eigenfunction $\varphi^{\mathrm{D}}_0$ that is positive on the interior of $\Sigma$. 

As the Dirichlet eigenfunctions are smooth (up to the boundary), an integration by parts gives the action 
\begin{equation}
\label{eq:dir-Q-action}
Q(\varphi^{\mathrm{D}}_i , \phi) = -\int_\Sigma \phi (\Lap+p)\varphi^{\mathrm{D}}_i + \int_{\pr\Sigma} \phi(\pr_\eta - q) \varphi^{\mathrm{D}}_i 
= \lambda^{\mathrm{D}}_i\int_\Sigma \phi\varphi^{\mathrm{D}}_i + \int_{\pr\Sigma} \phi\pr_\eta \varphi^{\mathrm{D}}_i,
\end{equation}
for any $\phi\in H^1(\Sigma)$. 

Let $\mathcal{W}_0$ denote the 0-Dirichlet eigenspace. (Note as per (\ref{eq:dir-Q-action}) that elements of $\mathcal{W}_0$ may \textit{not} be 0-eigenfunctions of the form $Q$ on $H^1(\Sigma)$ as their Neumann data should generally not vanish; indeed they should contribute to the \textit{index} rather than the nullity.)

%
%

\subsection{(Modified) Steklov spectrum}

We consider the Steklov-type eigenvalue problem
\begin{equation}
\label{eq:steklov-system}
\begin{cases}
(\Lap+p)u= 0 &, \text{ on } \Sigma, \\
(\pr_\eta - q)u = \lambda  u&, \text{ on } \pr\Sigma. 
\end{cases}
\end{equation}

Formally, this eigenvalue problem corresponds to the spectrum of a `Dirichlet-to-Robin' map as follows: Given a function $u$ on $\pr\Sigma$, we seek a $(\Lap+p)$-harmonic function $\hat{u}$ with $\hat{u}|_{\pr\Sigma}=u$, whence the image of $u$ is given by $(\pr_\eta - q)\hat{u}$. By the Fredholm alternative, this process may be complicated by the presence of Dirichlet kernel. 

To rigorously define the Dirichlet-to-Robin map and its spectrum, one may use (for instance) the `hidden compactness' theory of Arendt et. al. \cite{AEKS} (see also \cite{AM07, AM12} and \cite{Tr20, TZ23}). Rather than repeating the full details of their theory, we will describe the essential notions to define the Dirichlet-to-Robin map, and how they correspond to key notions in \cite{AEKS}. 

We take $V=H^1(\Sigma)$, $H=L^2(\pr\Sigma)$, and consider the graph (multi-valued operator) \[\Gamma=\{(f,h) \in H\times H | \exists u\in V: \tau(u) =f, Q(u,v) = \langle h, \tau(v) \rangle_{L^2(\pr\Sigma)} \text{ for all }v\in V\}.\] 
(Note that if the $(\Lap+p)$-harmonic function $u$ in the above is smooth up to the boundary, then indeed it will have Dirichlet data $u|_{\pr\Sigma}=f$ and Robin data $(\pr_\eta - q)u|_{\pr\Sigma}=h$.)

Recall that $V(Q)\subset H^1(\Sigma)$ is the space of (weak) solutions to $(\Lap+p)u=0$, and the Dirichlet kernel is given by $\mathcal{W}_0= H^1_0(\Sigma)\cap V(Q)$. (These correspond to the spaces $V(Q)$ in \cite[Theorem 4.10]{AEKS} and $W(Q)$ in \cite[Theorem 4.5]{AEKS} respectively.) The set of Robin data for the Dirichlet kernel is given by $\Gamma(0) = \{h \in H | (0,h) \in \Gamma\}$. As every $w\in \mathcal{W}_0$ is smooth up to the boundary, by the discussion above \[\Gamma(0) = \{ \pr_\eta w | w\in \mathcal{W}_0\}.\] 

Following \cite[Section 3]{AEKS}, there is a \textit{single-valued} operator $\mathring{\Gamma}: \mathcal{D}(\Gamma)\to \Gamma(0)^\perp$ by $\mathring{\Gamma}f = g$, where the domain $\mathcal{D}(\Gamma) = \{ f\in H| \exists (f,h) \in \Gamma\}$, and $h$ is the (unique) element of $\Gamma(0)^\perp$ for which $(f,h)\in \Gamma$. This is our Dirichlet-to-Robin (DtR) map. 

The form $Q$ is clearly symmetric, continuous on $H^1(\Sigma)$ and, by Lemma \ref{lem:coercive}, `compactly elliptic' with respect to the compact operator $j\oplus \tau: V\to L^2(\Sigma)\oplus L^2(\pr\Sigma)$. By \cite[Proposition 4.5, Proposition 4.8 and Theorem 4.13]{AEKS}, it follows that the graph $\Gamma$ is self-adjoint, has compact resolvent and is bounded below. By \cite[Section 3]{AEKS}, the single-valued operator $\mathring{\Gamma}$ has the corresponding properties, so by the usual spectral theory for such operators, has an orthonormal eigenbasis on $\mathring{H}=\Gamma(0)^\perp$. It is a byproduct of their theory that $\mathring{H} = \overline{\tau(V(Q))}$. 

Note that if $v\in \Gamma(0)^\perp$ is an eigenfunction, $\mathring{\Gamma}v = \mu v$, by definition of the DtR operator there is a unique $\hat{v} \in H^1(\Sigma)$ such that $\hat{v}|_{\pr\Sigma} =v$ and 
\begin{equation}
\label{eq:stek-Q-action}
Q(\hat{v}, \phi) = \mu \int_{\pr\Sigma}  v \phi
\end{equation}
for all $\phi \in H^1(\Sigma)$. 

\begin{remark}
Strictly, the theory in \cite{AEKS} is for \textit{complex} Hilbert spaces, so we should apply it on the complexified spaces (and complexified form $Q$). This does not cause any issues, as all the constructed spaces respect the splitting by real and imaginary parts, and similarly we can recover the real spectrum by diagonalising in the real parts of each (complex) eigenspace. 
\end{remark}

We summarise the above discussion as the following: 

\begin{proposition}
The (single-valued) DtR operator $\mathring{\Gamma}$ defined on $\mathring{H}=\Gamma(0)^\perp = \overline{\tau(V(Q))}$ is self-adjoint and bounded below with compact resolvent. 

There is an $L^2(\pr\Sigma)$-orthonormal basis of $\mathring{H}$ of eigenfunctions $\varphi^{\mathrm{S}}_k \in \tau(V(Q))$ of $\mathring{\Gamma}$, and the corresponding modified Steklov eigenvalues $\lambda^{\mathrm{S}}_k$ are discrete, accumulating only at $\infty$. Moreover, we have the variational characterisation
\begin{equation}
\label{eq:stek-variational}
 \lambda^{\mathrm{S}}_k 
 = \inf_{\substack{u\in \tau(V(Q)) \\ \forall j<k, \langle u, \varphi^{\mathrm{S}}_j\rangle_{L^2(\pr\Sigma)}=0} } \frac{Q(\hat{u},\hat{u})}{\|u\|_{L^2(\pr\Sigma)}^2}, 
 \end{equation} 
where $\hat{u} \in V(Q)$ is any extension with $T\hat{u}=u$. 
\end{proposition}

Note that any classical solution of (\ref{eq:steklov-system}) is certainly in $\tau(V(Q))$ (hence $\Gamma(0)^\perp$), and so is indeed a Steklov eigenfunction as defined above via the DtR map.

Also note that the 0-Steklov eigenspace indeed corresponds to the kernel $\ker(Q)$.

\begin{remark}
If $q_0$ is a strictly positive smooth function on $\pr\Sigma$, then the construction of the Dirichlet-to-Robin map and the resulting Steklov-type eigenvalues, proceeds in precisely the same manner if one imposes the $q_0$-weighted norm $\left(\int_{\pr\Sigma} q_0 u^2\right)^\frac{1}{2}$ on $L^2(\pr\Sigma)$. In this setting, the eigenvalue problem (\ref{eq:steklov-system}) instead has boundary condition $(\pr_\eta - q)u = \lambda q_0 u$. This flexibility can be useful - for instance, if the boundary weight $q$ associated to $Q_{p,q}$ is known to be positive. In this article, however, we will only need the case $q_0=1$, so we have made this choice to simplify the presentation.
\end{remark}

\subsection{Combining Dirichlet and Steklov spectra}

Recall that $\mathcal{W}_0$ denotes the Dirichlet kernel and $\Gamma(0) = \{ \pr_\eta w | w\in \mathcal{W}_0\}$. We will need the following Fredholm alternative result.

\begin{lemma}
\label{lem:fredholm}

We have $\tau(V)\cap \Gamma(0)^\perp = \tau(V(Q))$. That is, given $u\in H^1(\Sigma)$, there exists a (weak) solution $\hat{u} \in H^1(\Sigma)$ of 
\[\begin{cases} (\Lap+p)\hat{u}=0 , & \text{ on }\Sigma \\ \hat{u} = u , &\text{ on } \pr\Sigma,\end{cases}\]
if and only if $u|_{\pr\Sigma} \in \Gamma(0)^\perp$. The solution is unique up to adding elements of $\mathcal{W}_0$. 
\end{lemma}
\begin{proof}
We already have $\overline{\tau(V(Q))} = \Gamma(0)^\perp$, so the $\supset$ direction of the statement is clear. Now suppose $u\in \tau(V)\cap \Gamma(0)^\perp$. 
By the usual Fredholm alternative (see \cite[Theorem 8.6]{GT}), the desired weak solution exists if and only if 

\[\int_\Sigma \left(\langle \nabla u, \nabla w\rangle - puw\right) = 0\]

for all $w$ in the Dirichlet kernel $\mathcal{W}_0$ (which may be interpreted as $(\Lap+p)u \perp \mathcal{W}_0$, in the weak sense).
As in (\ref{eq:dir-Q-action}), the regularity of Dirichlet eigenfunctions gives \[Q(w,u) = \int_\Sigma \left(\langle \nabla u, \nabla w\rangle - puw\right)  = \int_{\pr\Sigma} u \pr_\eta w.\]

Thus the solution $\hat{u}\in H^1(\Sigma)$ exists if and only if $u|_{\pr\Sigma} \in \Gamma(0)^\perp$, which completes the proof.  
\end{proof}

We will show that the index of $Q$ is given by counting nonpositive eigenvalues of the Dirichlet problem (\ref{eq:Dir-spec})) and negative eigenvalues of the modified Steklov problem (\ref{eq:steklov-system}). 

One way to proceed is a `hands-on' manner as in Tran's work \cite{Tr20}: The key is that elements $w\in \mathcal{W}_0$ can be modified by adding $cb$, where $b$ is orthogonal to the negative Dirichlet directions, $b|_{\pr\Sigma} = (\pr_\eta-q)w$ and $c\in \mathbb{R}$ is such that $Q(w+cb) = 2c\int_{\pr\Sigma} b^2 + c^2 Q(b)<0$. In this way, the Dirichlet kernel contributes to the index as hinted earlier, in addition to the strictly negative directions for both eigenvalue problems.

Instead, for expedience we will use the somewhat more abstract theory in Tran-Zhou \cite{TZ23}. In particular, we will follow their Theorem 5.2 quite closely, with some small modifications due to our use of the DtR map. 

\begin{proposition}
\label{prop:index-sum}
$\ind(Q) = a+b$, where $a$ is the number of nonpositive Dirichlet eigenvalues, and $b$ is the number of negative (modified) Steklov eigenvalues. 
\end{proposition}
\begin{proof}
Let $\varphi^{\mathrm{D}}_0,\cdots,\varphi^{\mathrm{D}}_k$ be the Dirichlet eigenfunctions with eigenvalue at most 0, and $\mathcal{U}^{\mathrm{D}} = \spa \{\varphi^{\mathrm{D}}_i\}_{0\leq i\leq k}$. Similarly, let $\varphi^{\mathrm{S}}_0,\cdots, \varphi^{\mathrm{S}}_l$ be the Steklov eigenfunctions with eigenvalue less than 0, fix $(\Lap+p)$-harmonic extensions $\hat{\varphi}_{\mathrm{S}, j}$ and set $\mathcal{U}^{\mathrm{S}} = \spa \{\hat{\varphi}_{\mathrm{S}, j}\}_{0\leq j\leq l}$. 

Suppose $u\in\mathcal{U}^{\mathrm{D}}, v\in\mathcal{U}^{\mathrm{S}}$. As $u|_{\pr\Sigma}=0$ and $v$ is $(\Lap+p)$-harmonic, we certainly have $Q(u,v)=0$. Also, $Q$ is clearly negative semidefinite on $\mathcal{U}^{\mathrm{D}}$ and negative definite on $\mathcal{U}^{\mathrm{S}}$, so \[Q(u+v) = Q(u) + 2Q(u,v) + Q(v) \leq -\delta \|v\|_{L^2(\Sigma)}^2\] for some $\delta>0$. If $u+v=0$, then $Q(u+v)=0$ and the inequality above implies that $v=0$ and hence $u=0$. Thus $\mathcal{U}^{\mathrm{D}},\mathcal{U}^{\mathrm{S}}$ are linearly independent. The same inequality then also implies that $Q$ is negative semidefinite on $\mathcal{U}^{\mathrm{D}}\oplus \mathcal{U}^{\mathrm{S}}$. 

By \cite[Theorem 1.5]{TZ23}, it suffices to show that $Q$ is positive semidefinite when restricted to the space \[\mathcal{V} = \{ \phi \in H^1(\Sigma) | Q(\phi,f)=0 \text{ for all } f\in \mathcal{U}^{\mathrm{D}}\oplus \mathcal{U}^{\mathrm{S}}\}.\]

Let $\phi\in\mathcal{V}$. This assumption gives certain orthogonality relations: For the Dirichlet eigenfunctions $\varphi^{\mathrm{D}}_i$, by (\ref{eq:dir-Q-action}) we have 

\begin{equation}
\label{eq:dir-Q-action-orthog}
 Q(\varphi^{\mathrm{D}}_i,\phi)= \lambda^{\mathrm{D}}_i\int_\Sigma \phi\varphi^{\mathrm{D}}_i + \int_{\pr\Sigma} \phi\pr_\eta \varphi^{\mathrm{D}}_i=0.
  \end{equation}

For the (extended) Steklov eigenfunctions $\hat{\varphi}_{\mathrm{S}, j}$, by (\ref{eq:stek-Q-action}) we have 
\begin{equation}
\label{eq:stek-Q-action-orthog} 
Q(\hat{\varphi}_{\mathrm{S}, j} , \phi) =  \lambda^{\mathrm{S}}_j \int_{\pr\Sigma}  \phi \varphi^{\mathrm{S}}_j=0.
\end{equation}

From (\ref{eq:dir-Q-action-orthog}), we see that $\psi|_{\pr\Sigma}$ is orthogonal (in $L^2(\pr\Sigma)$) to $\pr_\eta w$, for any $w$ in the Dirichlet kernel $\mathcal{W}_0$. That is, $\phi|_{\pr\Sigma} \in \Gamma(0)^\perp$, so by the Fredholm alternative (Lemma \ref{lem:fredholm}), there exists $h \in H^1(\Sigma)$ that is a (weak) solution of $(\Lap+p)h=0$ on $\Sigma$, $h|_{\pr\Sigma} = \phi|_{\pr\Sigma}$. 

Let $u= \phi-h$ so that 
\[Q(\phi) = Q(u) + 2Q(u,h) + Q(h).\]

We will analyse each term in succession. First, noting that $u|_{\pr\Sigma}=0$, we find using (\ref{eq:dir-Q-action-orthog}) that
$Q(\varphi^{\mathrm{D}}_i, u) = \lambda^{\mathrm{D}}_i\int_\Sigma \phi\varphi^{\mathrm{D}}_i =0.$ In particular, $u$ is $L^2(\Sigma)$-orthogonal to each of the negative Dirichlet eigenfunctions, so by the variational characterisation (\ref{eq:dir-variational}) it follows that $Q(u)\geq 0$. 

As $u|_{\pr\Sigma}=0$ and $h$ is $(\Lap+p)$-harmonic, we certainly have $Q(u,h)=0$. 

For the last term, by (\ref{eq:stek-Q-action-orthog}) we see that $h|_{\pr\Sigma} = \phi|_{\pr\Sigma}$ is $L^2(\pr\Sigma)$-orthogonal to each of the negative Steklov eigenfunctions. As $h\in V(Q)$, by the variational characterisation (\ref{eq:stek-variational}), it follows that $Q(h)\geq 0$. 

Thus we have shown that $Q(\phi)\geq 0$ for any $\phi \in \mathcal{V}$, which completes the proof. 
\end{proof}

We now turn to the variational properties of the Dirichlet and Steklov eigenvalues. There are analogous statements for higher eigenvalues, but we will only need them for the lowest eigenvalues.

\begin{lemma}
\label{lem:comparison}
Consider a quadratic form $Q$ as above, with associated Dirichet and Steklov eigenvalues $\{\lambda^{\mathrm{D}}_i\}$, $\{\lambda^{\mathrm{S}}_j\}$ respectively. 
\begin{enumerate}
\item For any $u\in H^1_0(\Sigma)$, we have $\int_\Sigma (|\nabla u|^2-pu^2) \geq \lambda^{\mathrm{D}}_0\int_\Sigma u^2$. 
\item If $\lambda^{\mathrm{D}}_0>0$, then for any $u \in H^1(\Sigma)$ we have $ \int_\Sigma (|\nabla u|^2 -pu^2)\geq \int _{\pr\Sigma} (q+ \lambda^{\mathrm{S}}_0) u^2$. 
\item If $\lambda^{\mathrm{D}}_0>0$, then for any $u \in H^1(\Sigma)$ such that $u|_{\pr\Sigma}$ is $L^2(\pr\Sigma)$-orthogonal to the first (modified) Steklov eigenfunction $\varphi^{\mathrm{S}}_0$, we have $ \int_\Sigma (|\nabla u|^2 -pu^2)\geq \int _{\pr\Sigma} (q+ \lambda^{\mathrm{S}}_1 ) u^2$. 
\item If $\lambda^{\mathrm{D}}_0=0$, then for any $u \in H^1(\Sigma)$ such that $u|_{\pr\Sigma}$ is $L^2(\pr\Sigma)$-orthogonal to $\pr_\eta \varphi^{\mathrm{D}}_0$, we have $ \int_\Sigma (|\nabla u|^2 -pu^2)\geq \int _{\pr\Sigma} (q+ \lambda^{\mathrm{S}}_0 ) u^2$. 
\end{enumerate}
\end{lemma}
\begin{proof}
Item (1) follows directly from the variational characterisation of Dirichlet eigenvalues. 

For items (2-4), in each case (by Lemma \ref{lem:fredholm}) there is a $(\Lap+p)$-harmonic extension $\hat{u}\in V(Q)$ with $\tau(\hat{u})=\tau(u)$. In each case, the corresponding variational characterisation gives $\int_{\pr\Sigma} (q+\lambda^{\mathrm{S}}_i )u^2 \leq \int_\Sigma (|\nabla \hat{u}|^2 - p\hat{u}^2)$, where $i=0$ for items (2, 4) and $i=1$ for item (3). Each item then follows by Lemma \ref{lem:ext-min} below.
\end{proof}

\begin{lemma}
\label{lem:ext-min}
Suppose $\lambda^{\mathrm{D}}_0\geq 0$ and consider $u \in H^1(\Sigma)$. Suppose $\hat{u}\in V(Q)$ has the same boundary data $\tau(\hat{u})=\tau(u)$. Then \[\int_\Sigma (|\nabla u|^2 -pu^2) \geq \int_\Sigma (|\nabla \hat{u}|^2 -p\hat{u}^2).\]

If $\lambda^{\mathrm{D}}_0>0$, then equality holds only if $u \in V(Q)$. 
\end{lemma}
\begin{proof}
Consider $v:= u-\hat{u}$. Then $Q(u) = Q(\hat{u}) + Q(v) + 2Q(\hat{u},v)$. As $v\in H^1_0(\Sigma)$, we have
 \[Q(\hat{u},v) = \int_\Sigma (\langle \nabla \hat{u},\nabla v\rangle -p\hat{u}v) = -\int_\Sigma v (\Lap+p) \hat{u} + \int_{\pr\Sigma} v(\pr_\eta - q)\hat{u} = 0.\] 

By the min-max characterisation of Dirichlet eigenvalues we have $Q(v) \geq \lambda^{\mathrm{D}}_0 \|v\|_{L^2(\Sigma)}^2 $. Thus 
\[0 \leq \lambda^{\mathrm{D}}_0 \|v\|_{L^2(\Sigma)}^2 \leq Q(u)-Q(\hat{u}) = \int_\Sigma (|\nabla u|^2 -pu^2) - \int_\Sigma (|\nabla \hat{u}|^2 -p\hat{u}^2).\]
\end{proof}

Finally, we record a standard argument for simplicity of the least Steklov eigenvalue:

\begin{lemma}
\label{lem:steklov-bottom}
Suppose that $\Sigma$ is connected, and that $\lambda^{\mathrm{D}}_0>0$. Then the first (modified) Steklov eigenvalue $\lambda^{\mathrm{S}}_0$ is simple, and the corresponding eigenfunction $\varphi^{\mathrm{S}}_0$ has a unique extension $\hat{\varphi}_{\mathrm{S}, 0} \in V(Q)$ that is smooth and positive on the interior of $\Sigma$. 
\end{lemma}
\begin{proof}
Let $\varphi^{\mathrm{S}}_0\in \tau(V(Q))$ be an eigenfunction associated to $\lambda^{\mathrm{S}}_0$. As $\mathcal{W}_0=0$ there is a unique extension $\hat{\varphi}_{\mathrm{S}, 0} \in V(Q)$ with $\tau(\hat{\varphi}_{\mathrm{S}, 0})=\varphi^{\mathrm{S}}_0$. 

Consider the function $|\hat{\varphi}_{\mathrm{S}, 0}|$. We have $\int_\Sigma |\nabla |\hat{\varphi}_{\mathrm{S}, 0}||^2 = \int_\Sigma |\nabla \hat{\varphi}_{\mathrm{S}, 0}|^2$, and $|\hat{\varphi}_{\mathrm{S}, 0}|^2 = (\hat{\varphi}^{\mathrm{S}}_0)^2$. Therefore $|\hat{\varphi}_{\mathrm{S}, 0}| \in H^1(\Sigma)$, and $\frac{Q(|\hat{\varphi}_{\mathrm{S}, 0}|,|\hat{\varphi}_{\mathrm{S}, 0}|)}{\| |\hat{\varphi}_{\mathrm{S}, 0}| \|_{L^2(\pr\Sigma)}^2}= \frac{Q(\hat{\varphi}_{\mathrm{S}, 0},\hat{\varphi}_{\mathrm{S}, 0})}{\| \hat{\varphi}_{\mathrm{S}, 0} \|_{L^2(\pr\Sigma)}^2} = \lambda^{\mathrm{S}}_0$. It then follows from Lemma \ref{lem:comparison}(2) that $|\hat{\varphi}_{\mathrm{S}, 0}|$ must also be $(\Lap+p)$-harmonic. But by elliptic regularity, both $\hat{\varphi}_{\mathrm{S}, 0}$ and $|\hat{\varphi}_{\mathrm{S}, 0}|$ must be smooth on the interior. Therefore $\hat{\varphi}_{\mathrm{S}, 0}$ cannot change sign on the interior, so by the strong maximum principle must not vanish on the interior. 

Now suppose that $v'_0\in \tau(V(Q))$ is another eigenfunction associated to $\lambda^{\mathrm{S}}_0$, with unique extension $\hat{v}'_0 \in V(Q)$. Without loss of generality we may assume $\varphi^{\mathrm{S}}_0, v'_0>0$ on the interior of $\Sigma$. Then there exists $\alpha\in\mathbb{R}$ so that $\int_\Sigma (\hat{v}'_0-\alpha \hat{\varphi}_{\mathrm{S}, 0})=0$. But $\hat{v}'_0-\alpha \hat{\varphi}_{\mathrm{S}, 0}$ is also a $\lambda^{\mathrm{S}}_0$-eigenfunction, so cannot change sign on $\Sigma$. This implies $\hat{v}'_0 \equiv \alpha\hat{\varphi}_{\mathrm{S}, 0}$, so in particular $v'_0 = \alpha \varphi^{\mathrm{S}}_0$ and the first eigenfunction is simple. 
\end{proof}

\section{Spectral and Morse index of minimal submanifolds}
\label{sec:index-apps}

In this section, we discuss certain index problems for minimal submanifolds in the sphere, which admit natural eigenfunctions. Our main application will be to the Morse index of capillary surfaces in spherical caps. Recall that for a minimal hypersurface $(\Sigma, \pr \Sigma) \looparrowright (M,S)$ that contacts $S$ with angle $\gamma$ along $\pr \Sigma$, the second variation of the modified area functional is $Q^{\mathrm{A}} = Q_{p_{\mathrm{A}}, q_{\mathrm{A}}}$, where \[p_{\mathrm{A}} = |A|^2 + \Ric(\nu,\nu),\qquad q_{\mathrm{A}}=\frac{1}{\sin \gamma} k_S(\bar{\nu},\bar{\nu}) - \cot\gamma \, A(\eta,\eta).\] The usual Morse index of $\Sigma$ is precisely $\ind(\Sigma) = \ind(Q^{\mathrm{A}})$. 

Under certain conditions, we will prove an Urbano-type Theorem \ref{thm:urbano} partially characterising capillary minimal surfaces $(\Sigma, \pr\Sigma) \looparrowright (\mathbb{S}^3, \pr B_\pit)$ with $\ind(\Sigma)\leq 4$. 

First, as a warm-up, we apply the general setup in Section \ref{sec:quadratic} to the spectral index of a free boundary minimal submanifolds $(\Sigma^k,\pr\Sigma) \looparrowright (\mathbb{S}^n, \pr B_R)$. The spectral index was introduced in Karpukhin-M\'{e}tras \cite{KaM22} for $k$-harmonic maps (to $\mathbb{S}^n$ or $\mathbb{B}^n$) and considered by Medvedev \cite{Me23} for free boundary minimal submanifolds in $B^n_R \subset \mathbb{S}^n$. The spectral index may be used to give a condition which is closely related to the notion of being `immersed by first eigenfunctions' (see Proposition \ref{prop:energy-ind-1} and Remark \ref{rmk:closed-index}). Again, this was essentially pointed out in \cite{Me23}; the only difference here is that we discuss the possibility of \textit{unconstrained} free boundary minimal submanifolds. 

We also remark that lower bounds for the \textit{Morse} index of free boundary minimal submanifolds in a spherical cap were proven by \cite{LM23} and \cite{Spe25}; this topic will not be further investigated in this article. 

\subsection{Free boundary minimal submanifolds in a cap and the spectral index}
\label{sec:FBMS-spectral}

Consider a free boundary minimal submanifold $(\Sigma^k,\pr\Sigma) \looparrowright (\mathbb{M}^n, \pr B_R)$, where $\mathbb{M}^n$ is a space form of constant curvature $\kappa$ and $R>0$. In this setting, one may define a \textit{spectral index} as the index of the quadratic form $Q^{\mathrm{S}} = Q_{p_{\mathrm{S}}, q_{\mathrm{S}}}$, where $p_{\mathrm{S}} = (n-1)\kappa$ and \[q_{\mathrm{S}} = \ct_\kappa(R) = \begin{cases} \sqrt{\kappa} \cot(R\sqrt{\kappa}) &, \kappa >0 \\ \frac{1}{R}, &\kappa =0\\ \sqrt{|\kappa|}\coth(r\sqrt{|\kappa|}) &,\kappa <0.\end{cases}\] 

Here we will only be interested in the case $\mathbb{M} = \mathbb{S}^n$, so that $\kappa=1$ and explicitly 
\[p_{\mathrm{S}} = n-1 ,\qquad q_{\mathrm{S}} = \cot R.\]

Note that in this setting our $q_{\mathrm{S}}= q_{\mathrm{A}}$ and $p_{\mathrm{S}} = p_{\mathrm{A}} - |A|^2$. The quadratic form $Q^{\mathrm{S}}$ is essentially the Dirichlet form, but shifted by multiples of $\|\cdot\|_{L^2(\Sigma)}^2$ and $\|\cdot\|_{L^2(\pr\Sigma)}^2$. 

Standard computations give the following key equations for the coordinate functions on $\Sigma \looparrowright \mathbb{S}^n \hookrightarrow \mathbb{R}^{n+1}$ (cf. \cite[Lemma 2.9]{NZ25a} for the boundary calculations):

\begin{lemma}[\cite{NZ25a}]
\label{lem:x-comp}
Consider the functions $x_i = \langle x, e_i\rangle|_\Sigma$. 
\begin{itemize}
\item For any $i$ we have $(\Lap+n-1)x_i=0$ on $\Sigma$.
\item For $i>0$ we have $(\pr_\eta -\cot R)x_i=0$ on $\pr\Sigma$.
\item If $R\neq\pit$ we have $(\pr_\eta +\tan R)x_0=0$ on $\pr\Sigma$.
\item If $R=\pit$ we have $x_0=0$ on $\pr\Sigma$.
\end{itemize}
\end{lemma}

Define the spaces of (restrictions of) coordinate functions $\mathcal{C}= \spa \{x_i\}_{i\geq 0}$ and $\mathcal{C}_0= \spa \{x_i\}_{i> 0}$. Note that $\mathcal{C}_0 \subset \ker(Q^{\mathrm{S}})$. 

\begin{lemma}
\label{lem:coord-full}
If $\Sigma$ is not totally geodesic, then $\dim \mathcal{C} \geq k+2$. 
\end{lemma}
\begin{proof}
If $\dim \mathcal{C} \leq k+1$ then there is an $(n-k)$-dimensional subspace $\mathcal{V} \subset \mathbb{R}^{n+1}$ such that $\langle x,a\rangle$ vanishes on $\Sigma$ for any $a\in \mathcal{V}$. That is, $\Sigma$ is contained in the orthocomplement $\mathcal{V}^\perp$, hence in $\mathcal{V}^\perp \cap \mathbb{S}^n$, which is a totally-geodesic $k$-dimensional sphere in $\mathbb{S}^n$. 
\end{proof}

We consider the Dirichlet and Steklov problems associated to $Q^{\mathrm{S}}$ as in Section \ref{sec:quadratic}. (Note that our Steklov eigenvalues will be shifted by $\cot R$ compared to the Steklov eigenvalues considered by \cite{LM23}.)

%

If $(\Sigma^k,\pr\Sigma)$ is \textit{constrained} in $(B^n_R, \pr B_R)$, then $x_0>0$ on the interior of $\Sigma$. As in \cite{LM23}, it follows that if $R<\pit$, then $\lambda^{\mathrm{D}}_0>0$, $\lambda^{\mathrm{S}}_0 = -\tan R-\cot R$. Likewise, if $R=\pit$ then $\lambda^{\mathrm{D}}_0=0$. 

Instead of making the constrained assumption, we observe that $\ind(Q^{\mathrm{S}})\leq 1$ implies (for $R\leq \pit$) that $x_0\geq 0$. (As per Remark \ref{rmk:constrained}, this in turn allows one to use the maximum principle to recover that $x_0\geq \cos R$, that is, $\Sigma$ is constrained in $B_R$.) 

\begin{proposition}
\label{prop:energy-ind-1}
Consider a connected free boundary minimal submanifold $(\Sigma^k,\pr\Sigma) \looparrowright (\mathbb{S}^n, \pr B_R)$. 
Assume that $\Sigma$ is not totally geodesic. Then $\ind(Q^{\mathrm{S}})\geq 1$ (and $\ind_0(Q^{\mathrm{S}})\geq k+1$.)

Moreover:
\begin{itemize}
\item Let $R=\pit$. Then $\ind(Q^{\mathrm{S}})=1$ if and only if $x_0\in W$ is the zeroth Dirichlet eigenfunction and $\lambda^{\mathrm{D}}_0=\lambda^{\mathrm{S}}_0=0$. 
\item Let $R<\pit$. Then $\ind(Q^{\mathrm{S}})=1$ if and only if $x_0$ is the zeroth Steklov eigenfunction (with corresponding eigenvalue $\lambda^{\mathrm{S}}_0=-\tan R-\cot R$) and $\lambda^{\mathrm{D}}_0>0$, $\lambda^{\mathrm{S}}_1=0$. 
\end{itemize}
In either case above, we have $x_0>0$ on the interior of $\Sigma$. 
\end{proposition}
\begin{proof}
By Proposition \ref{prop:index-sum} we count nonnegative Dirichlet eigenvalues and negative Steklov eigenvalues.

If $\lambda^{\mathrm{D}}_0\leq 0$ then the zeroth Dirichlet eigenfunction already gives $\ind(Q^{\mathrm{S}})\geq 1$. 

If $\lambda^{\mathrm{D}}_0>0$, note that $\mathcal{C}_0$ consists of 0-Steklov eigenfunctions. As $\Sigma$ is not totally geodesic, it follows from Lemma \ref{lem:coord-full} that $ \dim \mathcal{C}_0\geq k+1>1$. By Lemma \ref{lem:steklov-bottom} the zeroth Steklov eigenvalue must be simple, which forces $\lambda^{\mathrm{S}}_0<0$ and hence $\ind(Q^{\mathrm{S}})\geq 1$. 

The `if' part is clear in both cases by simplicity of the relevant zeroth eigenvalue.

For the `only if' parts, assume that $\ind(Q^{\mathrm{S}})=1$. If $R=\pit$, then by Lemma \ref{lem:x-comp}, $x_0$ is a 0-Dirichlet eigenfunction. If $\lambda^{\mathrm{D}}_0<0$ then this would give $\ind(Q^{\mathrm{S}})\geq 2$, so we must have $\lambda^{\mathrm{D}}_0=0$. Similarly, as the $x_0$ are 0-Steklov eigenfunctions, we must have $\lambda^{\mathrm{S}}_0=0$. As the zeroth Dirichlet eigenfunction, $x_0$ cannot vanish on the interior of $\Sigma$. By the free boundary condition, we must have $x_0>0$ near the boundary, and hence on all of the interior. 


Finally, suppose $R<\pit$. Then $x_0$ is a Steklov eigenfunction with eigenvalue $-\tan R-\cot R <0$. This implies $\lambda^{\mathrm{S}}_0 = -\tan R -\cot R$, $\lambda^{\mathrm{S}}_1=0$ and $\lambda^{\mathrm{D}}_0>0$, as any other configuration would force $\ind(Q^{\mathrm{S}})\geq 2$. It follows from Lemma \ref{lem:steklov-bottom} that again $x_0>0$ on $\Sigma$ (noting that $x_0|_{\pr\Sigma}=\cos R>0$). 
\end{proof}

Note that being constrained is a consequence of $\ind(Q^{\mathrm{S}})= 1$. As a result, the work of Lima-Menezes \cite{LM23} applies (see also \cite[Proof of Corollary 5.11]{Me23}).

\begin{proposition}
\label{prop:FB-first-efns}
Consider a free boundary minimal annulus $(\Sigma^2,\pr\Sigma) \looparrowright (\mathbb{S}^n, \pr B_R)$, where $R\leq \pit$. Then $\ind(Q^{\mathrm{S}})=1$ if and only if $n=3$ and $\Sigma$ is rotationally symmetric, embedded and constrained in $B_R$. 
\end{proposition}
\begin{proof}
We consider the eigenvalue notions associated to $Q^{\mathrm{S}}$. 

First we consider $R<\pit$. For the `if' statement, suppose that $\Sigma$ is constrained, rotationally symmetric and $n=3$. By \cite[Theorem 4]{LM23}, $\Sigma$ is immersed by first eigenfunctions in the sense that (in our shifted notation) $\lambda^{\mathrm{D}}_0>0$, $\lambda^{\mathrm{S}}_1=0$. Proposition \ref{prop:energy-ind-1} then gives $\ind(Q^{\mathrm{S}})=1$. 

For the `only if' statement, suppose that $\ind(Q^{\mathrm{S}})=1$. By Proposition \ref{prop:energy-ind-1} (and Remark \ref{rmk:constrained}), $\Sigma$ is constrained in $B_R$. Moreover, if $R<\pit$ then $\lambda^{\mathrm{D}}_0>0$ and $\lambda^{\mathrm{S}}_1 = 0$. Then Lima-Menezes' \cite[Theorem C]{LM23} applies and gives the desired conclusions. 

Now consider $R=\pit$. In this setting $q_{\mathrm{S}}=0$, so the Robin eigenvalue problem for $Q^{\mathrm{S}}$ is just the Neumann eigenvalue problem for $\Lap+n-1$. 

For the `if' statement, note that when $n=3$, a constrained, rotationally symmetric annulus $\Sigma$ must be the half Clifford torus (see \cite[Lemmas 6.8 and 6.9]{NZ25a}). For the half Clifford torus, the Neumann eigenvalues of the Laplacian are explicit, and it follows easily that $\ind(Q^{\mathrm{S}})=1$ (and $\nul(Q^{\mathrm{S}})=3$.)

For the `only if' statement, in this case $\ind(Q^{\mathrm{S}})=1$ implies $\lambda^{\mathrm{D}}_0=\lambda^{\mathrm{S}}_0=0$. Moreover, $\ind(Q^{\mathrm{S}})=1$ implies that $\lambda_{1, \mathrm{R}} =0$. Now let $\Sigma'$ be the doubled surface obtained by reflection across the equator $\pr B_\pit$. By the free boundary condition, $\Sigma'$ is a $C^2$ minimal torus (hence smooth). We claim that $\Sigma'$ is immersed by first eigenfunctions, equivalently $Q'(f):=\int_{\Sigma'} ( |\nabla f|^2 - (n-1) f^2) \geq 0$ whenever $\int_{\Sigma'}f=0$. 

Indeed, decompose $f=f_1 + f_2$, where $f_1$ is odd and $f_2$ is even with respect to the reflection. By applying the reflection, we see that $Q'(f_1,f_2) = \int_{\Sigma'} (\langle \nabla f_1, \nabla f_2\rangle -(n-1)f_1f_2)=0$, and $\int_{\Sigma'} f = 2\int_\Sigma f_2 =0$. 

As $f_1$ is odd, it must vanish on the equator, so Dirichlet variational characterisation (\ref{eq:dir-variational}) applied to $f_1|_{\Sigma}$ gives $Q'(f_1)=2\int_{\Sigma} ( |\nabla f_1|^2 - (n-1) f_1^2) \geq0$. As $\int_\Sigma f_2=0$, the Robin variational characterisation (and the discussion above) applied to $f_2|_{\Sigma}$ implies that $Q'(f_2)=\int_{\Sigma'} ( |\nabla f_2|^2 - (n-1) f_2^2) \geq0$. Thus $Q'(f) = Q'(f_1)+Q'(f_2)\geq 0$, which establishes the claim. By Montiel-Ros' \cite[Theorem 4]{MR86}, we conclude that $\Sigma'$ is the Clifford torus. 
\end{proof}

We close this subsection with some remarks on related settings:

\begin{remark}
\label{rmk:closed-index}
For closed minimal submanifolds $\Sigma \looparrowright \mathbb{S}^n$, there is no boundary and the Dirichlet form corresponds to the spectrum of the Laplacian, $\Lap u = - \lambda u$ (for which $\lambda_{0}(\Lap)=0$, realised by the constant functions). The familiar notion of `immersed by first eigenfunctions' is precisely that $\lambda_1(\Lap) = n-1$. As in \cite{KaM22}, the condition $\lambda_1(\Lap)=n-1$ is equivalent to $\ind(Q^{\mathrm{S}})=1$, where \[Q^{\mathrm{S}}(u) = \int_\Sigma (|\nabla u|^2 - (n-1)u^2).\]
\end{remark}

\begin{remark}
For $R\in(0,\pit]$, the statements in this section all hold for free boundary minimal submanifolds in $(\mathbb{M}_{\kappa_R}, \pr B_{R\csc R})$, where $\kappa_R = \sin^2 R$ as considered in \cite[Section 2.5]{NZ25a}, because this setting is just a rescaling of $(\mathbb{S}^n, \pr B_R)$. 

In this case $R=0$, that is, for free boundary minimal submanifolds in $(\mathbb{B}^n, \pr \mathbb{B}^n)$, the modified energy form is $Q^{\mathrm{S}}(u) = \int_\Sigma |\nabla u|^2 - \int_{\pr\Sigma}u^2$, and It is well-known that the zeroth Dirichlet eigenvalue $\lambda^{\mathrm{D}}_0>0$. Moreover, (in our shifted setup) the constant functions are the zeroth Steklov eigenfunctions $\lambda^{\mathrm{S}}_0=-1$ and the coordinate functions $x_i|_\Sigma$, $i=1,\cdots,n$, are always 0-Steklov eigenfunctions. 

One can also use the general setup above to deduce that $\ind(Q^{\mathrm{S}})=1$ if and only if the coordinate functions are the \textit{first} eigenfunctions, in the sense that $\lambda^{\mathrm{S}}_1=0$. Consequently, a free boundary minimal annulus in $(\mathbb{B}^n, \pr \mathbb{B}^n)$ with $\ind(Q^{\mathrm{S}})=1$ must be the critical catenoid, by the work of \cite{FS16}.
These observations are implicit in \cite[Remark 5.7]{Me23}, so we omit the details as they are precisely analogous to the discussion above for $R\in(0,\pit]$. 
\end{remark}


\subsection{Capillary surfaces in the hemisphere and the Morse index}
\label{sec:cap-morse}

Consider 2-sided minimal hypersurfaces $(\Sigma,\pr\Sigma) \looparrowright (\mathbb{S}^n, \pr B_\pit)$ which contact $\pr  B_\pit$ with angle $\gamma$ along $\pr\Sigma$. Recall that the Morse index $\ind(\Sigma) = \ind(Q^{\mathrm{A}})$, where $Q^{\mathrm{A}} = Q_{p_{\mathrm{A}},q_{\mathrm{A}}}$ and in this setting \[p_{\mathrm{A}} = |A|^2+n-1  ,\qquad q_{\mathrm{A}} = -\cot\gamma\, A(\eta,\eta).\] 

Standard computations give the following key equations for components of the Gauss map of $\Sigma \looparrowright\mathbb{S}^n \hookrightarrow \mathbb{R}^{n+1}$ (cf. \cite[Lemma 2.9]{NZ25a} for the boundary calculations):

\begin{lemma}[\cite{NZ25a}]
\label{lem:nu-comp}
Consider the functions $\nu_i = \langle \nu, e_i\rangle|_\Sigma$. 
\begin{itemize}
\item For any $i$ we have $(\Lap+|A|^2)\nu_i=0$ on $\Sigma$.
\item For $i>0$ we have $(\pr_\eta - A(\eta,\eta)\cot\gamma)\nu_i=0$ on $\pr\Sigma$.
\item If $\gamma\neq\pit$ we have $(\pr_\eta +A(\eta,\eta) \tan \gamma)\nu_0=0$ on $\pr\Sigma$.
\item We have $\nu_0=-\sin R\cos\gamma$ on $\pr\Sigma$.
\item For any $i>0$ we have $\cos\gamma \, \nu_i +\sin\gamma \, \langle \eta, e_i\rangle =0$ on $\pr\Sigma$. 
\end{itemize}
\end{lemma}

We first record the simplest case, where $\Sigma$ is totally geodesic. 

\begin{lemma}
\label{lem:tot-geo-index}
Let $\Sigma$ be a totally geodesic half-equator in $\mathbb{S}^n$. Then $Q^{\mathrm{A}} = Q^{\mathrm{S}}$ and $\ind(\Sigma)=1$. 
\end{lemma}
\begin{proof}
By assumption we have $A_\Sigma= 0$, so $p_{\mathrm{S}} = p_{\mathrm{A}} = n-1$ and $q_{\mathrm{S}}=q_{\mathrm{A}}=0$. The Robin eigenvalue problem corresponding to $Q^{\mathrm{S}}=Q^{\mathrm{A}}=:Q$ then becomes the \textit{Neumann} eigenvalue problem 
\[
\begin{cases}
(\Lap+n-1)u= - \omega u &, \text{ on } \Sigma, \\
\pr_\eta u=0 &, \text{ on } \pr\Sigma. 
\end{cases}
\]
But $\Sigma$ is isometric to the hemisphere $\mathbb{S}^{n-1}_+$. The Neumann-Laplace spectrum for $\mathbb{S}^{n-1}_+$ is well-known, and the lowest two eigenvalues are $0$ (multiplicity $1$) and $n-1$ (multiplicity $n-1$). The above eigenvalue problem for $Q$ is just shifted down by $n-1$, so indeed $\ind(Q) = 1$ (and $\nul(Q)=n-1$.) 
\end{proof}

Similar to before, we now define the spaces $\mathcal{G}= \spa \{\nu_i\}_{i\geq 0}$ and $\mathcal{G}_0 = \spa \{\nu_i\}_{i> 0}$, spanned by components of the Gauss map $\nu$ of $\Sigma$. 

\begin{lemma}
\label{lem:gauss-full}
If $\Sigma$ is not totally geodesic then $\dim \mathcal{G} = n+1$. 
\end{lemma}
\begin{proof}
If $\dim\mathcal{G} < n+1$, then there is a unit vector $a$ for which $\langle \nu,a\rangle$ vanishes on $\Sigma$. Consider the open subset $U$ of $\Sigma$ on which $|A|\neq 0$. As $\langle \nu,a\rangle\equiv 0$, differentiating gives $\langle A(X), a\rangle =0$ for any $X\in T_x\Sigma$. As $A$ is nondegenerate at $x\in U$, it gives a linear automorphism $T_x\Sigma \to T_x\Sigma$, and in particular we see that $T_x\Sigma$ is orthogonal to $a$ in $\mathbb{R}^4$. As $\nu$ is also orthogonal to $a$, we must have $x=\pm a$. In particular, $x$ is locally constant on $U$, so if $U$ is nonempty then $x$ cannot be an immersion. We conclude that $U=\emptyset$ and so $|A|\equiv 0$ on $\Sigma$. 
\end{proof}

\begin{remark}
If one only wants to show $\dim\mathcal{G}_0 = n$, there is an argument similar to Urbano's original paper \cite{Ur90}: If $\dim\mathcal{G}_0<n$, then there is a unit vector $a$, orthogonal to $e_0$, so that $ \langle\nu, a\rangle$ vanishes on $\Sigma$. Without loss of generality, we may assume $a=e_1$. Then $\nu_1 \equiv 0$, and it follows that $\Hess^\Sigma x_1 = - x_1 g_\Sigma$. 
By Lemma \ref{lem:nu-comp}, we also have $\pr_\eta x_1 = \langle \eta, e_1\rangle = -\cot\gamma \, \nu_1 =0$ along $\pr\Sigma$. 
Applying the Obata-type theorem with Neumann boundary conditions proven by Escobar \cite[Theorem 4.2]{Esc90} and Xia \cite{Xia91} yields that $\Sigma$ is totally geodesic. 
\end{remark}

For the remainder of this section, we will actually consider spectral problems with respect to the form 
\begin{equation}
\label{eq:modified-Q-A}
Q^{\mathrm{A}}_* = Q_{p_{\mathrm{A}}-(n-1), q_{\mathrm{A}}}, \qquad p_{\mathrm{A}}-(n-1) = |A|^2, \quad q_{\mathrm{A}} = -\cot\gamma\, A(\eta,\eta).\end{equation} 
By this definition, $\mathcal{G}_0 \subset \ker(Q^{\mathrm{A}}_*)$. Moreover, the form $Q^{\mathrm{A}}_*$ is related to the index form by $Q^{\mathrm{A}}_*(u) = Q^{\mathrm{A}}(u) - (n-1)\int_\Sigma u^2$, which clearly implies
\begin{equation}
\label{eq:modified-Q-A-relation}
\ind_0(Q^{\mathrm{A}}_*) \leq \ind(Q^{\mathrm{A}})=\ind(\Sigma).
\end{equation} 

We consider the Dirichlet and modified Steklov problems for $Q^{\mathrm{A}}_*$, as in Section \ref{sec:quadratic}.


The following result is similar to its `dual' result Proposition \ref{prop:energy-ind-1}, but note that we only assume an integral condition on $q_{\mathrm{A}} = -\cot\gamma\, A(\eta,\eta)$, rather than a pointwise condition. (The analogous assumption in Proposition \ref{prop:energy-ind-1} is the condition $\cos R \geq 0$; see Section \ref{sec:dual} below, and \cite[Section 5]{NZ25a} for more discussion on the dual framework.) 

\begin{proposition}
\label{prop:area-ind-1}
Consider a connected, 2-sided minimal hypersurface $(\Sigma,\pr\Sigma) \looparrowright (\mathbb{S}^n, \pr B_\pit)$ which contacts $\pr  B_\pit$ with angle $\gamma$ along $\pr\Sigma$. 

Assume that $\Sigma$ is not totally geodesic. Then we have $\ind(Q^{\mathrm{A}}_*)\geq 1$ (and $\ind_0(Q^{\mathrm{A}}_*)\geq n+1$.) 

Henceforth suppose that $\ind_0(Q^{\mathrm{A}}_*)=n+1$. 

If $\gamma=\pit$, then $\lambda^{\mathrm{D}}_0=\lambda^{\mathrm{S}}_0=0$, and $\nu_0\in W$ is the zeroth Dirichlet eigenfunction.

If $\gamma \in (0,\pit)$, assume that $\int_{\pr\Sigma} q_{\mathrm{A}} \geq 0$. Then either:
\begin{itemize}
\item $\lambda^{\mathrm{D}}_0>0$, and $\lambda^{\mathrm{S}}_1=0$; or
\item $\lambda^{\mathrm{D}}_0=0$, and $\lambda^{\mathrm{S}}_0=0$. 
\end{itemize}
\end{proposition}
\begin{proof}

According to Proposition \ref{prop:index-sum} we should count nonnegative Dirichlet eigenvalues and negative Steklov eigenvalues.

If $\lambda^{\mathrm{D}}_0\leq 0$ then the zeroth Dirichlet eigenfunction already gives $\ind(Q^{\mathrm{A}}_*)\geq 1$. 

If $\lambda^{\mathrm{D}}_0>0$, note that $\mathcal{G}_0$ consists of 0-Steklov eigenfunctions. As $\Sigma$ is not totally geodesic, it follows from Lemma \ref{lem:gauss-full} that $ \dim \mathcal{G}_0= n >1$. By Lemma \ref{lem:steklov-bottom} the zeroth Steklov eigenvalue must be simple, which forces $\lambda^{\mathrm{S}}_0<0$ and hence $\ind(Q^{\mathrm{A}}_*)\geq 1$ (with equality only if $\lambda^{\mathrm{S}}_1=0$). 

Henceforth assume that $\ind_0(Q^{\mathrm{A}}_*)=n+1$. 

First consider the special case $\gamma =\pit$. By Lemma \ref{lem:nu-comp}, $\nu_0$ is a 0-Dirichlet eigenfunction. If $\lambda^{\mathrm{D}}_0<0$ then this would give $\ind(Q^{\mathrm{A}}_*)\geq 2$ (hence $\ind_0(Q^{\mathrm{A}}_*)\geq n+2$), so we must have $\lambda^{\mathrm{D}}_0=0$. Similarly, as the $\nu_i$ are 0-Steklov eigenfunctions, we must have $\lambda^{\mathrm{S}}_0=0$. As the zeroth Dirichlet eigenfunction, $\nu_0$ cannot vanish on the interior of $\Sigma$. 

Now consider $\gamma\in(0,\pit)$ and assume $\int_{\pr\Sigma} q_{\mathrm{A}} \geq 0$. Suppose for the sake of contradiction that $\lambda^{\mathrm{D}}_0< 0$. If we also had $\lambda^{\mathrm{D}}_1\leq 0$, then we would immediately have $\ind(Q^{\mathrm{A}}_*) \geq 2$, hence $\ind_0(Q^{\mathrm{A}}_*) \geq n+2$. So we must have $\lambda^{\mathrm{D}}_1>0$, and in particular the Dirichlet kernel $\mathcal{W}_0=0$. 

Although $\nu_0$ may not be an eigenfunction, we will still use it to show that $\ind_0(Q^{\mathrm{A}}_*)\geq n+2$. Indeed, let $\varphi^{\mathrm{D}}_0$ be a Dirichlet eigenfunction corresponding to $\lambda^{\mathrm{D}}_0$, and consider $\mathcal{V} := \spa \{\varphi^{\mathrm{D}}_0\} \oplus \mathcal{G}$. 

We claim that $\varphi^{\mathrm{D}}_0\notin \mathcal{G}$. Indeed, suppose that $\varphi^{\mathrm{D}}_0 = \sum_{i=0}^n a_i \nu_i =: \nu_a$ for some $a=(a_i)\in\mathbb{R}^{n+1}$. Then in particular, $\nu_a|_{\pr\Sigma}=\varphi^{\mathrm{D}}_0|_{\pr\Sigma}=0$, but we also have $(\Lap+|A|^2)\nu_a=0$ on $\Sigma$. That is, $\nu_a\in W=0$, so $\nu_a=0$. As $\dim\mathcal{G} = n+1$, we must have $a=0$ which completes the proof of the claim. 

Now that $\varphi^{\mathrm{D}}_0\notin \mathcal{G}$, by Lemma \ref{lem:gauss-full} we have $\dim \mathcal{V} = 1+\dim\mathcal{G} = n+2$. 

By (\ref{eq:dir-Q-action}) we have \[Q(\varphi^{\mathrm{D}}_0,\varphi^{\mathrm{D}}_0) = -\lambda^{\mathrm{D}}_0 \int_\Sigma (\varphi^{\mathrm{D}}_0)^2 \leq 0.\] As $\nu_0$ is smooth, we may integrate by parts to find (using $(\Lap+|A|^2)\nu_0=0$, $\pr_\eta \nu_0 = -A(\eta,\eta)\tan\gamma\, \nu_0$, $\nu_0^2|_{\pr\Sigma} = \cos^2\gamma$ and $\varphi^{\mathrm{D}}_0|_{\pr\Sigma}=0$) 

\[ Q(\nu_0,\nu_0) = -\int_{\pr\Sigma} \nu_0 (\pr_\eta-q_{\mathrm{A}})\nu_0 = - \int_{\pr\Sigma} q_{\mathrm{A}} \leq 0 ,\]

\[ Q(\varphi^{\mathrm{D}}_0,\nu_0) =  \int_{\pr\Sigma} \varphi^{\mathrm{D}}_0 (\pr_\eta-q_{\mathrm{A}})\nu_0 =0.\]

As $\mathcal{G}_0 \subset \ker(Q^{\mathrm{A}}_*)$, it follows that $Q$ is negative semidefinite on $\mathcal{V}$. Therefore $\ind_0(Q^{\mathrm{A}}_*)\geq \dim\mathcal{V} = n+2$, which is the desired contradiction. 

So we must have $\lambda^{\mathrm{D}}_0\geq 0$. As above, if $\lambda^{\mathrm{D}}_0>0$ then we must have $\lambda^{\mathrm{S}}_0<0$ and $\lambda^{\mathrm{S}}_1=0$. If $\lambda^{\mathrm{D}}_0=0$, then we must have $\lambda^{\mathrm{S}}_0=0$ (otherwise the zeroth Dirichlet and zeroth Steklov eigenfunctions would give $\ind(Q^{\mathrm{A}}_*) \geq 2$, hence $\ind_0(Q^{\mathrm{A}}_*)\geq n+2$.)
\end{proof}

\subsubsection{Dual surfaces and dual forms}
\label{sec:dual}

Here we first prove a characterisation of capillary minimal surfaces in the hemisphere by their index, under the assumption that the surface is diffeomorphic to an annulus. For this, it will be expedient to recall the dual operation from \cite{NZ25a}: In this section an $(R,\gamma)$-minimal surface will refer to a minimal surface $(\Sigma^2, \pr\Sigma) \looparrowright (\mathbb{S}^3, \pr B_R)$ which contacts $S=\pr B_R$ at angle $\gamma$. Moreover, in \cite[Section 5.2]{NZ25a} we showed that if $\Sigma$ is an annulus, then up to a sign $\varepsilon \in \{\pm1\}$ (chosen so that $\varepsilon A(\eta,\eta)>0$ on $\pr \Sigma$) the Gauss map $\varepsilon\nu$ gives a \textit{dual} $(\tilde{R},\tilde{\gamma})$-minimal annulus, where $\tilde{R}\in(0,\pi)$, $\tilde{\gamma}\in(0,\pit]$ satisfy
\begin{equation}
\label{eq:dual-params}\cos \tilde{R} = -\varepsilon \sin R\cos\gamma , \qquad \sin\tilde{R} \sin\tilde{\gamma} =\sin R\sin \gamma.\end{equation}

We will write $\tilde{Q}^{\mathrm{A}}$ for the index form on $\tilde{\Sigma}$, and so forth. 

\begin{proposition}
\label{prop:index-dual}
Let $\Sigma$ be an $(R,\gamma)$-minimal annulus. Suppose that $-\cot\gamma A(\eta,\eta)$ and $\cot R$ have the same sign, that is, $-\cot R\cot\gamma \, A(\eta,\eta) >0$ along $\pr\Sigma$. 
Then the index form of the dual satisfies $Q^{\mathrm{A}}(f) = \tilde{Q}^{\mathrm{A}}(f)$. \end{proposition}
\begin{proof}
For convenience, we write $\psi = \frac{1}{\sqrt{2}}|A|_g>0$ for the positive principal curvature of $\Sigma$. 
We will use tildes to denote corresponding geometric quantities on $\tilde{\Sigma}$. 

By the discussion in \cite[Section 5.1]{NZ25a}, we have $\tilde{g} = \psi^2 g$, and it follows that the curvature, gradient, area and length elements transform as
\[|\tilde{A}|_{\tilde{g}}^2 = 2\psi^{-2}, \qquad |\nabla^{\tilde{g}}f|_{\tilde{g}}^2 = \psi^{-2} |\nabla^g f|_g^2, \qquad \d\tilde{\mu} = \psi^2 \d\mu, \qquad \d\tilde{\sigma} = \psi \d\sigma.\] Therefore

\[\begin{split}
Q^{\mathrm{A}}(f) &= \int_\Sigma \left( |\nabla^g f|_g^2 - 2(\psi^2+1)f^2\right)\d\mu - \int_{\pr\Sigma} q_{\mathrm{A}} f^2 \d\sigma
\\&= \int_\Sigma \left( |\nabla^{\tilde{g}}f|_{\tilde{g}}^2 - 2(1+\psi^{-2})f^2\right)\d\tilde{\mu} - \int_{\pr\Sigma} \psi^{-1}q_{\mathrm{A}} f^2 \d\tilde{\sigma}
\end{split}\]

As $|\tilde{A}|_{\tilde{g}}^2 = 2\psi^{-2}$, we can recognise the first term as the interior term of $\tilde{Q}^{\mathrm{A}}$. For the boundary term, note that by the choice of sign $\varepsilon$, we have $\psi = \varepsilon\cot\gamma A(\eta,\eta)$, so that $q_{\mathrm{A}} = \csc\gamma\,\cot R -\varepsilon\psi\cot\gamma$. Similarly, (as $A(\eta,\eta)$ and $\tilde{A}(\tilde{\eta},\tilde{\eta})$ have the same sign) we have $\tilde{q}_A = \csc\tilde{\gamma} \cot\tilde{R} - \varepsilon \psi^{-1} \cot\tilde{\gamma}$. 

Dividing the equations (\ref{eq:dual-params}) by each other gives $\csc\tilde{\gamma}\cot\tilde{R} = - \varepsilon\cot\gamma$. Also as our construction ensured $\tilde{\gamma} \in (0,\pit]$, we have 

\[\begin{split} \cot\tilde{\gamma} &= \sqrt{\csc\tilde{\gamma}^2-1} = \sqrt{\frac{\sin^2 \tilde{R} - \sin^2 R\sin^2 \gamma}{\sin^2 R\sin^2 \gamma}}=\sqrt{\frac{1-\sin^2 R\cos^2\gamma- \sin^2 R\sin^2 \gamma}{\sin^2 R\sin^2 \gamma}}
\\& = |\cot R| \csc\gamma . 
\end{split}\]

By the sign assumption on $\cot R$ we have $|\cot R| = -\varepsilon \cot R$, thus \[\tilde{q}_A = - \varepsilon \cot\gamma -\varepsilon \psi^{-1}|\cot R| \csc\gamma = \psi^{-1}q_{\mathrm{A}}\] and we conclude that $Q^{\mathrm{A}}(f) = \tilde{Q}^{\mathrm{A}}(f)$ as claimed. 
\end{proof}

Consider a $(\pit,\gamma)$-capillary minimal annulus. Its dual is a $(\tilde{R},\pit)$-capillary minimal annulus, where $\tilde{R}=\gamma$ or $\pi-\gamma$, and has an associated modified energy form $\tilde{Q}^{\mathrm{S}}$. This is related to the modified index form as follows:

\begin{corollary}
\label{cor:index-energy}
Let $\Sigma$ be a $(\pit,\gamma)$-capillary minimal annulus. Then $Q^{\mathrm{A}}_*(f) = \tilde{Q}^{\mathrm{S}}(f)$. 
\end{corollary}
\begin{proof}
As in the proof of Proposition \ref{prop:index-dual}, we have

\[\begin{split}
Q^{\mathrm{A}}_*(f) &= \int_\Sigma \left( |\nabla^g f|_g^2 - 2\psi^2 f^2\right)\d\mu - \int_{\pr\Sigma} q_{\mathrm{A}} f^2 \d\sigma
\\&= \int_\Sigma \left( |\nabla^{\tilde{g}}f|_{\tilde{g}}^2 - 2 f^2\right)\d\tilde{\mu} - \int_{\pr\Sigma} \psi^{-1}q_{\mathrm{A}} f^2 \d\tilde{\sigma},
\end{split}\]
and in this case $\psi^{-1}q_{\mathrm{A}} =\cot\tilde{R}= \tilde{q}_{\mathrm{A}} = \tilde{q}_{\mathrm{S}}$, $\tilde{p}_{\mathrm{S}}=2$. 
\end{proof}


Note that if $\Sigma$ is a $(\pit,\gamma)$-capillary minimal annulus with $q_{\mathrm{A}}>0$, the dual must have $\tilde{R}\leq \pit$ (hence $\tilde{R}=\gamma$). We may then deduce the following index characterisation for annuli:

\begin{proposition}
\label{thm:index-4-annulus}
Let $\Sigma$ be a $(\pit,\gamma)$-capillary minimal annulus with $q_{\mathrm{A}}>0$. If $\ind(\Sigma) = 4$ then $\Sigma$ is rotationally symmetric. 
\end{proposition}
\begin{proof}
As $\mathcal{G}_0\subset \ker(Q^{\mathrm{A}}_*)$ has dimension $n=3$, we have $\ind(Q^{\mathrm{A}}_*) \leq 4-3=1$. 
By Corollary \ref{cor:index-energy}, the dual $\tilde{\Sigma}$ has $\ind(\tilde{Q}^{\mathrm{S}}) \leq 1$. By Proposition \ref{prop:FB-first-efns}, it follows that $\tilde{\Sigma}$ is rotationally symmetric. It follow that $\Sigma$ is also rotationally symmetric (for instance, because the double dual recovers $\Sigma$, modulo the antipodal map). 
\end{proof}

\subsubsection{Characterisation by index}

Finally, we will prove the following Urbano-type theorem, which (under certain conditions) characterises capillary minimal surfaces in the hemisphere $\mathbb{S}^3_+$ with second lowest Morse index (after the totally geodesic equators). Recall that (by Corollary \ref{cor:bdry-mc}) in this setting \[q_{\mathrm{A}} = -\cot\gamma\, A(\eta,\eta)= - H^{\pr\Sigma}_\Sigma.\] 

\begin{theorem}
\label{thm:urbano}
Let $(\Sigma,\pr\Sigma) \looparrowright (\mathbb{S}^3, \pr B_\pit)$ be a 2-sided capillary minimal surface with contact angle $\gamma$, which is not totally geodesic. Then $\ind(\Sigma)\geq 4$. 

Further, assume that $\int_{\pr\Sigma} q_{\mathrm{A}} \geq 0,$
and either
\begin{itemize}
\item $\cos\gamma=0$; or
\item $\cos\gamma \in [0, \frac{2}{9})$, and each component of $\pr\Sigma$ is embedded in $\pr B_\pit$; or
\item $\cos\gamma \in [0, \frac{2}{5})$, and $\Sigma$ has embedded wet surface $S^- \subset \pr B_\pit$.
\end{itemize}

Then $\ind(\Sigma)=4$ holds only if $\Sigma$ is a (rotationally symmetric) annulus.
\end{theorem}

Note that when $\gamma=\pit$, we will require effectively no additional assumptions (as $q_\mathrm{A} \equiv 0$). Otherwise, we require $\gamma$ to be close to $\pit$; the assumptions have been listed in order of increasing range of $\gamma$, or correspondingly, increasing strength of hypothesis on the wet surface.

The assumptions on the wet surface are natural if $\Sigma$ arises from the global variational problem for the wetting energy. As in the introduction, we remark that the additional boundary condition $\int_{\pr\Sigma} q_{\mathrm{A}}\geq 0$ should be related to some inherent nonuniqueness for capillary minimal surfaces. This is related to the problem of identifying contact angles $\gamma$ and $\pi-\gamma$, and also (using the dual picture of \cite{NZ25a})  to possible nonuniqueness for free boundary minimal surfaces in $B_{\pi-\gamma}$. See \cite[Remark 1.3]{NZ25a} for more discussion and a visual representation. There, we noted that the capillary minimal surfaces dual to free boundary minimal surfaces in $B_{\pi-\gamma}$ do not seem to be radial graphs, and one may also verify that they do not satisfy the sign condition on $q_{\mathrm{A}}$.  Of course, the expected examples in \cite[Remark 1.3]{NZ25a} are still rotationally symmetric, so we believe that the statement above may still hold without the boundary curvature condition, but the nonuniqueness would likely complicate attempts at a proof. 

The key step, as in Urbano's work \cite{Ur90} (see also \cite{Dev21}), is to control the topology of $\Sigma$ by its index. 

\begin{proposition}
\label{prop:urbano-topo}
Let $(\Sigma,\pr\Sigma) \looparrowright (\mathbb{S}^3, \pr B_\pit)$, be a 2-sided capillary minimal surface with contact angle $\gamma$. Assume that $\int_{\pr\Sigma} q_{\mathrm{A}} \geq 0.$

Further assume that either:
\begin{itemize}
\item $\cos\gamma=0$; or
\item $\cos\gamma \in [0, \frac{2}{9})$, and each component of $\pr\Sigma$ is embedded in $\pr B_\pit$; or
\item $\cos\gamma \in [0, \frac{2}{5})$, and $\Sigma$ has embedded wet surface $S^- \subset \pr B_\pit$.
\end{itemize}

If $\Sigma$ is not totally geodesic and $\ind_0(Q^{\mathrm{A}}_*)\leq 4$, then $\Sigma$ must be  diffeomorphic to an annulus.
\end{proposition}
\begin{proof}
First note that $\Sigma$ must be connected: $\Sigma$ has at least one component that is not totally geodesic, and any such component must itself have $\ind_0(Q^{\mathrm{A}}_*)\geq 4$ by Proposition \ref{prop:area-ind-1}. As in Lemma \ref{lem:tot-geo-index}, any totally geodesic component of $\Sigma$ will have $\ind_0(Q^{\mathrm{A}}_*)=1$. In sum, if $\Sigma$ has more than 1 component then it must have $\ind_0(Q^{\mathrm{A}}_*)>4$, contrary to our hypothesis.

Let $\Phi \in \Conf(B_\pit)$ be a fixed conformal diffeomorphism to be chosen later, and define $u_i := \langle \Phi\circ x, e_i\rangle$. Then $\sum_{i=0}^3 u_i^2=1$, and $\sum_{i=0}^3 |\nabla u_i|^2 = 2 e^\phi$, where $e^\phi$ is the conformal factor. As $\Phi$ preserves $\pr B_\pit$, we have that $u_0|_{\pr \Sigma}\equiv 0$. 

Now since $\ind_0(Q^{\mathrm{A}}_*)=4$, by Proposition \ref{prop:area-ind-1} we have $\lambda^{\mathrm{D}}_0\geq 0$. Thus by item (1) of Lemma \ref{lem:comparison}, we have 
\begin{equation}\label{eq:u_0-comparison}
\int_\Sigma |A|^2 u_0^2 \leq \int_\Sigma |\nabla u_0|^2.
\end{equation}

We will divide into the two cases of Proposition \ref{prop:area-ind-1}. We claim that there is a choice of $\Phi$ so that 

\begin{equation}\label{eq:u_i-comparison}
\int_\Sigma |A|^2 u_i^2 +\int_{\pr\Sigma} q_{\mathrm{A}} u_i^2   \leq \int_\Sigma |\nabla u_i|^2 .\end{equation}

Case 1: $\lambda^{\mathrm{D}}_0=0, \lambda^{\mathrm{S}}_0=0$. Let $\varphi^{\mathrm{D}}_0$ be a 0-Dirichlet eigenfunction which is positive on $\Sigma$. The Hopf lemma implies $\pr_\eta \varphi^{\mathrm{D}}_0 <0$ on $\pr\Sigma$. By the conformal balancing argument of \cite{He70, LY82} (see \cite[Proof of Lemma 5.1]{CFP14} for surfaces with boundary), we may choose $\Phi\in\Conf(B_\pit)$ so that $\int_{\pr \Sigma} u_i \pr_\eta\varphi^{\mathrm{D}}_0 =0$ for $i=1,2,3$.
Then item (4) of Lemma \ref{lem:comparison} yields (\ref{eq:u_i-comparison}) as $\lambda^{\mathrm{S}}_0=0$.

Case 2: $\lambda^{\mathrm{D}}_0>0, \lambda^{\mathrm{S}}_1=0$. 
By Lemma \ref{lem:steklov-bottom}, the modified Steklov eigenfunction $\varphi^{\mathrm{S}}_0$ corresponding to $\lambda^{\mathrm{S}}_0$ has a unique extension $\hat{\varphi}_{\mathrm{S}, 0} \in V(Q)$ that is smooth and positive on the interior of $\Sigma$. In particular, we certainly have $\varphi^{\mathrm{S}}_0\geq 0$, and again by conformal balancing, we may choose $\Phi\in\Conf(B_\pit)$ so that $\int_{\pr \Sigma} u_i \varphi^{\mathrm{S}}_0  =0$ for $i=1,2,3$.
Then item (3) of Lemma \ref{lem:comparison} yields (\ref{eq:u_i-comparison}) as $\lambda^{\mathrm{S}}_1=0$.

In either case, summing (\ref{eq:u_0-comparison}) and (\ref{eq:u_i-comparison}) gives (recall $u_0|_{\pr\Sigma}\equiv 0$)

\begin{equation}
\label{eq:urbano-ineq}
\int_\Sigma |A|^2 + \int_{\pr\Sigma} q_{\mathrm{A}} \leq \int_\Sigma 2e^\phi = 2|\Phi_* \Sigma|.
\end{equation}

The Gauss equations give that $|A|^2 = 2(1-K_\Sigma)$, where $K_\Sigma$ is the Gaussian curvature. For convenience denote $\kappa=H^{\pr\Sigma}_S$, so that by Corollary \ref{cor:bdry-mc} we have $q_{\mathrm{A}} = - H^{\pr\Sigma}_\Sigma = \kappa \cos\gamma$. Gauss-Bonnet gives that 
\begin{equation}
\label{eq:GB-Sig}
|\Sigma| - \frac{1}{2}\int_\Sigma |A|^2 - \cos\gamma \int_{\pr\Sigma} \kappa = 2\pi \chi(\Sigma).
\end{equation}

Suppose $\Sigma$ has genus $g$ and $b\geq 1$ boundary components. If $\chi(\Sigma)= 2-2g-b>0$, then $\Sigma$ is a disc and hence totally geodesic by \cite[Proposition 4.1]{NZ25a}. Henceforth, we will assume $\chi(\Sigma)\leq 0$. 

If $\gamma =\pit$, then the boundary terms vanish in (\ref{eq:urbano-ineq}), (\ref{eq:GB-Sig}). Eliminating the $|A|^2$ term gives 
\[ 2|\Sigma| \leq 2|\Phi_*\Sigma| + 4\pi \chi(\Sigma).\]
Using the conformal maximisation Theorem \ref{thm:conf-max-sph-intro} gives that $|\Phi_* \Sigma|\leq |\Sigma|$, hence $\chi(\Sigma)=0$ and $\Sigma$ is an annulus.

To proceed with the other contact angle cases, we involve the wet surface $S^-$. Note that the Gaussian curvature of $S=\pr B_\pit$ is $1$, so Gauss-Bonnet gives 

\begin{equation}
\label{eq:GB-S}
|S^-| + \int_{\pr\Sigma} \kappa = 2\pi\chi(S^-).
\end{equation}

As above, we use (\ref{eq:GB-Sig}) to eliminate the $|A|^2$ term from (\ref{eq:urbano-ineq}), then use (\ref{eq:GB-S}) to eliminate any excess $\kappa$ terms. In this way, we find that

\begin{equation}\begin{split}
\label{eq:urbano-ineq-1}
2|\Sigma| - \cos\gamma \left(2\pi\chi(S^-) - |S^-|\right) - 4\pi\chi(\Sigma) &\leq 2|\Phi_*\Sigma|
\\&\leq 2|\Sigma| +2\cos\gamma\, \left( |S^-| - |\Phi(S^-)|\right),
\end{split}\end{equation}
where on the second line we have used that $|\Phi_*\Sigma| + \cos\gamma \, |\Phi_* S^-| \leq |\Sigma|+\cos\gamma \,|S^-|$ by the conformal maximisation Corollary \ref{cor:conf-max-sph-cap-intro}. Rearranging (\ref{eq:urbano-ineq-1}) gives 

\begin{equation}
\label{eq:urbano-topo}
 \cos\gamma \left( - 2\pi\chi(S^-) - |S^-| +2|\Phi(S^-)| \right)  \leq 4\pi\chi(\Sigma) .
\end{equation}

Now assume that the wet surface $S^-$ is embedded and $\cos\gamma \in [0,\frac{2}{5})$. Then $S^-$ is a region in $\pr B_\pit \simeq \mathbb{S}^2$ (in particular, having genus 0) with $b$ boundary components, so $2-b \leq \chi(S^-) \leq b$. Using $\chi(S^-) \leq b$, $|S^-|\leq |\mathbb{S}^2|=4\pi$ and $|\Phi(S^-)|\geq 0$ in (\ref{eq:urbano-topo}) and rearranging gives

\[  (2-\cos\gamma) b  \leq 4(1-g)+2\cos\gamma.\]

Note that $\cos\gamma <\frac{2}{5}<\frac{1}{2}$ implies that $2\cos\gamma < 2-\cos\gamma$, so (as $b\geq 1$) we cannot have $g\geq 1$. Thus $g=0$, and $(2-\cos\gamma)b \leq 4+2\cos\gamma$. As $\cos\gamma <\frac{2}{5}$ implies that  $\frac{4+2\cos\gamma}{2-\cos\gamma}<3$, we must have $b=2$ and $\Sigma$ is an annulus as desired. 

In the final case, we assume only that each component of $\pr\Sigma$ is embedded in $\pr B_\pit$, but $\cos\gamma \in [0,\frac{2}{9})$. Then as in Remark \ref{rmk:wet-surface}, $\Sigma$ has a wet surface $S^-$, each component of which is embedded in $\pr B_\pit$. Now we still have $\chi(S^-)\leq b$, but we only bound the area by $4\pi$ for each component, which gives the weaker bound $|S^-| \leq 4\pi b$. Using these bounds and $|\Phi(S^-)|\geq0$ in (\ref{eq:urbano-topo}) yields
\[ (2-3\cos\gamma) b  \leq 4(1-g) .\]
As $\cos\gamma < \frac{2}{9}<\frac{2}{3}$ and $b\geq 1$ we must again have $g=0$. As $\cos\gamma <\frac{2}{9}$ implies $\frac{4}{2-3\cos\gamma} <3$, we must have $b=2$, which completes the proof. 
\end{proof}

We now have all the ingredients needed to prove Theorem \ref{thm:urbano}:

\begin{proof}[Proof of Theorem \ref{thm:urbano}]
As $\Sigma$ is not totally geodesic, Proposition \ref{prop:area-ind-1} gives that $\ind(\Sigma) \geq \ind_0(Q^{\mathrm{A}}_*)\geq 4$. If $\ind(\Sigma)=4$, then we must have equality in the above, so in particular $\ind_0(Q^{\mathrm{A}}_*)=4$. Under the stated additional assumptions we may apply Proposition \ref{prop:urbano-topo} to deduce that $\Sigma$ is diffeomorphic to an annulus. By \cite[Corollary 4.2]{NZ25a}, $q_{\mathrm{A}} =-\cot\gamma\, A(\eta,\eta)$ has the same sign at all points on $\pr\Sigma$, which by our assumptions must be $q_{\mathrm{A}} >0$. Finally, Proposition \ref{thm:index-4-annulus} gives that $\Sigma$ must be rotationally symmetric. 
\end{proof}

\appendix

\section{Additional conformal monotonicity results}
\label{sec:conf-mono-overflow}

In this section, we record a more explicit realisation of the conformal group $\Conf(B^n_R)$ in terms of conformal flows. Previously, in Proposition \ref{prop:conf-description} we described elements of $\Conf(B^n_R)$ (up to rotations in $\mathrm{SO}(n+1)^o$) by conjugation from $\Conf(B^n_\pit)$. We will now give a precise correspondence between the conjugate presentation and a description via the conformal flow. In particular, Lemma \ref{lem:conf-Y} identifies the image $Y$ of the origin for the conjugate map $\Phi_{s_R e_0} \circ \Phi_y \circ \Phi_{-s_R e_0}$. This differs from $\Phi_Y$ by a rotation (in $\mathrm{SO}(n+1)$), which will be identified in Proposition \ref{prop:flow-correspondence}.
 
First, we remark that elements of $\Conf(B^n_R)$ (realised as a subgroup of $\Conf(\mathbb{B}^{n+1})$) are also determined (up to rotations) by where they map the $\textit{origin}$ in $\mathbb{B}^{n+1}$. 

\begin{lemma}
Identify $\Conf(B^n_R)  = \{ \Psi \in \Conf(\mathbb{B}^{n+1}) | \Psi(B^n_R)=B^n_R\}$ as in Proposition \ref{prop:conf-description}. Suppose $\Psi \in \Conf(B^n_R)$ is such that $\Phi(0)=0$. Then $\Psi$ is given by a rotation in $\mathrm{SO}(n+1)^o$. 

If $\Psi, \Psi' \in \Conf(B^n_R)$ and $\Psi(0)=\Psi'(0)$, then there exists $\Theta\in\mathrm{SO}(n+1)^o$ so that (acting on $\mathbb{B}^{n+1}$) we have $\Psi' = \Theta\circ \Psi$. 
\end{lemma}
\begin{proof}
First consider $\Psi \in \Conf(B^n_R)$ with $\Phi(0)=0$. By the description of $\Conf(\mathbb{B}^{n+1})$, we know that $\Psi$ must be a rotation in $\mathrm{SO}(n+1)$. But any rotation which preserves a cap $B^n_R$ must fix the centre $o$. 

For the last statement, note that $\Psi' \circ \Psi^{-1}$ fixes $0$ and so must be a rotation in $\mathrm{SO}(n+1)^o$. 
\end{proof}

Next, we check the image of the origin for elements of $\Conf(B^n_R)$ according to the conjugate description in Proposition \ref{prop:conf-description}. 

\begin{lemma}
\label{lem:conf-Y}
Let $y\in \mathbb{B}^n$ and consider the map $\Phi^n_y \in \Conf(\mathbb{B}^n)$ and (by identifying $\mathbb{B}^n$ with $\mathcal{S}_0 \subset \mathbb{B}^{n+1}$) the map $\Phi^{n+1}_y \in \Conf(\mathbb{B}^{n+1})$. 

Then $\Phi^{n+1}_y\in\Conf(B^n_\pit)$, and its restriction to $B^n_\pit$ coincides with $\Xi \circ \Phi^n_y \circ \Xi^{-1}$, where $\Xi:\mathbb{B}^n \to B^n_\pit$ is the (inverse) stereographic projection.

Moreover, for any $R$ we have that 
\[  Y=(\Phi_{s_R e_0} \circ \Phi_y \circ \Phi_{-s_R e_0})(0) = \frac{|y|^2(\cos R) e_0 + (\sin R) y}{\sin^2 R + |y|^2 \cos^2 R} \in \mathcal{S}_{\cos R}.\]
\end{lemma}
\begin{proof}
The first statement may be verified by direct computation, but we also point out that checking the image of $e_0$ is a simple way to verify the claim up to rotations: As $\langle y,e_0\rangle=0$, certainly $\Xi \circ \Phi^n_y \circ \Xi^{-1}$ and $\Phi^{n+1}_y$ are conformal diffeomorphisms that preserve $B^n_\pit$. But by the explicit formula (\ref{eq:stereo}), note that 
\[ \Xi(\Phi^n_y(\Xi^{-1}(e_0))) = \Xi(\Phi_y(0)) = \Xi(y) = \frac{1-|y|^2}{1+|y|^2}e_0 +  \frac{2y}{1+|y|^2}.\] This is precisely as we computed in (\ref{eq:Phi-y-o}) for $\Phi^{n+1}_y$. By Lemma \ref{lem:rotation}, up to a rotation in $\mathrm{SO}(n+1)^o$, $\Xi \circ \Phi^n_y \circ \Xi^{-1}$ must coincide with $\Phi^{n+1}_y$ on $B^n_\pit$. 

The last statement for $Y=(\Phi_{s_R e_0} \circ \Phi_y \circ \Phi_{-s_R e_0})(0)$ is an essentially algebraic computation using (\ref{sec:conf-euc}), so we outline some key points:

 By definition we have $\Phi_{-s_R e_0}(0) = -s_R e_0$. Now as $\langle y,e_0\rangle=0$, we have

\[ X:=\Phi_y(-s_R e_0) =  \frac{-(1-|y|^2)s_R e_0 +(1+s_R^2)y }{1+s_R^2 |y|^2 }.\]

In particular, $1+ 2\langle X, s_R e_0\rangle = \frac{1-2s_R^2 + 3s_R^2 |y|^2}{1+s_R^2 |y|^2 }$ and $|X|^2 = \frac{|y|^2 + s_R^2}{1+s_R^2 |y|^2 }$. Substituting in to (\ref{eq:Phi-y}) once more, we have

\[
\begin{split}
Y &= \Phi_{s_R e_0}(\Phi_y(-s_R e_0)) = \Phi_{s_R e_0}(X)
\\&= \frac{(1-s^2)(-(1-|y|^2)s_R e_0 +(1+s_R^2)y ) + \left((1-2s_R^2 + 3s_R^2 |y|^2)+ (|y|^2 + s_R^2) \right) s_R e_0}{(1-2s_R^2 + 3s_R^2 |y|^2) + (|y|^2 + s_R^2) s_R^2}
\\&=  \frac{2s_R(1+s_R^2)|y|^2 e_0 + (1-s_R^4)y}{(1-s_R^2)^2 + 4s_R^2 |y|^2}
\end{split}
\]

Recalling that $s_R = \tan \frac{\pit-R}{2}$, it follows that 

\[Y=\frac{|y|^2(\cos R) e_0 + (\sin R) y}{\sin^2 R + |y|^2 \cos^2 R}.\]

Recalling again that $\langle y,e_0\rangle =0$, one may readily verify that indeed
\[\langle Y,e_0\rangle = \frac{|y|^2}{\sin^2 R+ |y|^2 \cos^2 R}\cos R  = |Y|^2 \cos R.\] 
\end{proof}

Recall that $\mathcal{S}_c = \{ x\in\mathbb{B}^{n+1} | \langle x,e_0\rangle = |x|^2 c \}$. By the previous lemma, we may parametrise elements of $\Conf(B^n_R)$ (up to rotations) by $Y\in \mathcal{S}_{\cos R}$. We now describe how to realise these maps by conformal flows:

\begin{proposition}
\label{prop:flow-correspondence}
Let $Y\in \mathcal{S}_{\cos R}$. Set $\bar{s}=|Y|$ and choose coordinates so that $Y= \bar{s}((\cos\alpha)e_0 + (\sin\alpha) e_1)$ for some $\alpha \in [0,\pi]$. Further define $a = (-\cos\alpha)e_0 + (\sin\alpha)e_1$. Finally, let $\Theta_{\beta}\in\mathrm{SO}(n+1)$ be the rotation about the plane spanned by $\{e_0,e_1\}$, which maps $e_0 \mapsto (\cos \beta)e_0 + (\sin\beta)e_1$. 

Then the map $\Psi=\Theta_{-(\pi-2\alpha)} \circ \Phi_{\bar{s}a}$ is a conformal diffeomorphism in $\Conf(B^n_R)$ which maps $\Psi(0)=Y$.
\end{proposition}
\begin{proof}
First note that by definition of $\mathcal{S}_{\cos R}$ we have $\cos \alpha = s\cos R$. By Proposition \ref{prop:flow-ball}, we know that $\Psi^a_t(B^n_R) = B^n_{R_t}(o_t)$. As in the proof of Proposition \ref{prop:flow-ball}, it will be useful to consider the evolution of the diameter with endpoints $z^\pm_t$. Define $\bar{t}=2\artanh\bar{s}$, so that (acting on $\mathbb{S}^n$) we have $\Psi^a_{\bar{t}} = \Phi_{sa}$. 

We divide into cases depending on the sign of $\cot R$. 

Case 1: $\cot R=0$. Then $R=\alpha = \pit$. In particular, we note that $\Psi^a_t$ preserves $B^n_\pit$ for all $t$. Indeed, by Lemma \ref{lem:u-evo} we have $u_a(z^\pm_t) \equiv \pm1$ for all $t$, hence $R_t\equiv \pit$. This implies $o_t =\pm e_0$, so by continuity $o_t\equiv e_0$. We conclude that $\Psi^a_{\bar{t}} = \Phi_{sa}$ is a conformal diffeomorphism in $\Conf(B^n_\pit)$ which maps \[\Psi^a_t(0) = \bar{s}a = \bar{s}e_1=Y.\] 

Case 2: $\cot R\neq 0$. We first claim that, flowing in the particular direction $a$ as above, $\cot R_t$ always has the same sign as $\cot R$. Indeed, by (\ref{eq:moving-radius}) for any $t=2\artanh s$ we have \[\cot R_t = \frac{ 1+s^2 - 2s \bar{s}}{1-s^2}\cot R,\] 
where again $\bar{s}  = \frac{\cos\alpha}{\cos R}$. This establishes the claim as $\bar{s} = |Y| <1$. In particular, $\cot R_t$ will never vanish along this trajectory. Moreover, when $s=\bar{s}$ the above gives that $\cot R_{\bar{t}} = \cot R$, hence $R_{\bar{t}}=R$.

We may locate the centre $o_{\bar{t}}$ in similar fashion. By symmetry, $z^\pm_t$ (hence $o_t$) lie in the plane spanned by $\{e_0,e_1\}$. Note that $u$ is monotone increasing on any trajectory, for instance by (\ref{eq:u-monotone}). In particular, $z^\pm_t$ can never reach $\{\pm a\}$, and so they lie in a fixed semicircle
\[ S= \begin{cases}
\{ (\cos \theta)e_0 + (\sin\theta)e_1| \theta\in (-\alpha, \pi-\alpha) \} &, \text{if }R <\pit \\
\{ (\cos \theta)e_0 + (\sin\theta)e_1| \theta\in (\pi-\alpha, 2\pi-\alpha) \} &, \text{if } R>\pit,
\end{cases}\]
for all $t$. Now by Lemma \ref{lem:u-evo} and (\ref{eq:endpt-initial}), we have at time $\bar{t}$ that

\[u_a(z^\pm_{\bar{t}}) = \frac{2\bar{s} - (1+\bar{s}^2)\cos(\alpha\pm R)}{1+\bar{s}^2-2\bar{s} \cos(\alpha\pm R)} =  \cos(\alpha \mp R).\] 

Thus $z^\pm_{\bar{t}} = (\cos \theta^\pm)e_0 + (\sin\theta^\pm)e_1$, where $\theta^+$ is either $\pi-R$ or $\pi-2\alpha +R$, and $\theta^-$ is either $\pi+R$ or $\pi-2\alpha -R$. As $\bar{s} =\frac{\cos \alpha}{\cos R}\in(0,1)$, we see that if $R\in(0,\pit)$ then $\alpha\in (R,\pit)$, and if $R\in(\pit,\pi)$ then $\alpha\in(\pit,R)$. In either case, the only option for $\theta^\pm$ that satisfies $z^\pm_{\bar{t}}\in S$ is $\theta^\pm = \pi-2\alpha \pm R$. It follows that \[o_{\bar{t}} = \cos(\pi-2\alpha) e_0 + \sin(\pi-2\alpha)e_1.\] 

In particular, $\Theta_{-(\pi-2\alpha)}(o_{\bar{t}}) = e_0$, and as we also have $R_{\bar{t}}=R$, we conclude that the map $\Psi=\Theta_{-(\pi-2\alpha)} \circ \Psi^a_t$ sends $B^n_R$ to itself, so in particular $\Psi \in \Conf(B^n_R)$. By Lemma \ref{lem:flow-relation}, we have $\Psi^a_t(0) = \Phi_{\bar{s}a}(0) = \bar{s}a$. Finally we compute that, as desired,\[\Psi(0) = \Theta_{-(\pi-2\alpha)}(\bar{s}a) = \bar{s}((\cos\alpha)e_0 + (\sin\alpha)e_1) = Y.\] 

\end{proof}

\section{Spherical caps as conformal balls}
\label{sec:model}

In this section, we briefly and explicitly describe how the Euclidean ball $\mathbb{B}^n$ or the hemisphere $B^n_\pit \subset \mathbb{S}^n$ may be used as a conformal model for any $B^n_R \subset \mathbb{S}^n$. Recall the notation $s_R = \tan \frac{\pit-R}{2}$. 

\subsection{Hemisphere model}

Let $\rho$ denote the spherical distance from $o=e_0$. Here we recall the following maps, by which one may use the hemisphere $B^n_\pit$ as a conformal model for $B^n_R$ and $\mathbb{B}^n$:

\begin{enumerate}

\item The map $\Phi_{s_R e_0}$ gives a conformal automorphism of $\mathbb{S}^n$ which takes $B^n_\pit$ to $B_R$. 

\item The stereographic projection $\Xi^{-1}:\mathbb{S}^n\setminus\{-o\}$ takes $B^n_\pit$ to $\mathbb{B}^n$. Explicitly, 

\[\Xi^{-1}(x) = \frac{x-\langle x,e_0\rangle e_0}{1+\langle x,e_0\rangle} .\]

\end{enumerate}

\begin{lemma}
\,\\
\begin{enumerate}
\item Under the conformal diffeomorphism $\Phi_{s_R e_0}$, the pullback metric is $(\Phi_{s_R e_0})^* \bar{g} = e^{2\psi_R} \bar{g} $, where $\psi_R:\mathbb{S}^n\to\mathbb{R}$ is defined by

\[e^{\psi_R}= \frac{\sin R}{1+\cos R \cos \rho}. \]

\item Under the conformal diffeomorphism $\Xi^{-1}$, the pullback metric is $(\Xi^{-1})^*\delta = e^{2\psi_0} \bar{g}$, where $\psi_0:\mathbb{S}^n\setminus\{-o\}\to\mathbb{R}$ is defined by
\[e^{\psi_0} = \frac{1}{1+\cos\rho}.\]
\end{enumerate}
\end{lemma}

\begin{proof}
By Lemmas \ref{lem:flow-relation}, \ref{lem:conf-factor}, and by (\ref{eq:e0-to-e0}), for $x\in B^n_\pit$ we have $(\Phi_{s_R e_0})^*\bar{g} = e^{2\psi_R}\bar{g}$, where
\begin{equation}
\label{eq:conf-factor-psiR}
\begin{split}
e^{2\psi_R} &= \frac{1-\langle \Phi_{s_R e_0}(x), e_0\rangle^2}{1-\langle x,e_0\rangle^2} 
= \frac{1-\left(\frac{\langle x,e_0\rangle +\cos R}{1 + \langle x, e_0\rangle \cos R}\right)^2}{1-\langle x,e_0\rangle^2} 
\\&= \frac{\sin^2 R}{(1+\langle x,e_0\rangle \cos R)^2}. 
\end{split}
\end{equation}

It is well-known that for $z\in \mathbb{B}^n$ the (inverse) stereographic projection $\Xi: \mathbb{B}^n \to B^n_\pit$ satisfies 
\begin{equation}
\label{eq:conf-factor-Xi}
\Xi^*\bar{g} = \left(\frac{2}{1+|z|^2}\right)^2 \delta.
\end{equation}
 Thus $(\Xi^{-1})^* \delta =  e^{2\psi_0}\bar{g}$, where 

\[ e^{\psi_0} = \frac{1+|\Xi^{-1}(x)|^2}{2} = \frac{1+\frac{1-\langle x,e_0\rangle^2}{(1+\langle x,e_0\rangle)^2}}{2} = \frac{1}{1+\langle x,e_0\rangle} \]

Noting that $\cos\rho=\langle x,e_0\rangle$ implies both results. 
\end{proof}

\subsection{Euclidean ball model}

Here, to use $\mathbb{B}^n$ as a conformal model for $B^n_R\subset \mathbb{S}^n$, recall that $\Xi_R=\Phi_{s_R e_0}\circ \Xi$ defines a conformal equivalence $\Xi_R:\mathbb{B}^n \to B^n_R$. 

\begin{lemma}
\label{lem:h-R}
Under the conformal diffeomorphism $\Xi_R:\mathbb{B}^n \to B^n_R$, the pullback metric is given by $\Xi_R^*\bar{g} = f_R^2 \delta$, where $f_R:\mathbb{B}^n\to\mathbb{R}$ is defined by \[f_R(z)=\frac{2\sin R}{1+|z|^2 + (1-|z|^2)\cos R}.\] 
\end{lemma}
\begin{proof}

By the explicit formula (\ref{eq:stereo}), we have $\langle \Xi(z) ,e_0 \rangle = \frac{1-|z|^2}{1+|z|^2}$. Using (\ref{eq:conf-factor-psiR}) and (\ref{eq:conf-factor-Xi}) again, we find that \[\Xi_R^* \bar{g} =  \Xi^*(e^{2\psi_R}\bar{g}) = f_R^2 \delta,\]

where 
\[\begin{split}
f_R(z) &= \left(\frac{\sin R}{1+\frac{1-|z|^2}{1+|z|^2} \cos R}\right)\left( \frac{2}{1+|z|^2}\right) = \frac{2\sin R}{1+|z|^2 + (1-|z|^2)\cos R}. 
\end{split}\]

\end{proof}

\begin{corollary}
\label{cor:area-limit}
The rescaled metrics $(\csc^2 R)\, \Xi_R^* \bar{g}$ on $\mathbb{B}^n$ converge smoothly to $\delta$ as $R\to 0$.

If $(\Sigma,\pr\Sigma)$ is an immersed submanifold in $(\mathbb{B}^n, \pr\mathbb{B}^n)$, and $\Sigma_R = (\Xi_R)_*\Sigma$, then 
\[|\Sigma|_\delta =  \lim_{R\to 0} (\sin R)^{-k} |\Sigma_R|_{\bar{g}}\]

\[|\pr\Sigma|_\delta = \lim_{R\to 0} (\sin R)^{1-k} |\pr\Sigma_R|_{\bar{g}}\]
\end{corollary}
\begin{proof}
As $\Xi_R^*\bar{g} = f_R^2 \delta$, the first statement follows immediately from 
\[ \lim_{R\to 0} \frac{f_R}{\sin R} = \lim_{R\to 0} \frac{2}{1+|z|^2 + (1-|z|^2)\cos R} =1.\]
The second follows from the smooth convergence of metrics and from the scaling relations $|\Sigma_R|_{(\csc^2 R)\,\bar{g}} = (\sin R)^{-k} |\Sigma_R|_{\bar{g}}$, $|\pr\Sigma_R|_{(\csc^2 R)\, \bar{g}} = (\sin R)^{1-k} |\pr\Sigma_R|_{\bar{g}}$. 
\end{proof}

\bibliography{low-genus}

\end{document}